\documentclass[11pt,a4paper]{article}
\usepackage{epstopdf}
\usepackage{enumerate}        
\usepackage{caption}
\usepackage{subfig}  
\usepackage{amsmath}               
\usepackage{amsfonts}              
\usepackage{amsthm}               
\usepackage{color}
\usepackage{empheq}

\usepackage{amssymb,amsbsy,amsmath,amsfonts,amssymb,amscd,amsthm, mathrsfs}
\usepackage{dsfont}

\newcommand{\supp}{\operatorname{supp}}

\newcommand{\NN}{\mathbb{N}}
\newcommand{\RR}{\mathbb{R}}

\newtheorem{theorem}{Theorem}
\newtheorem{proposition}[theorem]{Proposition}
\newtheorem{lemma}[theorem]{Lemma}
\newtheorem{corollary}[theorem]{Corollary}

\theoremstyle{definition}
\newtheorem{definition}[theorem]{Definition}
\newtheorem{example}{Example}
\newtheorem{remark}[theorem]{Remark}

\newcommand{\beqn}{\begin{equation}}
\newcommand{\eeqn}{\end{equation}}
\newcommand{\bear}{\begin{eqnarray}}
\newcommand{\eear}{\end{eqnarray}}
\newcommand{\bean}{\begin{eqnarray*}}
\newcommand{\eean}{\end{eqnarray*}}
\begin{document}
\title{On a Boltzmann equation  for Compton scattering  from non relativistic electrons at low density.}
\maketitle
\begin{center}
E. Cort\'es\footnotemark[1], M. Escobedo\footnotemark[2]
\end{center}

\footnotetext[1]{BCAM - Basque Center for Applied Mathematics. Alameda de Mazarredo 14,
E--48009 Bilbao, Spain.
E-mail~: {\tt e.cortes.coral@gmail.com}
}
\footnotetext[2]{Departamento de Matem\'aticas, Universidad del
Pa{\'\i}s Vasco, Apartado 644, E--48080 Bilbao, Spain.
E-mail~: {\tt miguel.escobedo@ehu.es}}

\abstract{A Boltzmann equation, used to describe the evolution of the density function of a gas of photons interacting by  Compton scattering  with  electrons at low density and non relativistic equilibrium, is considered. A truncation of the very singular redistribution function is introduced and justified. The existence of weak solutions is proved for a large set of initial data. A simplified equation, where only the quadratic terms are kept and  that appears at very low temperature of the electron gas, for small values of the photon's energies,  is also studied. The existence of weak solutions, and also of more regular solutions that are very flat near the origin, is proved. The long time asymptotic behavior of weak solutions of the simplified equation is  described.}
\tableofcontents

\section{Introduction.}

\setcounter{equation}{0}
\setcounter{theorem}{0}

When only Compton scattering events are considered,  the evolution of the particle density of a gas of photons that interact with electrons at non relativistic equilibrium is usually described by means of a Boltzmann equation that  may be found  in \cite{1964PhFl....7..735D}, \cite{1957Komp}, \cite{1965Weyman} and  many others.  For  a spatially  homogeneous  isotropic gas of photons and  non relativistic electrons at equilibrium, the equation is simplified to the following expression,
with a  notation more usual  in  the literature of  physics:
\begin{align}
&k^2{\partial f\over \partial t}(t, k)=Q_{\beta}(f, f)(t, k),\qquad t>0,\;k\geq 0, \label{S1E0} \\
&Q_{\beta}(f, f)(t, k)=\int_0^\infty \left(\frac {} {} f(t, k')
\,(1+f(t, k))  e^{-\beta k} - \right.\nonumber \\
&\hskip 2cm \left. -f(t, k) (1 + f(t, k')) e^{-\beta k'} \right)k k' \mathcal B _{ \beta  }(k, k')dk'. \label{S1E1}
\end{align}
The variable  $k=|\textbf k|$ denotes  the energy of a photon of momentum 
$\textbf k\in\RR^3$ (taking the speed of light $c$ equal to one), $\beta =(\hbar T)^{-1}$, with $T$ the temperature of the gas of electrons, $(4\pi /3)k^2 f(t,  k)\geq 0$ is the  particle density, and $\mathcal B _{ \beta  }(k, k')$ is a function called sometimes the redistribution function. 

We emphasize that only elastic collisions of one photon and one electron giving  rise to one photon and one electron are considered in this equation, and no radiation effects are taken into account.  As shown in \cite{F},  the cross section for
emission of an additional photon of energy $k $  diverges as  $k $  approaches zero, and so  the probability of a Compton process  unaccompanied  by such emission  is zero. It follows that the equation (\ref{S1E0}), (\ref{S1E1}) can not take accurately  into account  photons with too small energy.

When the speed of light $c$ is taken into account, the corresponding  equation (\ref{S1E0}), (\ref{S1E1}) is very often approximated  by  a  nonlinear Fokker Planck equation (cf. \cite{1957Komp}).  For $\beta >>(m c^2)^{-1}$ (that corresponds to non relativistic electrons with  mass $m$), the scattering cross section of photons with energies $k << mc^2$ may be approximated by the Thompson scattering cross section. It is then possible to deduce the following expression of $\mathcal B_\beta (k, k')$:
\begin{align}
\mathcal B _{ \beta  }(k, k')&=\sqrt \beta \;e^{\beta \frac {(k'+k )} {2}}\int _0^\pi 
\frac {(1+\cos^2\theta)} { |\textbf{k}'-\textbf{k}| }e^{-\beta \frac {\Delta ^2+\frac {m^2v^4} {4}} {2mv^2}} d\cos \theta,\label{S2E009}\\
v&=\frac {1} {m }|\textbf{k}'-\textbf{k}|,\quad \Delta =k'-k,  \label{S2E007}
\end{align}
(cf. \cite{Mes} and \cite{EMV}). It is then argued (cf. \cite{1957Komp} for example) that $\mathcal B _{ \beta  }(k, k')$ is strongly peaked in the region 
\begin{equation}
\label{S1CK}
\left\{k>0, \;k'>0;\,\,|k-k'|<<\min\{k, k'\}\right\}
\end{equation} 
for large values of $\beta $, (cf. Figure \ref{fig234} in Appendix B.2) and then, if the variations of $f$ are not too large, it is  possible to expand the integrand of (\ref{S1E0}) around $k$ and, after a suitable rescaling of the time variable, the equation (\ref{S1E0}), (\ref{S1E1}) is approximated by:
\bear
\frac{\partial f}{\partial t}=\frac{1}{k^2}\frac{\partial }{\partial k}\left(k^4
\left(\frac{\partial f}{\partial k}+f^2+f \right)\right),\label{S1E2}
\eear
the  Kompaneets equation (\cite{1957Komp}).  However, it is  difficult to determine under what conditions on the initial data and in what range of  photon energies $k$, is this approximation correct.

Due in particular  to its importance in modern cosmology and high energy astrophysics, the  Kompaneets equation has received great attention in the literature of physics (cf. the review \cite{BIRKINSHAW199997}). It has also been studied from a more strictly mathematical point of view (\cite{doi:10.1063/1.865928}, \cite{MR1491861}, \cite{OK}), and several of its possible approximations  have also been considered (\cite{Ballew2016BoseEinsteinCI}, \cite{MR3523075}). 
It was first observed  in \cite{1969JETP...28.1287Z} that for a large class of initial data, as $t$ increases, the solutions of (\ref{S1E2})  may develop steep profiles, very close to a shock wave, near $k=0$.  This was proved to happen in \cite{MR1491861} for some of the solutions,  for  $k$ in a neighborhood of the origin and large times.

On the basis of the equilibrium distributions $F_M$ of (\ref{S1E0}), (\ref{S1E1}) given  by
\begin{align}
&k^2F_M=k^2f_\mu +\alpha \delta _0, \,\,\,\mu \le 0,\,\,\alpha \ge 0,\,\,\,\alpha \mu =0, \label{S1Estat1}\\
&f_\mu(k) =\frac {1} {e^{k-\mu }-1},\,\,\int  _{ 0 }^\infty k^2f_\mu (k)dk=M_\mu,\,\,M=\alpha +M_\mu,\label{S1Estat2}
\end{align}
some of its unsteady  solutions are also expected to develop, asymptotically in time,  very large values and strong variation in very small regions near the origin. This  was proved to be true  in \cite{ESMI} where, under the assumptions that   $e^{-\eta(k+k')}(kk')^{-1}\mathcal B _{ \beta  }(k, k')$ is a bounded  function on $[0, \infty)^2$ for some $\eta\in [0, 1)$,  it is shown that, as $t\to \infty$, certain solutions  form a Dirac mass at the origin. A detailed description of this formation was  given later   in  \cite{Escobedo2004208},  assuming  $\mathcal B _{ \beta  }(k, k')(kk')^{-1}\equiv 1$  and for some  classes  of  initial data. Of course, in the region where this delta formation takes place, the equation (\ref{S1E0}), (\ref{S1E1}) can not be approximated by the Kompaneets equation (\ref{S1E2}).  

It is obvious however that the function  $\mathcal B _{ \beta  }(k, k')$  in (\ref{S2E009}), (\ref{S2E007}) does not satisfies the  conditions  imposed in \cite{ESMI} or \cite{Escobedo2004208}.  On the other hand, the Boltzmann equation  (\ref{S1E0}), (\ref{S1E1})  with the kernel    (\ref{S2E009}), (\ref{S2E007}) was considered  in \cite{Nouri1} and \cite{Nouri2}. Local existence for small initial data with a  moment of order $-1$  was proved in \cite{Nouri1}. It was proved in \cite{Nouri2} that,  although globally solvable  in time for  initial data bounded from above by the Planck distribution, the Cauchy problem has no solution, even local in time, for initial data  greater that the Planck distribution. This seems to be an effect  of the very small values of $k$ and $k'$ with respect to $|k-k'|$ in the collision integral, and   indicates that some  truncation is needed in order to have a reasonable theory for the Cauchy problem.  (cf.  Section \ref{truncation} below). 
 
In this article we  consider first the Cauchy problem for an equation where  the kernel (\ref{S2E009}), (\ref{S2E007}) is  truncated in a region where $k$ or $k'$ are much smaller than $|k-k'|$, although  the strong singularity at the origin $k=k'=0$ is kept. This is achieved by multiplying the kernel $\mathcal B _{ \beta  }$ by a suitable cut off function $\Phi(k, k')$,
\begin{align}
&k^2{\partial f\over \partial t}(t, k)=\widetilde {Q}_{\beta}(f, f)(t, k) \label{S1E0B} \\
&\widetilde{Q}_{\beta}(f, f)(t, k)=\int_0^\infty \left(\frac {} {} f(t, k')
\,(1+f(t, k))  e^{-\beta k} - \right.\nonumber \\
&\hskip 1cm \left. -f(t, k) (1 + f(t, k')) e^{-\beta k'} \right)k k'  \Phi (k, k')\mathcal B _{ \beta  }(k, k')dk' \label{S1E1B}
\end{align}
The Cauchy problem for (\ref{S1E0B}), (\ref{S1E1B})  proved to have weak solutions for a large class of initial data in the space of non negative measures. Because of some difficulties   coming from the kernel $\mathcal B_\beta $ and its truncation, it  is not possible to perform   the same analysis as in \cite{ESMI} or  \cite{Escobedo2004208}, where the asymptotic behavior of the solutions was described. 

Some further insight may be obtained from a simplified equation, proposed  in 
\cite{1972ZhETF..62.1392Z} and  \cite{1972JETP...35...81Z}, where the authors suggest  to keep only the quadratic terms in (\ref{S1E1}) when $f>>1$ (or when the function $f$ has a large derivative) and consider,
\begin{equation}
k^2{\partial f\over \partial t} (t, k)
=f(t, k) \int_0^\infty f(t, k') \bigl( e^{-\beta k} - e^{-\beta k'} \bigr) kk'\mathcal B _{ \beta  }(k, k')dk' .\label{S1EwLp}
\end{equation}
This equation may be  formally obtained in the limit  $\beta \to \infty$ and $\beta k$ of order one (cf. Section \ref{deduction} below).  If the reasoning leading from equation (\ref{S1E0}) to the Kompaneets equation goes for the equation (\ref{S1EwLp}), the following non linear first order equation is obtained,
\bear
\frac{\partial f}{\partial t}=\frac{1}{k^2}\frac{\partial }{\partial k}\left(k^4 f^2\right),\,\,\,t>0, \, k>0. \label{S1E2K3}
\eear
For the same reasons as for the  equation (\ref{S1E0}), we  consider the equation (\ref{S1EwLp}) with the truncated redistribution function,
\begin{equation}
k^2{\partial f\over \partial t} (t, k)
=f(t, k) \int_0^\infty \!\!\!f(t, k') \bigl( e^{-\beta k} - e^{-\beta k'} \bigr) kk'\Phi (k, k')\mathcal B _{ \beta  }(k, k')dk' .\label{S1Ew}
\end{equation}

As for the equation  (\ref{S1E0B}), (\ref{S1E1B}), equation (\ref{S1Ew})  has weak solutions for a large set of initial data. Moreover, if the initial data is an integrable function, sufficiently flat around the origin, it has a global solution, that remains, for all time,  an integrable function, flat around the origin.  The weak solutions  of (\ref{S1Ew}) converge as $t$ tends to infinity to a limit  that may be almost completely characterized. It is formed by an at most countable number of Dirac masses, whose locations are determined by the way in which the mass of the initial data is distributed. This  suggests a possible transient behavior for the solutions of the complete equation (\ref{S1E0B}), where large and concentrated peaks could form and remain for some time.

We refer to  \cite{Grachev} for  recent numerical simulations on the behavior of the solutions of the equation (\ref{S1E0}) and the Kompaneets approximation. The anisotropic case has also been  recently considered in \cite{Buet}.

We describe now  our results in more detail.

\subsection{The function $\mathcal B _{ \beta  }(k, k')$. Weak formulation.}
Due to the $k^2$ factor in the left hand side of (\ref{S1E0}), it is natural to introduce the new variable
\begin{equation}
\label{S1ENU}
v(t, k)=k^2f(t, k).
\end{equation}
This  variable $v$ is now, up to a constant, the photon density in the radial variables, and  equation (\ref{S1E0}), (\ref{S1E1}) reads,
\begin{align}
&\frac {\partial v} {\partial t}(t, k) =  \mathcal{Q}_{\beta}(v, v)(t,k),\qquad t>0,\;k\geq 0, \label{S1EN0} \\
&\mathcal{Q}_{\beta}(v, v)(t,k)=\int_0^\infty q_\beta (v, v') \frac {\mathcal B _{ \beta  }(k, k')} {kk'}dk', \label{S1EN1}\\
&q_\beta (v, v')=v' (k^2+v)  e^{-\beta k} - v (k'^2 + v') e^{-\beta k'}, \label{S1Z203}
\end{align}
where we use the common notation $v=v(t,k)$ and $v'=v(t,k')$.
As a consequence of the change of variables (\ref{S1ENU}), the factor $kk'$ in the collision integral has been changed to $(kk')^{-1}$. 

An expression  of  $\mathcal B _{ \beta  }(k, k')$ may be obtained at low density of electrons and using the non relativistic approximation of the Compton scattering cross section (cf. \cite{Mes, EMV}). It may  be seen in particular that  $\mathcal B _{ \beta  }(k, 0)>0$ for all $k>0$, and
\begin{align}
\mathcal B _{ \beta  }(k, k')&=\frac {44} {15 }\left(\frac {2} {k+k '}+1\right)+\mathcal O(k+k'),\quad k+k'\to 0, \label{S1ER1'}\\
\mathcal B _{ \beta  }(k, 0)&= \frac {8 \sqrt \beta} {3}e^{\frac {k} {2}}\frac {e^{- \frac {\beta ^2+k^2} {2\beta m}}} {k}.\label{S1ER2}
\end{align}
The kernel  $\mathcal B _{ \beta  }(k, k')(kk')^{-1}$ is then rather singular near the axes,
and  the collision integral $\mathcal Q(v,v)$ is not  defined for $v(t)$ a general non negative bounded measure. In order to overcome this problem, it is usual to introduce weak solutions. A natural definition of weak solution is:
\begin{align}
\label{S1Z202}
\frac {d} {dt} \!\int\limits  _{ [0, \infty) }\!\!v(t, k)\varphi (k)dk=\frac {1} {2}\iint\limits _{[0, \infty)^2 } \big(\varphi-\varphi'\big) q_\beta (v, v')\frac {\mathcal B _{ \beta  }(k, k')}{kk'}dk dk'
\end{align}
for a suitable space of test functions $\varphi $. Again, we use the notation $\varphi=\varphi(k)$ and $\varphi'=\varphi(k')$. Since $\mathcal B _{ \beta  }(k, 0)>0$ for all $k>0$, the integral in the right hand side of (\ref{S1Z202})  may still  diverge. It  was actually   proved in \cite{Nouri2} that for initial data $v_0$ such that
$$
v_0(x)> \frac {x^2} {e^x-1},\quad\forall x\ge 0,
$$ 
the  (\ref{S1Z202}) has no solution in $C([0, T), \mathscr{M}_+^1([0,\infty)))$, for any $T>0$. 

Kernels with that  kind of singularities have been considered in coagulation equations. One possible way to overcome this difficulty  and obtain global solutions is to impose   test functions $\varphi $  compactly supported on $(0, \infty)$, like in \cite{Norris}, or such that $\varphi (x)\sim x^\alpha $ as $x\to 0$ for some $\alpha $ large enough, like for example in \cite{Fournier}, (but in that case we could not expect to obtain any information on what happens near the origin), or also to look for solutions $v$ in suitable weighted spaces like in \cite{Kumar2} and \cite{Camejo} (but that would exclude the Dirac delta at the origin). In all these cases, the propagation  of negative moments for all  $t>0$ is necessary. That property does not seem to hold true for (\ref{S1E0B}), cf. Remark \ref{S2EXIP1} for the local propagation of some negative moments.  See Remark \ref{S1RXP} and Remark \ref{S5RXBT2} for the equation (\ref{S1Ew}). 

\subsubsection{Truncated  kernel: why and how.}
\label{truncation}

As we have already mentioned, the equation (\ref{S1E0}), (\ref{S1E1}) does not describe the Compton scattering if ``too'' low energy photons  are considered, since in that case the spontaneous emission of photons must be taken into account (cf. \cite{F}). At this level of description then, some cut off  seems   necessary for  a coherent  description, where only collisions of one photon and one electron giving one photon and one electron are considered.

In view of the properties of the function $\mathcal B_\beta $ for $\beta $ large presented in Appendix \ref{deduction}, and since no precise indication is available in the literature of physics, we  use  a mathematical criteria as follows:  

(i) - We  truncate the kernel $\mathcal B_\beta $, down  to zero, out of the  following subset  of  $[0, \infty)\times [0, \infty)$:
\begin{align}
&\forall (k,k')\in [0,\delta_*]^2,\quad |k-k'|\le \rho_*(kk')^{\alpha_1} (k+k')^{\alpha_2},  \label{S1EN100}\\
&\forall (k, k')\in [0,\infty)^2\setminus[0,\delta_*]^2,\quad\theta k\le k'\le \theta ^{-1} k, \label{S1EN100B}
\end{align}
for some constants $\delta_*>0$, $\rho_*>0$, $\alpha_1\ge 1/2$, $2\alpha_2\ge 3-4\alpha$, and $\theta \in (0, 1)$.

(ii) - In order to  minimize the region of this truncation, we choose $\alpha_1=\alpha_2 =1/2$.

(iii) - We  leave $\mathcal B_\beta $ unchanged as much as possible inside that region, but at the same time we want the resulting truncated kernel to belong to $C((0, \infty)\times (0, \infty))$.

\begin{remark}
It is suggested in \cite{1975SvPhU..18...79Z} that for very large values of $\beta $, the support of $\mathcal B_\beta $ is a 
subdomain of $|k-k'|\!<\!2k^2/mc^2$ for small values of $k$ and $k'$. That would be a stronger truncation than  in (ii).
\end{remark}
\noindent
Then we multiply  $\mathcal B _{ \beta  }(k, k')$ by $\Phi (k, k')$, where:\\

\noindent 1. $\Phi (k,k')=\Phi(k',k)$ for all $k>0$, $k'>0$,\\
2. $\Phi \in C([0,\infty)^2\setminus\{(0,0\})$,\\
3. $\supp (\Phi)=D$, where $(k,k')\in D$ if and only if (\ref{S1EN100}) and (\ref{S1EN100B}) hold for $\alpha_1=\alpha_2=1/2$, and some constants $\theta\in(0,1)$, $\delta_*>0$ and $\rho_*=\rho_*(\theta,\delta_*)$.\\
4. $\Phi(k,k')=1\;\forall(k,k')\in D_1\subset D$, where $(k,k')\in D_1$ if and only if,
\begin{equation*}
\begin{array}{lll}
|k-k'|\le \rho_1\sqrt{kk'(k+k')}&\hbox{if}&(k,k')\in [0,\delta_*]^2,\\
\theta_1 k\le k'\le \theta_1 ^{-1} k&\hbox{if}&(k, k')\in [0,\infty)^2\setminus[0,\delta_*]^2,
\end{array}
\end{equation*}
for some $\theta_1\in (\theta,1)$ and $\rho _1=\rho_1(\theta_1,\delta_*)>0$. (Cf. also (\ref{B1})--(\ref{Bbound})).

Then, for all $\varphi \in C^1([0, \infty))$, 
\begin{equation*}
\big(e^{-\beta k}-e^{-\beta k'}\big)\big(\varphi(k) -\varphi(k')\big)\frac {\mathcal B _{ \beta  }(k, k')} {kk'}\Phi (k , k' )\in L^\infty _{ loc }([0, \infty)\times [0, \infty)),
\end{equation*}
and if $\varphi '(0)=0$,
\begin{equation*}
\big(e^{-\beta k}-e^{-\beta k'}\big)\big(\varphi(k) -\varphi(k')\big)\frac {\mathcal B _{ \beta  }(k, k')} {kk'}\Phi (k , k' )\in C([0, \infty)\times [0, \infty))
\end{equation*}
(cf. Lemma \ref{klbounds}, Lemma \ref{k continuous} and (\ref{k})).\\

In the  first  part of this work, we then consider the problem
\begin{align}
\label{S1ZTr}
\frac{\partial v}{\partial t}(t,k)=\int_{[0,\infty)}q_{\beta}(v,v')\frac{\mathcal{B}_{\beta}(k,k')\Phi(k,k')}{kk'}dk'.
\end{align}
We need the following notations:

$C^1_b([0,\infty))$ is the space of bounded continuous functions, with continuous bounded derivative, on $[0,\infty)$. 

The space of nonnegative bounded Radon measures is denoted  $\mathscr{M}_+([0,\infty))$, and 
\begin{align}
&\mathscr{M}^{\rho }_+([0,\infty))=\{v\in\mathscr{M}_+([0,\infty)): M_{\rho }(v)<\infty\},\,\,\,\forall \rho  \in \RR,\nonumber\\
&M_{\rho }(v)=\int_{[0,\infty)}k^{\rho }v(k)dk\quad(\text{moment of order $\rho $}),\label{definition moment}\\
\label{definition exponential moment}
&X_{\rho}(v)=\int_{[0,\infty)}e^{\rho k}v(k)dk.
\end{align}
We use the notation $\int v(k)dk$ instead of $\int dv(k)$, even if the measure $v$ is not absolutely continuous with respect to the Lebesgue measure. 

Unless stated otherwise, the space $\mathscr{M}_+([0,\infty))$ is considered with the narrow topology. We recall that the narrow topology is generated by the metric
$d_0(\mu,\nu)=\|\mu-\nu\|_0$, where (cf. \cite{BOG}, Theorem 8.3.2),
\begin{gather}
\label{mu0}
\|\mu\|_0=\sup\bigg\{\int_{[0,\infty)}\varphi d\mu:\varphi\in\text{Lip}_1([0,\infty)),\;\|\varphi\|_{\infty}\leq 1\bigg\},\\
\label{Lip1}
 \text{Lip}_1([0,\infty))=\{\varphi:[0,\infty)\to\RR:|\varphi(x)-\varphi(y)|\leq|x-y|\}. 
\end{gather}
The following is an existence result for the problem (\ref{S1ZTr}).

\begin{theorem}
\label{MT1}
Given any $v_0\in\mathscr{M}_+([0,\infty))$ satisfying 
\begin{equation}
\label{eu0}
X_{\eta}(v)<\infty,
\end{equation}
for some $\eta\in\left(\frac{1-\theta}{2},\frac{1}{2}\right)$, then there exists  $v\in C([0,\infty),\mathscr{M}_+([0,\infty)))$  weak solution of (\ref{S1ZTr}), i.e., such that satisfies the following 
(i)-(ii):
\begin{enumerate}[(i)]
\itemsep=0em
\item 
For all $\varphi\in C_b([0,\infty))$,
\begin{gather}
\int_{ [0, \infty) }v(\cdot , k)\varphi (k)dk \in C( [0, \infty);\RR),\\
\int_{ [0, \infty) }v(0 , k)\varphi (k)dk=\int_{ [0, \infty) }v_0(k)\varphi (k)dk,
\end{gather}
\item 
For all $\varphi\in C^1_b([0,\infty))$ with $\varphi'(0)=0$,
\begin{gather}
\int_{ [0, \infty) }v(\cdot , k)\varphi (k)dk \in W^{1, \infty}_{loc} ([0, \infty); \RR),\\
\intertext{and for almost every $t>0$,}
\label{S1MT105}
\!\!\!\!\!\!\frac {d} {dt}\int\limits_{ [0, \infty) }v(t , k)\varphi (k)dk=\frac{1}{2}\iint\limits_{[0, \infty)^2}\frac{\Phi\mathcal B_\beta}{kk'} q_{\beta}(v,v')(\varphi-\varphi')dkdk'.
\end{gather}
\end{enumerate}
The measure $v(t)$ also satisfies, for all $t\geq 0$,
\begin{gather}
M_0(v(t))=M_0(v_0) \label{mass}\\
X_{\eta}(v(t)) \leq e^{C_{\eta}t}X_{\eta}(v_0), \label{expmoment}\\
\shortintertext{where}
\label{Ceta}
C_{\eta}=\frac{C_*}{2\theta^2}\frac{(1-\theta)}{(1+\theta)}\frac{\eta}{\left(\frac{1}{2}-\eta\right)},\quad C_*>0.
\end{gather}
\end{theorem}

\begin{remark}
Theorem \ref{MT1} does not precludes the formation, in finite time, of  a Dirac measure at the origin in  the weak solutions of (\ref{S1ZTr}) with  integrable initial data. Such a possibility  was actually   considered for  the solutions of the Kompaneets
equation (cf. \cite{1975SvPhU..18...79Z, 1969JETP...28.1287Z, 1972JETP...35...81Z} and others). It was proved in  \cite{MR1491861} and \cite{ESMI} that, for  large sets of initial data, this does not happen, neither in the  Kompaneets equation, nor in equation (\ref{S1EN0}) with a very simplified kernel. But it is not known  yet if it may happen  for the equation (\ref{S1EN0}) with the  kernel $\Phi (k, k')\mathcal B_\beta (k, k')$.
\end{remark}

Given  a weak solution  $\{u(t)\} _{ t>0 }$ of (\ref{S1ZTr}) whose Lebesgue decomposition is $u(t)=g(t)+G(t)$, with $g(t)\in L^1([0, \infty))$, the natural physical entropy is
\begin{align}
&H(u(t))=\int  _{ (0, \infty) }h(x, g(t, x))dx-\int  _{ (0, \infty) } xG(t, x)dx, \label{S2EH}\\
&h(x, s)=(x^2+s)\log (x^2+s)-s\log s-x^2\log x^2-sx.\label{S2Eh}
\end{align}
It was proved in \cite{ESMI} that the maximum of $H$ over all the non negative measures of total mass $M$ is achieved at, and only at, $U_M=k^2F_M$ given in (\ref{S1Estat1}). 
But  the corresponding  dissipation of entropy used  in  \cite{ESMI}, is not defined here  due to the singularity of the kernel  $\frac{\Phi\mathcal B_\beta}{kk'} $ at the origin. The study of the long time behavior of the weak solutions obtained in Theorem \ref{MT1}  seems then to be more involved than in \cite{ESMI}, (cf.  also Section \ref{entropy}).

\subsection{The simplified equation } 
In view of the exponential terms in (\ref{S1E1}), it is very natural to consider the scaled variable $\beta k=x$, then scale the time variable too as $\beta ^3 t=\tau$,  and the dependent variable  as $\beta^2 k^2 f(t, k)=u(\tau, x)$ in order to let the total number of particles  to be unchanged (cf. Section \ref{betascaling}). When this is done, it  appears that the linear term is  formally of  lower order in $\beta >>1$:

\begin{align}
\frac {\partial u} {\partial \tau }(\tau,x)&=\int _0^\infty \frac {\widetilde B_{\beta}(x,y)} {xy}\big(e^{-x}-e^{-y}\big)u(\tau,x)u(\tau,y)dy+\nonumber\\
&+\beta ^{-3}\int _0^\infty  \frac {\widetilde B_{\beta}(x,y)} {xy}\big(u(\tau,y)x^2e^{-x}-u(\tau,x)y^2e^{-y}\big)dy, \label{EscEuIntro}
\end{align}
where $\widetilde B_{\beta}(x, y)$ is the suitably scaled version of $ \mathcal{B}_{\beta}(k, k')$.

 If only the quadratic term is kept in (\ref{S1ZTr}), the following equation follows:
\begin{equation}
\label{S1ER1}
\frac {\partial v} {\partial t}(t,k)=v(t, k)\!\!\!\!\int\limits_{[0, \infty)} \!\!\!\! v(t, k')
\big(e^{- \beta k }-e^{-\beta  k'}\big)\frac{\mathcal B_\beta (k, k')\Phi (k, k')}{kk'}dk'.
\end{equation}

Weak solutions $u\in C([0,\infty),\mathscr{M}_+([0,\infty)))$ to (\ref{S1ER1}) for all initial data $u_0\in\mathscr{M}_+([0,\infty))$ satisfying (\ref{eu0}) are proved to exist (cf. Theorem \ref{S8Th1}) with similar arguments as for the complete equation.

But  equation (\ref{S1ER1}) also  has solutions $v\in C([0, \infty), L^1([0, \infty)))$ for initial data $v_0\in L^1([0, \infty))$ that are sufficiently flat around the origin. This ``flatness''  condition happens then to be sufficient to prevent the finite time formation of a Dirac measure at the origin in the solutions of (\ref{S1ER1}).
\begin{theorem}
\label{MT3t}
For any nonnegative initial data $v_0\in L^1([0, \infty))$   such that:
\begin{align}
\label{flat}
\forall r>0,\quad \int _0^\infty v_0(k)\left(e^{\frac {r} {k^{3/2}}}+e^{\eta k}\right) dk<\infty,
\end{align}
for some $\eta>(1-\theta)/2$,  
there exists  a  nonnegative  global weak solution $v\in C([0, \infty), L^1([0, \infty)))$ of (\ref{S1ER1}) that also satisfies
\begin{align}
\label{FOBIS}
v(t, k)=v_0(k)e^{\int_0^{t}\int_0^\infty \big(e^{- \beta k }-e^{-\beta  k'}\big)\frac{\mathcal B_\beta (k, k')\Phi (k, k')}{kk'}\, v(s, k')\, dk'\, ds},
\end{align} 
for all $t>0$, and $ a. e.\, k>0$. Moreover, for all $t>0$,
 \begin{align}
&\|v(t)\|_1 =  \|v_0\|_1, \label{S2Ec602}\\
&v(t, k)\le v_0(k)e^{\frac {tC_0} {k^{3/2}}},\,\forall t>0, \,\, and\,\,\,\,a. e.\, k>0,\label{S2Ec6021} 
\end{align}
where $C_0=\frac{\rho_* C_*}{\sqrt{\theta(1+\theta)}}X_{\eta}(v_0)$.
\end{theorem}

\begin{remark}
\label{S1RXP}
It follows from (\ref{S2Ec6021}) that for the solution $v$ obtained in  Theorem \ref{MT3t}, $v(t)$ satisfies (\ref{flat}) for almost every $t>0$. That property is then propagated globally in time.  
\end{remark}

For a solution $v$ to equation (\ref{S1ER1}), the moment $M_{\rho}(v(t))$ defined in (\ref{definition moment}) is proved to be a Lyapunov function on $[0,\infty)$ for all $\rho\geq 1$ (cf. Lemma \ref{LmomZ}). With some abuse of language, we sometimes refer to $M_{\rho}(v)$ as  an entropy functional for 
equation (\ref{S1ER1}).
It is possible to characterize the  nonnegative measures that minimize $M_{\rho}(v)$ for $\rho>1$, 
or satisfy $D_\rho (v)=0$, where
\begin{align}
\!\!\!D_\rho (v)=&\iint_{ (0, \infty)^2 }\frac {\Phi\mathcal B_\beta} {kk'}
 \big(e^{-\beta k}-e^{-\beta k'} \big)(k^\rho -k'^\rho)v(k)v(k')dkdk' \label{S1ERDH}
\end{align}
is the corresponding entropy dissipation functional.  
This question is  solved, usually, at mass $M$ and energy $E$ fixed. Because of  the truncated kernel $\Phi \mathcal B_\beta $, it is also necessary to introduce  the following  property about  the support of the measure $v$ in connection  with the support of the kernel $\Phi \mathcal B_\beta $.

Given a measure $v\in \mathscr{M}_+([0,\infty))$, we denote $\{A_n(v)\} _{ n\in \NN }$ the, at most, countable collection of disjoint closed  subsets of the support of $v$ such that, 
\begin{equation}
\begin{aligned}
\label{S1PQ}
&\hskip -0.5cm \small (k, k')\in A_n\!\times \!A_n\hbox{ for some} \,n\in \NN ,\hbox{ if and only if},\,\Phi \mathcal (k,  k')\not =0\,\,\,\hbox{or} \\
&\hskip -0.5cm \exists \{k_n\} _{ n\in \NN }\subset [k, k'); \,k_1=k, \lim _{ n\to \infty } k_n=k',\Phi \mathcal (k_n,  k _{ n+1 })\not =0\;\forall n\in\NN.
\end{aligned}
\end{equation}
(cf. Section \ref{entropy} for a   precise definition of $A_n(v)$). 

Let us define now, for any countable collection $\mathcal C=\{C_n, M_n\} _{ n\in \NN }$  of disjoint, closed subsets $C_n\subset [0, \infty)$ enjoying the property (\ref{S1PQ}),  and positive real  numbers $M_n$, the following family of non negative measures,
\begin{align*}
\mathcal{F} _{ \mathcal C,\alpha  }=\bigg\{ v\in \mathscr M^\alpha _+([0,\infty)):C_n=A_n(v),\; M_n=\int  _{ A_n(v) }v(k)dk \bigg\}.
\end{align*}
\begin{theorem}
\label{characterization}
For any $\mathcal C$ and $\alpha >1$ as above, the following statements are equivalent:
\begin{enumerate}[(i)]
\item $v\in \mathcal{F} _{ \mathcal C,\alpha  }$ and $D_{\alpha}(v)=0$.
\item $M_{\alpha}(v)=\min\{M_{\alpha}(v):v\in\mathcal{F} _{ \mathcal C,\alpha  }\}$.
\item $v=\sum_{n=0}^{\infty} M_n\delta_{k_n}$, where $k_n=\min \{k\in A_n\}$.
\end{enumerate}
\end{theorem}

\begin{remark}
\label{rem10}
For any sequence $\{x_n\} _{ n\in \NN }$ such that $x_n>0$, $x_n\to 0$ as $n\to\infty$, and $\Phi (x_n, x _m)=0$ for all $n\not =m$,  the measure
\begin{equation*}
u=\sum _{ n=0 }^\infty \alpha _n\delta_{x_n}
\end{equation*} 
satisfies the conditions (i)--(ii) in Theorem \ref{characterization}. Although $0\in \supp u$, there is no Dirac measure at the origin.

\end{remark}

The long time behavior of the weak solutions of  (\ref{S1ER1}) is described in  the following Theorem,

\begin{theorem}
\label{S7E001}
Let $v$ be a weak solution of (\ref{S1ER1}) constructed in Theorem \ref{S8Th1} for an initial data $v_0\in\mathscr{M}_+([0,\infty))$ satisfying $X_{\eta}(v_0)<\infty$ for some $\eta\geq (1-\theta)/2$.

Then, as $t\to\infty$, 
 $v(t)$ converges   in $C([0,\infty),\mathscr{M}_+([0,\infty)))$ to the measure 
 \begin{align}
\label{stat}
\mu =\sum_{i=0}^{\infty}M_i'\delta_{k_i'},
\end{align}
where $M_i'\geq 0$, $k_i'\geq0$ satisfy the following properties:  
\begin{enumerate}[1.]
\item $k_i'\in\supp (v_0)$ for all $i\in\NN$,
\item $\Phi(k_i',k_j')=0$ for all $i\neq j$,
\item
If we define $\mathcal{J}_n\!=\!\{i\in\NN:k_i'\!\in\! A_n(v_0),\,M_i'>0\}$ for all $n\!\in\!\NN$,
\begin{align}
\label{sum masses}
\sum_{i\in\mathcal{J}_n}M_i'=M_n,
\end{align}
\item
For all $n\in\NN$, if $k_n=\min\{k\in A_n(v_0)\}>0$, then there exists $k_i'$ such that $k_i'=k_n$.
\end{enumerate}
 \end{theorem}
\begin{remark}
If in Point 4 of  Theorem \ref{S7E001}, $k_n=\min\{k\in A_n(v_0)\}=0$ for some $n\in\NN$,  but $v_0$ has no Dirac measure at $k=0$, we do not know if  $k'_i=0$ for some $i\in \NN$, even if  the origin belongs to the support of the limit measure $\mu $ (cf. Remark \ref{rem10} for example).
\end{remark}

The  measure  $\mu $ is  of course determined by the initial data $v_0$, but its complete description  (i.e. the values of $k'_i$ and $M'_i$) is not known, only  the locations $k'_i$  of some of the  Dirac masses. For example, it is possible to have $k' _{i }= k_n=\min\{k\in A_n(v_0)\}$ and $k' _{i }< k' _{j }\in A_n(v_0)$ for some $n$, $i$, $j$ in $\NN$, and
$k'_{ i }$, $k' _{ j}$ not seeing each other, i.e., $\Phi(k'_{i},k'_{j})=0$ (cf. Example \ref{S5Eex1}, Section \ref{behavior} ). The location of a Dirac measure at $k' _{i}=x_n$ is just given by the support of the initial data, but the appearance of a Dirac measure at  $k' _{ j}$ is  more difficult to be determined. 

The long time behaviour that is proved in Theorem \ref{S7E001} for the solutions of the simplified equation (\ref{S1ER1}) can not be expected of course to hold for the solutions of the complete equation (\ref{S1ZTr}). But in combination with the  equation (\ref{EscEuIntro}) for $\beta$ large,  it could indicate that the solutions of the complete problem (\ref{S1ZTr}) also undergo  the formation of large an concentrated peaks, that could remain for some long, although finite,  time.
\subsection{General comment}
The main results of the article are stated in this Introduction in terms of the original variables,  $t$, $k$, and $v(t, k)=k^2f(t, k)$. However, in order to make clearly appear some important aspects of the equation,  it is  useful to introduce $\tau$, $x$, and $u(\tau , x)$, variables scaled with the  parameter $\beta $. This is  a natural parameter since it is related with the inverse of the temperature of the gas of electrons.. This scaling makes clearly appear two features  of the equation for $\beta >>1$, namely, the fact that $\mathcal B_\beta $ is very much peaked along the diagonal, and the different scaling properties  of the quadratic and linear part of the collision integral in (\ref{S1EN0}) (cf. Section \ref{betascaling} for details).

However, since in all this work the value of the parameter $\beta $ remains fixed, it is taken equal to one, without any loss of generality.  Therefore, except in Section \ref{deduction}, we have $\tau \equiv t$, $x\equiv k$ and $u\equiv v$. In particular, for the sake of brevity, we do not re-write again the main results in terms of the variables $x$ and $u$, although the proofs will be written in those terms. 

The main results are actually proved for general kernels $B$ satisfying some of the properties that  the truncated kernel $\Phi (k, k') \mathcal B_\beta (k, k')$ is proved to enjoy, and that are sufficient for our purpose. 

\section{Existence of weak solutions.}

\setcounter{equation}{0}
\setcounter{theorem}{0}
In this Section we prove existence of weak solutions to the following problem:
\begin{align}
&\frac{\partial u}{\partial t}(t,x)=Q(u,u)=\int_{[0,\infty)}b(x,y)q(u,u)dy, \label{P}\\
&u(0)=u_0\in \mathscr M_+([0, \infty)), \label{P2}
\end{align}
where $t>0$, $x\geq 0$, 
\begin{gather}
\label{q}
\hskip -0.1cm q(u,u)=u(t,y)(x^2+u(t,x))e^{-x}-u(t,x)(y^2+u(t,y))e^{-y},\\
\label{b}
\hskip -0.2cm b(x,y)=\frac{B(x,y)}{xy},
\end{gather}
under the following assumptions on the kernel $B$:
\begin{flalign}
\label{B1}
&\text{(i) $B(x,y)\geq 0$ for all $(x,y)\in[0,\infty)^2$},&\\
\label{B2}
&\text{(ii) $B(x,y)=B(y,x)$ for all $(x,y)\in[0,\infty)^2$,}&\\
&\text{(iii) $B\in C([0,\infty)^2\setminus\{(0,0)\})$,}\label{Bcont}&\\
&\text{(iv) There exist $\theta\in (0,1)$, $\delta_*>0$ and $\rho_*=\rho_*(\theta,\delta_*)>0$ such that}\nonumber&\\
\label{Bsupport}
&\qquad\qquad\qquad\qquad\supp(B)=\Gamma=\Gamma_1\cup\Gamma_2,&\\
\label{Bsupport1}
&\quad\qquad\Gamma_1=\big\{(x,y)\in[0,\infty)^2\setminus[0,\delta_*]^2:\theta x\leq y\leq \theta^{-1}x\big\},&\\
\label{Bsupport2}
&\quad\qquad\Gamma_2=\big\{(x,y)\in[0,\delta_*]^2:|x-y|\leq \rho_*\sqrt{xy(x+y)}\big\}&\\
&\text{(v) There exists a constant $C_*>0$ such that, for all $(x,y)\in\Gamma$,}\nonumber&\\
\label{Bbound}
&\qquad\qquad\quad B(x,y)\leq B\left(\frac{x+y}{2},\frac{x+y}{2}\right)\leq \frac{C_*e^{\frac{x+y}{2}}}{x+y}.&
\end{flalign}

\begin{remark}
\label{cone}
The region $\Gamma$ in (\ref{Bsupport})--(\ref{Bsupport2}) is such that:
\begin{equation}
\label{S3EGamma20}
\Gamma =\left\{(x, y)\in[0,\infty)^2: y\in \big(\gamma _1(x), \gamma _2(x)\big) \right\},
\end{equation}
where
\begin{equation}
\label{g}
\gamma_1(x)=\left\{
\begin{array}{ll}
\frac{2x+\rho_*^2x^2-\rho_* x^{3/2}\sqrt{\rho_*^2x+8}}{2(1-\rho_*^2x)}&\text{if }x\in[0,\delta_*]\\
\theta x&\text{if }x\in(\delta_*,\infty),
\end{array}
\right.
\end{equation}

\begin{equation}
\label{f}
\gamma_2(x)=\left\{
\begin{array}{ll}
\frac{2x+\rho_*^2x^2+\rho_* x^{3/2}\sqrt{\rho_*^2x+8}}{2(1-\rho_*^2x)}&\text{if }x\in[0,\theta\delta_*]\\
\theta^{-1}x&\text{if }x\in(\theta\delta_*,\infty).
\end{array}
\right.
\end{equation}
In particular $\theta x\leq \gamma_1(x)\leq x\leq \gamma_2(x)\leq \theta^{-1}x$ for all $x\geq 0$.The value of $\rho_*=\rho_*(\theta,\delta_*)$ is chosen so that $\gamma_1$ and $\gamma_2$ are continuous. 
\end{remark}

\begin{definition}
\label{def:ws}
We say that  a map $u:[0, \infty)\to\mathscr M_+([0, \infty))$ is a weak solution of  (\ref{P})--(\ref{P2})  if
\begin{align}
\label{ws1}
&(i)\,\, \forall  \varphi \in C_b([0, \infty)),\,\,\int_{ [0, \infty) }u(\cdot , x)\varphi (x)dx \in C( [0, \infty);\RR)\\
&\qquad \hbox{and} \hskip 1.6cm \int_{ [0, \infty) }u(0 , x)\varphi (x)dx=\int_{ [0, \infty) }u_0(x)\varphi (x)dx,\label{wsID} \\
\label{ws2}
&(ii)\,\, \forall \varphi \in C^1_b([0, \infty)),\, \varphi '(0)=0,\\
\label{ws3}
&\quad\int_{ [0, \infty) }u(\cdot , x)\varphi (x)dx \in W^{1,\infty}_{loc}([0, \infty); \RR),\\
\label{ws4}
&\quad\frac {d} {dt}\!\!\int\limits  _{ [0, \infty) }\!\!\!u(t , x)\varphi (x)dx=\frac{1}{2}\!\!\iint\limits_{[0, \infty)^2}\!\!b(x, y)q(u,u)(\varphi(x)\!-\!\varphi(y))dydx.
\end{align}
\end{definition}
The existence of weak solutions for the problem (\ref{P}), (\ref{P2}) was proved in \cite{ESMI} under  conditions on the kernel $b$  not fulfilled in our case. In order to use that result in \cite{ESMI},
we first consider a regularised version of (\ref{P}), with a truncated  function $b_n\in L^\infty ([0, \infty)\times [0, \infty))$.

It is not possible to define
the   dissipation of entropy  for the weak solutions of  (\ref{P}) as in  \cite{ESMI}, for the same reason as for the equation (\ref{S1ZTr}). However, it may be defined  for the solutions $u_n$ of the regularised version of (\ref{P}),  with the truncated kernel $b_n$,
\begin{align}
&D^{(n)}(u_n)=\frac {1} {2}D^{(n)}_1(g_n)+D^{(n)}_2(g_n, G_n)+\frac {1} {2}D^{(n)}_3(G_n),\label{S2ED}\\
&D^{(n)}_1(g_n)=\iint_{ (0, \infty)^2 }b_n(x, y)j\left((x^2+g_n)e^{-x}g_n', (y^2+g_n')e^{-y}g_n \right)dydx,\label{S2ED1}\\
&D^{(n)}_2(g_n, G_n)=\iint_{ (0, \infty)^2 }b_n(x, y)j\left((x^2+g_n)e^{-x}, g_ne^{-y} \right)G_n(y)dydx,\label{S2ED2}\\
&D^{(n)}_3(G_n)=\iint_{ (0, \infty)^2 }b_n(x, y) j\left(e^{-x}, e^{-y} \right)G_n(y)G_n(x)dydx,\label{S2ED3}\\
&j(a, b)=(a-b)(\ln a -\ln b),\,\,\,\forall a>0, b>0, \label{S2Ej}
\end{align}
where $u_n=g_n+G_n$ is the Lebesgue's decomposition of $u_n$.

\subsection{Regularised problem}
For $n\in\NN$, let $\phi_n\in C_c((0,\infty))$ be such that $0\leq \phi_n(x)\leq x^{-1}$ for all $x\geq 0$, $\supp(\phi_n)=[1/(n+1),n+1]$ and $\phi_n(x)=x^{-1}$ for $x\in[1/n,n]$, so that $\lim_{n\to\infty}\phi_n(x)=x^{-1}$. Then we define
\begin{align}
\label{bn}
b_n(x,y)=B(x,y)\phi_n(x)\phi_n(y),
\end{align}
and consider the problem
\begin{align}
&\frac{\partial u_n}{\partial t}(t,x)=Q_n(u_n,u_n)=\int_{[0,\infty)}b_n(x,y)q(u_n,u_n)dy, \label{P100}\\
&u_n(0)=u_0\in \mathscr M_+([0, \infty)). \label{P200}
\end{align}
If we denote
\begin{align}
\label{K}
&K_{\varphi}(u,u)=\frac{1}{2}\iint_{[0,\infty)^2}k_{\varphi}(x,y)u(t,x)u(t,y)dydx,\\
\label{k}
&k_{\varphi}(x,y)=b(x,y)(e^{-x}-e^{-y})(\varphi(x)-\varphi(y)),\\
\label{L}
&L_{\varphi}(u)=\frac{1}{2}\int_{[0,\infty)}\mathcal{L}_{\varphi}(x)u(t,x)dx,\\
\label{ll}
&\mathcal{L}_{\varphi}(x)=\int_0^{\infty}\ell_{\varphi}(x,y)dy,\\
\label{l}
&\ell_{\varphi}(x,y)=b(x,y)y^2e^{-y}(\varphi(x)-\varphi(y)),
\end{align}
then (\ref{ws4}) reads
\begin{align}
\label{WFKL}
\frac{d}{dt}\int_{[0,\infty)}\varphi(x)u(t,x)=K_{\varphi}(u,u)-L_{\varphi}(u),
\end{align}
and the weak formulation of (\ref{P100}) reads
\begin{align}
\label{WFn}
\frac{d}{dt}\int_{[0,\infty)}\varphi(x)u_n(t,x)=K_{\varphi,n}(u_n,u_n)-L_{\varphi,n}(u_n),
\end{align}
where $b$ is replaced by $b_n$ in  the formulas  (\ref{K})--(\ref{l}). Since 
$b_n\in L^{\infty}([0,\infty)^2)$ for all $n\in \NN$,  Theorem  3 in \cite{ESMI}  may be applied (cf. Proposition \ref{theoremESMI}).
For any $u\in\mathscr{M}_+([0,\infty))$, we denote $u=u_r+u_s$ the Lebesgue decomposition of $u$ into an absolutely continuous measure with respect to the Lebesgue measure, $u_r$, and a singular measure, $u_s$.
\begin{remark}
\label{Ksym}
By symmetry and Lemma \ref{k continuous}, for all $\varphi\in C^1_b([0,\infty))$, 
$$K_{\varphi}(u,u)=\int_{[0,\infty)}\int_{[0,x)}k_{\varphi}(x,y)u(t,x)u(t,y)dydx.$$
\end{remark}

\begin{proposition}
\label{theoremESMI}
For any $n\in\NN$ and any initial data $u_0=u_{0,r}+u_{0,s}\in \mathscr{M}_+^1([0,\infty))$, there exists a unique weak  solution $u_n=u_{n,r}+u_{n,s}\in C([0,\infty),\mathscr{M}_+^1([0,\infty)))$ to (\ref{P100}), (\ref{P200}) that satisfies
\begin{align}
\label{CONSER}
&M_0(u_n(t))=M_0(u_0)\qquad\forall t\geq 0,\\
&\supp(u_{n,s}(t))\subset\supp (u_{0,s})\qquad\forall t\geq 0,
\end{align}
and for all $\varphi\in C_c([0,\infty)\times[0,\infty))$,
\begin{align}
\label{WFn'}
&\int_{[0,\infty)}\varphi(t,x)u_n(t,x)dx=\int_{[0,\infty)}\varphi(0,x)u_0(x)dx\\
&+\int _0^t\int _{[0,\infty )}\varphi _t(t, x)u(t, x)dxds
+\int_0^t\int_{[0,\infty)}Q_n(u_n,u_n)\varphi(s,x)dxds,\nonumber
\end{align}
and for all $t_1$ and $t_2$ with $t_2\geq t_1\geq 0$,
\begin{align}
\int  _{ t_1 }^{t_2}D^{(n)}(u_n(t))dt=H(u_n(t_1))-H(u_n(t_2)).
\end{align}
Moreover, if $u_0\in L^1([0,\infty))$ then  $u_n\in C([0,\infty),L^1([0,\infty)))$.
\end{proposition}
\begin{proof}
Theorem  3 in \cite{ESMI}.
\end{proof}
\begin{remark}
In Proposition \ref{theoremESMI}, the space $\mathscr{M}_+^1([0,\infty))$ is endowed with the total variation norm.
\end{remark}

\begin{corollary}
\label{extendedWF}
Let $u_n$ be as in Proposition \ref{theoremESMI} for $n\in\NN$. Then (\ref{WFn}) holds for all $t>0$ and for all nonnegative $\varphi\in C([0,\infty))$ such that
$\int_{[0,\infty)}\varphi(x)u_0(x)dx<\infty$.
\end{corollary}

\begin{proof}
Given a nonnegative function $\varphi\in C([0,\infty))$ such that \\
$\int_{[0,\infty)}\varphi(x)u_0(x)dx<\infty$, let $\{\varphi_k\}_{k\in\NN}\subset C_c([0,\infty))$ be such that 
$\varphi_k(x)\to\varphi(x)$ as $k\to\infty$ for all $x\in[0,\infty)$, and $\varphi_k\leq\varphi_{k+1}\leq\varphi$ for all $k\in\NN$. By (\ref{WFn'}) with test function $\varphi_k$, and recalling that $\phi_n$ is compactly supported, it is easy to deduce using Fubini's theorem, the symmetry of $B$, and the antisymmetry of $q(u_n,u_n)$, that for all $k\in\NN$,
\begin{align}
\label{www}
\int_{[0,\infty)}\varphi_k(x)u_n(t,x)dx&=\int_{[0,\infty)}\varphi_k(x)u_0(x)dx\nonumber\\
&+\int_0^t \big(I_{\varphi_k,n}(u_n,u_n)-L_{\varphi_k,n}(u_n)\big)ds.
\end{align}
Using again that $\phi_n$ is compactly supported, we can pass to the limit as $k\to\infty$ in (\ref{www}) by monotone and dominated convergence theorems to obtain (\ref{www}) with $\varphi$ instead of $\varphi_k$.

Now, since $u_n\in C([0,\infty),\mathscr{M}_+^1([0,\infty)))$, where the topology on $\mathscr{M}_+^1([0,\infty))$ is the total variation norm, it follows that the maps 
$$t\mapsto K_{\varphi,n}(u_n(t),u_n(t)),\quad t\mapsto L_{\varphi,n}(u_n(t))$$
are continuous for all $n\in\NN$ and all $t>0$. Then (\ref{WFn}) follows from (\ref{www}) with $\varphi$ istead of $\varphi_k$, by the fundamental theorem of calculus.
\end{proof}

\subsection{The limit $n\to\infty$}
The goal now is to pass to the limit as $n\to\infty$ in (\ref{WFn}) and obtain a weak solution of (\ref{P})--(\ref{Bbound}). We start with the following uniform estimate.
\begin{proposition}
\label{exponential decay}
Let $u_n$ and $u_0$ be as in Proposition \ref{theoremESMI}. If $X_{\eta}(u_0)<\infty$
for some $\eta\in (0,1/2)$, then for all $t>0$ and all $n\in\NN$,
\begin{equation}
\label{exp}
X_{\eta}(u_n(t))\leq e^{C_{\eta}t}X_{\eta}(u_0)\\
\end{equation}
where $C_{\eta}$ is defined in (\ref{Ceta}).
\end{proposition}

\begin{proof}
Let $\eta\in(0,1/2)$ and take $\varphi(x)=e^{\eta x}$ in (\ref{WFn}), which is allowed by Proposition \ref{extendedWF}. If we drop all the negative terms in (\ref{WFn}), we use (\ref{lbound}) in Appendix \ref{A1} (for $C^1$ functions instead of Lipschitz functions), and $\phi_n(x)\leq x^{-1}$, then

\begin{align*}
\frac{d}{dt}\int_{[0,\infty)}e^{\eta x}&u_n(t,x)dx\leq\frac{1}{2}\int_{[0,\infty)}u_n(t,x)\int_x^{\infty}|\ell_{\varphi}(x,y)|dydx\\
&\leq \frac{C_*(1-\theta)}{2\theta^2(1+\theta)}\int_{[0,\infty)}u_n(t,x)e^{\frac{x}{2}}\int_x^{\infty}\varphi'(y)e^{-\frac{y}{2}}dydx\\
&\leq C_{\eta}\int_{[0,\infty)}e^{\eta x}u_n(t,x)dx,
\end{align*}
from where (\ref{exp}) follows using Gronwall's inequality.
\end{proof}

\noindent
We prove now the following pre-compactness result of  $\{u_n(t)\} _{ n\in \NN }$ for any fixed $t>0$.
\begin{proposition} 
\label{precompactness} Let $u_n$ and $u_0$ be as in Proposition \ref{theoremESMI}. Then, for every fixed $t>0$, there exist a subsequence of $\{u_n(t)\}_{n\in\NN}$ (not relabelled) and $U \in\mathscr{M}_+([0,\infty))$ 
such that, for all $\varphi\in C_0([0,\infty))$,
\begin{align}
\label{WC1}
\lim_{n\to\infty}\int_{[0,\infty)}\varphi(x)u_n(t,x)dx=\int_{[0,\infty)}\varphi(x)U(x)dx.
\end{align}
Moreover, if $u_0$ satisfies $X_{\eta}(u_0)<\infty$ for some $\eta\in(0,1/2)$, then 
\begin{align}
\label{expU}
X_{\eta}(U)\leq e^{C_{\eta}t}X_{\eta}(u_0),
\end{align}
where $C_{\eta}$ is defined in (\ref{Ceta}), and (\ref{WC1}) holds for all $\varphi\in C([0,\infty))$ satisfying the growth condition
\begin{align}
\label{GC}
|\varphi(x)|\leq c e^{\alpha x}\quad \forall x\in[0,\infty),\;c>0,\;0\leq \alpha<\eta.
\end{align}
\end{proposition}

\begin{proof}
By (\ref{CONSER}), the sequence $\{u_n(t)\}_{n\in\NN}$ is uniformly bounded in $\mathscr{M}_+([0,\infty))$, and thus has a subsequence, still denoted 
$u_n(t)$, that converges to some $U\in\mathscr{M}_+([0,\infty))$ in $\sigma(\mathscr{M}([0,\infty)),C_0([0,\infty)))$ (the weak* topology), i.e., (\ref{WC1}) holds for all $\varphi\in C_0([0,\infty))$.
Moreover, if $\zeta_j\in C_c([0,\infty))$ is such that $0\leq\zeta_j\leq 1$, 
$\zeta_j(x)=1$ for all $x\in[0,j]$ and $\zeta_j(x)=0$ for all $x\geq j+1$, so that $\zeta_j\to 1$, then by weak* convergence and (\ref{CONSER}),
\begin{align*}
\int_{[0,\infty)}\zeta_j(x)U(x)dx&=\lim _{ n\to \infty  }\int_{[0,\infty)}\zeta_j(x)u_n(t,x)dx \nonumber \\
&\leq \lim_{n\to\infty}\int_{[0,\infty)}u_n(t,x)dx=\int_{[0,\infty)}u_0(x)dx,
\end{align*}
and then, as $j\to\infty$,
\begin{equation}
\label{WC1'}
\int_{[0,\infty)}U(x)dx\le \int_{[0,\infty)}u_0(x)dx.
\end{equation}

Suppose now that $u_0$ satisfies (\ref{eu0}) for some $\eta\in (0,1/2)$, and let $\psi(x)=e^{\eta x}$ and $\psi_j=\psi\zeta_j$, where $\zeta_j$ is as before. Then, by weak* convergence and Proposition \ref{exponential decay},

\begin{align*}
\int_{[0,\infty)}\psi_j(x)&U(x)dx=\lim_{n\to\infty}\int_{[0,\infty)}\psi_j(x)u_n(t,x)dx\\
&\leq\liminf_{n\to\infty}\int_{[0,\infty)}e^{\eta x}u_n(t,x)dx\leq e^{C_{\eta}t}\int_{[0,\infty)}e^{\eta x}u_0(x)dx,
\end{align*}
and letting $j\to\infty$, (\ref{expU}) holds. 

Let now $\varphi\in C([0,\infty))$ satisfying (\ref{GC}), and define $\varphi_j=\varphi\zeta_j$, with $\zeta_j$ as before, so that $\varphi_j\to\varphi$ pointwise as $j\to\infty$.
Then, for all $j\in\NN$, 
\begin{align}
\label{three terms}
&\bigg|\int_{[0,\infty)}\varphi(x)u_n(t,x)dx-\int_{[0,\infty)}\varphi(x)U(x)dx\bigg|\nonumber\\
&\leq\bigg|\int_{[0,\infty)}\varphi_j(x)u_n(t,x)dx-\int_{[0,\infty)}\varphi_j(x)U(x)dx\bigg|\\
&+\int_{[0,\infty)}|\varphi(x)-\varphi_j(x)|u_n(t,x)dx+\int_{[0,\infty)}|\varphi(x)-\varphi_j(x)|U(x)dx.\nonumber
\end{align}
By (\ref{WC1}), the first term in the right hand side above converges to zero as $n\to\infty$ for all $j\in\NN$. We just need to prove that the second and the third terms are arbitrarly small (for $j$ large enough). Both terms are treated in the same way. We use that $\varphi_j=\varphi$ on $[0,j]$, (\ref{GC}), and Proposition \ref{exponential decay} to obtain
\begin{align*}
&\int_{[0,\infty)}|\varphi(x)-\varphi_j(x)|u_n(t,x)dx=\int_{(j,\infty)}|\varphi(x)-\varphi_j(x)|u_n(t,x)dx\\
&\leq 2\int_{(j,\infty)}|\varphi(x)|u_n(t,x)dx\leq 2c\int_{(j,\infty)}e^{\alpha x}u_n(t,x)dx\\
&\leq 2ce^{(\alpha-\eta) j}\int_{(j,\infty)}e^{\eta x}u_n(t,x)dx\leq 2ce^{(\alpha-\eta)j}e^{C_{\eta}t}\int_{[0,\infty)}e^{\eta x}u_0(x)dx,
\end{align*}
and by similar estimates, and (\ref{expU}),
\begin{align*}
\int_{[0,\infty)}&|\varphi(x)-\varphi_j(x)|U(x)dx\leq 2ce^{(\alpha-\eta)j}e^{C_{\eta}t}\int_{[0,\infty)}e^{\eta x}u_0(x)dx.
\end{align*}
Since $\alpha<\eta$, both terms converges to zero as $j\to\infty$.
\end{proof}

The equicontinuity  of  $\{u_n\}_{n\in\NN}$ in the narrow topology is proved in the following Proposition,
\begin{proposition}
\label{equicontinuity} Let $u_n$ and $u_0$ be as in Proposition \ref{theoremESMI}, and suppose that $X_{\eta}(u_0)<\infty$ for some $\eta\in\left[\frac{1-\theta}{2},\frac{1}{2}\right)$.
Then, for all $n\in\NN$, $\varphi$ $L$-Lipschitz on $[0,\infty)$, $0<T<\infty$ and $t$, $t_0\in[0,T]$,
\begin{gather}
\label{S522}
\bigg|\int_{[0,\infty)}\varphi(x)u_n(t,x)dx-\int_{[0,\infty)}\varphi(x)u_n(t_0,x)dx\bigg|\leq C(u_0,T)|t-t_0|,
\shortintertext{where}
C(u_0,T)=
LC_*\left[AM_0(u_0)+\frac{(1-\theta)}{2\theta^2(1+\theta)}\right]e^{TC_{\frac{(1-\theta)}{2}}}X_{\frac{(1-\theta)}{2}}(u_0),\nonumber
\end{gather}
and $A$ is given in (\ref{kbound}).
In particular, the sequence $\{u_n\} _{ n\in \NN }$ is equicontinuous from $[0,\infty)$ into $\mathscr{M}_+([0,\infty)$ with the narrow topology.
\end{proposition}

\begin{proof}
Let $\varphi$ be $L$-Lipschitz, $0<T<\infty$ and let $t$, $t_0\in[0,T]$ with $t_0\leq t$. By (\ref{WFn}) 
\begin{align}
\label{equi}
&\bigg|\int_{[0,\infty)}\varphi(x)u_n(t,x)dx-\int_{[0,\infty)}\varphi(x)u_n(t_0,x)dx\bigg|\nonumber\\
&\leq\int_{t_0}^t\Big(|K_{\varphi,n}(u_n(s),u_n(s))|+|L_{\varphi,n}(u_n(s))|\Big)ds.
\end{align}
By (\ref{Kbound}), Remark \ref{nbounds}, (\ref{CONSER}), and Proposition \ref{exponential decay},
\begin{align}
\label{Kequi}
\int_{t_0}^t&|K_{\varphi,n}(u_n(s),u_n(s))|ds\leq LC_*AM_0(u_0)\int_{t_0}^tX_{\frac{(1-\theta)}{2}}(u_n(s))ds\nonumber\\
&\leq LC_*AM_0(u_0)e^{tC_{\frac{(1-\theta)}{2}}}X_{\frac{(1-\theta)}{2}}(u_0)(t-t_0),
\end{align}
and by (\ref{Lbound}) (positive part only), Remark \ref{nbounds} and Proposition \ref{exponential decay},
\begin{align}
\label{Lequi}
\int_{t_0}^t& |L_{\varphi,n}(u_n(s))|ds\leq\frac{LC_*(1-\theta)}{2\theta^2(1+\theta)}
\int_{t_0}^tX_{\frac{(1-\theta)}{2}}(u_n(s))ds\nonumber\\
&\leq \frac{LC_*(1-\theta)}{2\theta^2(1+\theta)}e^{tC_{\frac{(1-\theta)}{2}}}X_{\frac{(1-\theta)}{2}}(u_0)(t-t_0).
\end{align}
Using (\ref{Kequi}) and (\ref{Lequi}) in (\ref{equi}),  the estimate (\ref{S522}) follows. 
For the equicontinuity, let $\varepsilon>0$ and consider $\delta<\varepsilon/ C(u_0,T)$. 
By (\ref{mu0}), (\ref{Lip1}), if we take the supremum on (\ref{S522}) among all 
$\varphi\in\text{Lip}_1([0,\infty))$ with $\|\varphi\|_{\infty}\leq 1$, we deduce that for all $t\in[0,T]$, $t_0\in[0,T]$ such that $|t-t_0|<\delta$, then
$d_0(u_n(t),u_n(t_0))<\varepsilon$ for all $n\in\NN$, that is, $\{u_n\}_{n\in\NN}$ is equicontinuous on $[0,T]$.
\end{proof}
As a Corollary of Proposition \ref{precompactness} and Proposition \ref{equicontinuity}, we obtain that a subsequence of 
$\{u_n\}_{n\in\NN}$ converges to a limit 
$u$ in the space $C([0,\infty),\mathscr{M}_+([0,\infty)))$.
\begin{corollary}
\label{cor:u}
Let $u_n$ and $u_0$ be as in Proposition \ref{theoremESMI}, and suppose that $X_{\eta}(u_0)<\infty$ for some $\eta\in\left[\frac{1-\theta}{2},\frac{1}{2}\right)$. Then there exist a subsequence of $\{u_n\}_{n\in\NN}$ (not relabelled) and  
$u\in C([0,\infty),\mathscr{M}_+([0,\infty)))$ such that 
\begin{align}
\label{d0 convergence}
\lim_{n\to\infty}d_0(u_n(t),u(t))=0\quad\forall t\geq 0,
\end{align}
and the convergence is uniform on the compact sets of $[0,\infty)$.
Moreover,
\begin{align}
\label{exp u}
X_{\eta}(u(t))\leq e^{C_{\eta}t}X_{\eta}(u_0)\qquad\forall t\geq 0,
\end{align}
where $C_{\eta}$ is given in (\ref{Ceta}), and for all $\varphi\in C([0,\infty))$ satisfying (\ref{GC}),
\begin{align}
\label{u99}
\lim_{n\to\infty}\int_{[0,\infty)}\varphi(x)u_n(t,x)dx=\int_{[0,\infty)}\varphi(x)u(t,x)dx\quad\forall t\geq 0.
\end{align}
\end{corollary}
\begin{remark}
(\ref{d0 convergence}) implies that, for every $\varphi \in 
C_b([0,\infty))$,
\begin{equation}
\label{NAU}
\lim _{ n\to \infty }\sup _{ t_1\le t\le t_2 }\bigg|\int_{ [0, \infty) }u_n(t, x)\varphi (x)dx-\int  _{ [0, \infty) }u(t, x)\varphi (x)dx\bigg|=0.
\end{equation}
\end{remark}

\begin{proof}
By Proposition \ref{precompactness}, the sequence $\{u_n\}_{n\in\NN}$ is relatively compact on $(\mathscr{M}_+([0,\infty)),d_0)$, and by Proposition \ref{equicontinuity}, the sequence $\{u_n\}_{n\in\NN}$ is equicontinuous from $[0,\infty)$ into $(\mathscr{M}_+([0,\infty)),d_0)$. Then, from Arzel\`{a}-Ascoli theorem, $u_n$ converges pointwise (for all $t\geq 0$) to a continuous function $u$, and the convergence is uniform on compact sets. Since the metric $d_0$ generates the narrow topology, and the convergence in (\ref{d0 convergence}) is uniform on compact sets, then (\ref{NAU}) follows.
The estimate (\ref{exp u}) and the limit (\ref{u99}) are obtained as in Proposition \ref{precompactness}, since the time $t$ is fixed. 
\end{proof}

We prove now that the limit $u$ of the sequence $\{u_n\}_{n\in\NN}$ is indeed a weak solution of (\ref{P})--(\ref{P2}). 

\begin{corollary}
\label{S2Cor15}
Given any $v_0\in\mathscr{M}_+([0,\infty))$ satisfying (\ref{eu0})
for some $\eta\in\left(\frac{1-\theta}{2},\frac{1}{2}\right)$, there exists  $v\in C([0,\infty),\mathscr{M}_+([0,\infty)))$  weak solution of (\ref{P})--(\ref{P2}), that also satisfies (\ref{mass}) and (\ref{expmoment}).
\end{corollary}

\begin{proof}
Let $\{u_n\}_{n\in\NN}$ be the sequence of solutions for the regularised problem (\ref{P100}), (\ref{P200}).
By Corollary \ref{cor:u}, a subsequence of $\{u_n\}_{n\in\NN}$ converges to a limit
$u\in C([0,\infty),\mathscr{M}_+([0,\infty)))$. Since $u$ is continuous from $[0,\infty)$ to ($\mathscr{M}_+([0,\infty)),d_0)$ and $d_0$ generates the narrow topology, then (\ref{ws1}) holds. Next, we prove that $u$ satisfies (\ref{wsID})--(\ref{ws4}). To this end, let $\varphi\in C^1_b([0,\infty))$ with $\varphi'(0)=0$. By (\ref{WFn}), for all $n\in\NN$ and all $t\geq0$,
\begin{align}
\label{IWF}
\int_{[0,\infty)}\varphi(x)u_n(t,x)dx&=\int_{[0,\infty)}\varphi(x)u_0(x)dx\\
&+\int_0^t\Big(K_{\varphi,n}(u_n(s),u_n(s))+L_{\varphi,n}(u_n(s))\Big)ds,\nonumber
\end{align}
and our goal is now  to pass to the limit as $n\to\infty$ term by term. By (\ref{NAU}), for all $t\geq 0$,
\begin{align}
\label{left}
\lim_{n\to\infty}\int_{[0,\infty)}\varphi(x)u_n(t,x)dx=\int_{[0,\infty)}\varphi(x)u(t,x)dx.
\end{align}
Let us prove that for all $t\geq 0$,
\begin{align}
\label{Lconvergence}
&\lim_{n\to\infty}L_{\varphi,n}(u_n(t))=L_{\varphi}(u(t)),\\
\label{Kconvergence}
&\lim_{n\to\infty}K_{\varphi,n}(u_n(t),u_n(t))=K_{\varphi}(u(t),n(t)).
\end{align}
Starting with (\ref{Lconvergence}), we have
\begin{align}
\label{LL1}
\!\!\!\big|L_{\varphi}(u)-L_{\varphi,n}(u_n)\big|\leq&\big|L_{\varphi}(u)-L_{\varphi}(u_n)\big|+|L_{\varphi}(u_n)-L_{\varphi,n}(u_n)|.
\end{align}
Since $\mathcal{L}_{\varphi}\in C([0,\infty))$ and $\mathcal{L}_{\varphi}$ satisfies the growth condition (\ref{GC}) with $\alpha=(1-\theta)/2$, (cf. Lemma \ref{klbounds}), then by (\ref{u99}) the first term in the right hand side of (\ref{LL1}) converges to zero as $n\to\infty$. For the second term we have, for any $R>0$,
\begin{align*}
|L_{\varphi}(u_n)-L_{\varphi,n}(u_n)|&\leq \int_{[0,R]} |\mathcal{L}_{\varphi}(x)-\mathcal{L}_{\varphi,n}(x)|u_n(t,x)dx\\
&+\int_{(R,\infty)} |\mathcal{L}_{\varphi}(x)-\mathcal{L}_{\varphi,n}(x)|u_n(t,x)dx.
\end{align*}
On the one hand, using (\ref{CONSER}),
\begin{align*}
\int_{[0,R]} |\mathcal{L}_{\varphi}(x)-\mathcal{L}_{\varphi,n}(x)|u_n(t,x)dx\leq M_0(u_0)\|\mathcal{L}_{\varphi}-\mathcal{L}_{\varphi,n}\|_{C([0,R])},
\end{align*}
which converges to zero as $n\to\infty$ by Lemma \ref{llconver}.
On the other hand, by (\ref{llbound}),
\begin{align}
\label{Ltail}
&\int_{(R,\infty)} |\mathcal{L}_{\varphi}(x)-\mathcal{L}_{\varphi,n}(x)|u_n(t,x)dx
\leq 2\int_{(R,\infty)}|\mathcal{L}_{\varphi}(x)|u_n(t,x)dx\nonumber\\
&\leq C \int_{(R,\infty)}e^{\frac{(1-\theta)x}{2}}u_n(t,x)dx\leq C e^{R\left(\frac{1-\theta}{2}-\eta\right)}\int_{(R,\infty)}e^{\eta x}u_n(t,x)dx,
\end{align}
where $C=\frac{LC_*(1-\theta)}{\theta^2(1+\theta)}$,
and by Proposition \ref{exponential decay} we deduce that (\ref{Ltail}) converges to zero as $R\to\infty$. 
That concludes the proof of (\ref{Lconvergence}).

In order to prove (\ref{Kconvergence}), we use
\begin{align}
\label{KK1}
|K_{\varphi}(u,u)-K_{\varphi,n}(u_n,u_n)|&\leq|K_{\varphi}(u,u)-K_{\varphi}(u_n,u_n)|\nonumber\\
&+|K_{\varphi}(u_n,u_n)-K_{\varphi,n}(u_n,u_n)|.
\end{align}

Then, for the first term in the right hand side of (\ref{KK1}), given $R>0$, we use
\begin{align*}
&\iint_{[0,\infty)^2}k_{\varphi}(x,y)u(x)u(y)dydx\\
&\leq\bigg(\iint_{[0,R]^2}+\iint_{[\gamma_1(R),\infty)^2}-\iint_{[\gamma_1(R),R]}\bigg)k_{\varphi}(x,y)u(x)u(y)dydx,
\end{align*}
to deduce
\begin{align}
\label{K-K}
&|K_{\varphi}(u,u)-K_{\varphi}(u_n,u_n)|\leq I_1+I_2+I_3,\\
&I_1=\bigg|\iint_{[0,R]^2}k_{\varphi}u(x)u(y)dydx-\iint_{[0,R]^2}k_{\varphi}u_n(x)u_n(y)dydx\bigg|,\nonumber\\
&I_2=\bigg|\iint_{[\gamma_1(R),R]^2}k_{\varphi}u(x)u(y)dydx-\iint_{[\gamma_1(R),R]^2}k_{\varphi}u_n(x)u_n(y)dydx\bigg|,\nonumber\\
&I_3=\bigg|\iint_{[\gamma_1(R),\infty)^2}k_{\varphi}u(x)u(y)dydx-\iint_{[\gamma_1(R),\infty)^2}k_{\varphi}u_n(x)u_n(y)dydx\bigg|.\nonumber
\end{align}
Since $k_{\varphi}\in C([0,\infty)^2)$ (cf. Lemma \ref{k continuous}), then by Stone-Weierstrass theorem, $k_{\varphi}(x,y)$ can be approximated on any compact subset $X\subset[0,\infty)^2$ by functions of the form $\psi_1(x)\psi_2(y)$, with $\psi_i\in C(X)$ for $i=1,2$. By Tietze extension theorem we may assume that $\psi_i\in C([0,\infty))$ for $i=1,2$.  Then, using that $u_n$ converges narrowly to $u$, we deduce that for any $\varepsilon>0$, $R>0$, there exists $n_*\in\NN$ such that for all $n\geq n_*$  
\begin{align}
\label{K-K2}
I_1<\varepsilon,\qquad I_2<\varepsilon.
\end{align}
Then, for $I_3$ we have the following. 
\begin{align}
\label{I3bound1}
I_3\leq \iint_{[\gamma_1(R),\infty)^2}k_{\varphi}(x,y)\big(u(x)u(y)+u_n(x)u_n(y)\big)dydx,
\end{align}
and by (\ref{kbound}), calling $C=\|\varphi'\|_{\infty} C_* A$,
\begin{align}
\label{I3bound2}
&\iint_{[\gamma_1(R),\infty)^2}\!\!\!k_{\varphi}u(t,x)u(t,y)dydx
\leq C \iint_{[\gamma_1(R),\infty)^2}\!\!\!e^{\frac{|x-y|}{2}}u(t,x)u(t,y)dydx\nonumber\\
&\leq 2C\int_{[\gamma_1(R),\infty)}e^{\frac{(1-\theta)x}{2}}u(t,x)\int_{[\gamma_1(R),x]}u(t,y)dydx\nonumber\\
&\leq 2C X_{\eta}(u(t))\int_{[\gamma_1(R),\infty)}u(t,y)dy.
\end{align}
We now use that for all $x>0$, $t>0$, there exists $R>0$ such that
\begin{align}
\label{I3bound3}
\int_{[\gamma_1(R),x]}&u(t,y)dy\leq \frac{1}{\gamma_1(R)}\int_{[\gamma_1(R),\infty)}yu(t,y)dy\nonumber\\
&\leq \frac{1}{\gamma_1(R)}\int_{[\gamma_1(R),\infty)}e^{\eta y}u(t,y)dy
\leq\frac{e^{C_{\eta}t}X_{\eta}(u_0)}{\gamma_1(R)},
\end{align}
where we have used (\ref{exp u}). Using (\ref{I3bound3}) in (\ref{I3bound2}), and (\ref{exp u}) again,
\begin{align*}
&\iint_{[\gamma_1(R),\infty)^2}k_{\varphi}u(t,x)u(t,y)dydx\leq\frac{2Ce^{2C_{\eta}t}(X_{\eta}(u_0))^2}{\gamma_1(R)},
\end{align*} 
and the same estimate holds when $u$ is replaced by $u_n$. We then obtain from (\ref{I3bound1}) that, for any $\varepsilon>0$, there exists $R>0$ such that $I_3<\varepsilon$ for all $n\in\NN$. Combining this with (\ref{K-K2}), we then deduce from (\ref{K-K}) that for all $t>0$
\begin{align}
\label{K-K3}
\lim_{n\to\infty}|K_{\varphi}(u(t),u(t))-K_{\varphi}(u_n(t),u_n(t))|=0.
\end{align}

Now, for the second term in the right hand side of (\ref{KK1}), we have
\begin{align*}
&|K_{\varphi}(u_n,u_n)-K_{\varphi,n}(u_n,u_n)|\\
&\leq \int_{[0,\infty)}\int_{[0,x]}|k_{\varphi}(x,y)||1-xy\phi_n(x)\phi_n(y)|u_n(t,x)u_n(t,y)dydx,
\end{align*} 
and we decompose the integral above as follows:
\begin{align}
\label{split}
\int_{[0,\infty)}\int_{[0,x]}=\int_{\frac{1}{n}}^n\int_{\frac{1}{n}}^x+\int_n^{\infty}\int_0^x+\int_0^n\int_0^{\min\{x,\frac{1}{n}\}}.
\end{align}
By definition $\phi_n(x)=x^{-1}$ for all $x\in[1/n,n]$, and then
\begin{align*}
\int_{\frac{1}{n}}^n\int_{\frac{1}{n}}^x|k_{\varphi}(x,y)||1-xy\phi_n(x)\phi_n(y)|u_n(t,x)u_n(t,y)dydx=0.
\end{align*}
Now, by (\ref{Kbound}) and (\ref{CONSER}),
\begin{align*}
&\int_n^{\infty}\int_0^x |k_{\varphi}(x,y)|u_n(t,x)u_n(t,y)dydx\\
&\leq LC_*AM_0(u_0)\int_n^{\infty}e^{\frac{(1-\theta)x}{2}}u_n(t,x)dx\\
&\leq  LC_*AM_0(u_0) e^{n\left(\frac{1-\theta}{2}-\eta\right)}\int_n^{\infty}e^{\eta x}u_n(t,x)dx,
\end{align*}
and from Proposition \ref{exponential decay} we deduce that it converges to zero as $n\to\infty$.
For the las term in the right hand side of (\ref{split}), we argue as follows. Let us define $x_n=\gamma_2(1/n)$ and $D_n=[0,x_n]\times[0,1/n]$. Notice that $x_n\to 0$ as $n\to\infty$.
Then by (\ref{CONSER})
\begin{align*}
\int_0^n\int_0^{\min\{x,\frac{1}{n}\}}|k_{\varphi}(x,y)||1-xy\phi_n(x)\phi_n(y)|u_n(t,x)u_n(t,y)dydx\nonumber\\
\leq\max_{(x,y)\in D_n}|k_{\varphi}(x,y)|M_0(u_0)^2.
\end{align*}
Since $k_{\varphi}(0,0)=0$ and $k_{\varphi}$ is continuous (cf. Lemma \ref{k continuous}), it follows that  $k_{\varphi}(x,y)\to 0$ for all $(x,y)\in D_n$ as $n\to \infty$. That concludes the proof of (\ref{Kconvergence}). 

From the limits (\ref{Lconvergence}), (\ref{Kconvergence}), the uniform bounds (\ref{Kequi}), (\ref{Lequi}), dominated convergence theorem and (\ref{left}), we obtain 
\begin{align}
\label{IWF'}
\int_{[0,\infty)}\varphi(x)u(t,x)dx&=\int_{[0,\infty)}\varphi(x)u_0(x)dx\nonumber\\
&+\int_0^t\Big(K_{\varphi}(u(s),u(s))+L_{\varphi}(u(s))\Big)ds.
\end{align}
The identity (\ref{wsID}) then follows from (\ref{IWF'}) for $t=0$.
It follows from Proposition \ref{equicontinuity}, by passage to the limit as $n\to \infty$,  that for any $\varphi\in C^1_b([0,\infty))$ with $\varphi'(0)=0$, the map  $t\mapsto \int  _{ [0, \infty) }u(t, x)\varphi (x)dx$ is locally Lipschitz on $[0, \infty)$, i.e., (\ref{ws2}) holds, and then from (\ref{IWF'}), the weak formulation (\ref{ws4}) follows. Taking $\varphi=1$ in (\ref{ws4}), we obtain (\ref{mass}). The estimate (\ref{expmoment}) is just (\ref{exp u}).
\end{proof}

\begin{remark} 
\label{S2EXIP1}Because of the exponential growth of the kernel $B$,  an exponential moment is required on the initial data $u_0$. This exponential moment is propagated to the solution for all $t>0$. 
Using that exponential moment, it easily follows that for any $\rho \ge 1$, if  $M _{ -\rho  }(u_0)<\infty$, there exists a constant $C_1>0$, and a non  negative locally bounded function $C_2(t)$ such that,
\begin{align*}
\frac {d} {dt} \int  _{ [0, \infty )}u(t, x)x^{-\rho }dx \le C_1\bigg( \int  _{ [0, \infty )} u(t, x)x^{-\rho }dx\bigg)^2+C_2(t),
\end{align*}
from where it follows that   $M _{ -\rho  }(u(t))<\infty$ for $t$ in a bounded interval, that depends on $M_\rho (u_0)$. 
\end{remark}

\begin{proof}[\bfseries\upshape{Proof of Theorem \ref{MT1} }]
 Theorem \ref{MT1} follows from Corollary \ref{S2Cor15} since the function $b(k, k')=\frac{\Phi\mathcal B_\beta}{kk'}$ satisfies (\ref{B1})--(\ref{Bbound}).
\end{proof}

\section{The singular part of the solution.}
\setcounter{equation}{0}
\setcounter{theorem}{0}

If $u$ is a  weak solution of (\ref{P})--(\ref{Bbound})  obtained in Theorem \ref{MT1}, for all $t>0$, the measure $u(t)$ may now be decomposed by the Lebesgue's decomposition Theorem as
\begin{align}
\label{deco1}
&u(t)=g(t)+\alpha(t)\delta_0+G(t),\\
\label{deco2}
&g(t)\in L^1([0,\infty)),\;\alpha\geq 0,\;G(t)\perp dx,\; G(t,\{0\})=0.
\end{align}
In this Section we give some properties of  $u$, $\alpha$, and $G$.

We first notice that the weak solution $u$ of (\ref{P})--(\ref{Bbound}) obtained in Theorem \ref{MT1}, satisfies the equation (\ref{P}) in the sense of distributions. This  follows from the properties of the support of the function $B$ and  Fubini's Theorem. A similar argument may be used for slightly more general  test functions $\varphi$.
To be more precise, let us define the set 
\begin{align}
\label{subspace}
\mathscr{C}=\Big\{\varphi\in C_b([0,\infty)):\sup_{x\geq 0}\frac{|\varphi(x)|}{x^{3/2}}<\infty\Big\}.
\end{align}  

\begin{proposition}
Let  $u$ be a solution of (\ref{P})--(\ref{Bbound}) obtained in Theorem \ref{MT1}. Then, for almost every $t>0$, $\partial u/\partial t \in \mathscr D'((0, \infty)) $, $Q(u(t), u(t))\in \mathscr D'((0, \infty))$, and
\begin{align}
\forall  \varphi \in C_c((0, \infty)),\quad
\frac {d} {dt}\langle u(t), \varphi  \rangle =\langle Q(u(t), u(t)), \varphi  \rangle . \label{halfweak1}
\end{align}
Moreover, 
\begin{align}
\label{halfweak2}
\forall\varphi\in\mathscr{C},\quad\frac {d} {dt}\langle u(t), \varphi  \rangle =\langle \mathcal{Q}(u(t), u(t)), \varphi  \rangle,
\end{align}
where
\begin{align}
\label{anotherQ}
\mathcal{Q}(u(t),u(t))=\int_{[0,\infty)}b(x,y)&\Big[(e^{-x}-e^{-y})u(t,x)u(t,y)\nonumber\\
&-u(t,x)y^2e^{-y}+u(t,y)x^2e^{-x}\Big]dy.
\end{align}
\end{proposition}
\begin{remark}
Notice that in (\ref{anotherQ}),  the integral containing the factor $(e^{-x}-e^{-y})$  is convergent near the origin even for test functions $\varphi \in \mathscr{C}\setminus C_c((0, \infty))$. That is not true anymore if we consider each of the terms $e^{-x}$ and $e^{-y}$ separately.
\end{remark}

\begin{proof}
By  (\ref{Bsupport})--(\ref{Bbound}) and (\ref{mass}),
\begin{align*}
\int_{[0,\infty)}|\varphi(x)|\int_{[0,\infty)}\frac{B(x,y)}{xy}E(x,y)dydx<\infty,
\end{align*}
 where $E(x, y)$ is one of the functions in  
 $$\Big\{u(y)x^2e^{-x}, u(x)y^2e^{-y}, u(x)u(y)e^{-x}, u(x)u(y)e^{-y}\Big\}$$ when $\varphi\in C^1_c((0,\infty))$, 
 or, one of the functions in
$$\Big\{ u(x)u(y)|e^{-x}-e^{-y}|,u(x)y^2e^{-y},u(y)x^2e^{-x}\Big\}$$ when $\varphi\in\mathscr{C}\cap C^1_b([0,\infty)).$
Since $u$ is a weak solution and satisfies (\ref{ws4}), we deduce from Fubini's Theorem the identity (\ref{halfweak1}) for 
$\varphi \in C^1_c((0, \infty))$, and the identities (\ref{halfweak2})-(\ref{anotherQ}) for 
$\varphi\in\mathscr{C}\cap C^1_b([0,\infty))$. By a density argument the Proposition follows.
\end{proof}

\noindent
We may prove now the following property of the singular  measure $G(t)$.
\begin{theorem}
Let  $u$ be a weak solution of (\ref{P})--(\ref{Bbound}) obtained in Theorem \ref{MT1}, and consider the decomposition (\ref{deco1}), (\ref{deco2}). If $G(0)=0$ in $\mathscr D'((0, \infty))$, then $G(t)=0$ in $\mathscr D'((0, \infty))$ for all $t>0$.
\end{theorem}

\begin{proof}
By (\ref{halfweak1}), for a.e. $t>0$ and for all $\varphi\in C_c((0,\infty))$,
\begin{align*}
\frac {d} {dt}\int  _{[0, \infty) }u(t, x)\varphi (x)dx=\int _{[0,\infty)} \varphi (x) Q(u(t), u(t))(x)dx,
\end{align*}
and then, after integration in time:
\begin{align*}
\int  _{[0, \infty) } \left\{u(t,x)-u(0,x)-\int _0^t Q(u(s), u(s))(x)ds\right\}\varphi (x)dx=0
\end{align*}
for a.e. $t>0$. 
If we plug now $u=g+\alpha \delta _0+G$ in this formula and use that $\varphi\in C_c((0,\infty))$, we obtain for a.e. $t>0$,
\begin{align*}
&\int  _{[0, \infty) } \Big\{g(t)+G(t)-g(0)-G(0)-R(t)-S(t)\Big\}\varphi (x)dx=0,
\end{align*}
where
\begin{align}
&R(t,x)=\int_0^t\!\bigg(g(s,x)W(s,x)+x^2e^{-x}\int_{ [0, \infty)}\!\!b(x, y)u(s, y)dy\bigg)ds, \label{Regular 123}\\
&S(t,x)=\int _0^t G(s,x)W(s, x)ds,\label{Singular 56}\\
&W(s, x)=\int_{[0,\infty)}b(x,y)(e^{-x}-e^{-y})u(s,y)dy-\int_{[0,\infty)}b(x,y)y^2e^{-y}dy.\label{W expression}
\end{align}
It follows that, for a.e. $t>0$,
\begin{equation}
\label{S6EX3}
g(t)+G(t)-g(0)-G(0)-R(t)-S(t)=0\,\,\hbox{in}\,\,\mathscr D'((0, \infty)).
\end{equation}
Let us prove now that $R(t,\cdot)\in L^1_{loc}((0,\infty))$ for all $t\geq 0$. To this end, we first show that $W(t,\cdot)\in L^{\infty}_{loc}((0,\infty))$ for all $t\geq 0$. Let then $x$ be in a compact set $[a,c]$, with $0<a<c<\infty$, and let $t\geq 0$. Using that $\supp(b)=\Gamma\subset\{(x,y)\in[0,\infty)^2:\theta x\leq y\leq \theta^{-1}x\}$, the bound (\ref{Bbound}), and that $x\in[a,c]$, it is easily proved that there exists a constant $0<C<\infty$ that depends only on $a$, $c$, $\theta$ and $C_*$, such that for all $(x,y)\in\Gamma$ with $x\in[a,c]$,
\begin{equation}
\label{local estimates 23}
b(x,y)|e^{-x}-e^{-y}|\leq C\qquad\text{and}\qquad b(x,y)\max\{x^2,y^2\}\leq C.
\end{equation}
We then obtain from (\ref{W expression}) that for all $t\geq 0$, $x\in [a,c]$,
\begin{equation*}
|W(t,x)|\leq C(M_0(u(t))+1),
\end{equation*}
and by the conservation of mass (\ref{mass}),
\begin{equation}
\label{W bound local}
\sup_{t\geq 0}\|W(t,\cdot)\|_{L^{\infty}([a,c])}\leq C(M_0(u_0)+1).
\end{equation}
Using now (\ref{local estimates 23}), (\ref{W bound local}) and (\ref{mass}), we deduce from (\ref{Regular 123}) that for all $t\geq 0$, $x\in [a,c]$,
\begin{equation}
\label{R bound 456}
|R(t,x)|\leq C(M_0(u_0)+1)\int_0^t  g(s,x)ds +CM_0(u_0) t.
\end{equation}
Then, since $\sup_{t\geq 0}\|g(t,\cdot)\|_{L^1([a,c])}\leq \sup_{t\geq 0} M_0(u(t))=M_0(u_0)$, it follows from (\ref{R bound 456}) that
\begin{equation}
\label{R bound 457}
\|R(t,\cdot)\|_{L^1([a,c])}\leq C(M_0(u_0)+1) M_0(u_0) t+(c-a)CM_0(u_0) t.
\end{equation}

On the other hand, using the Lebesgue decomposition Theorem, we have for all $t\geq 0$:
\begin{equation*}
S(t)=S_{ac}(t)+S_s(t),\qquad S_{ac}(t)\in L^1([0,\infty)),\quad S_s(t)\perp dx.
\end{equation*}
Using this decomposition in (\ref{S6EX3}), we deduce that for a.e. $t>0$,
\begin{equation*}
g(t)-g(0)-R(t)-S_{ac}(t)=-G(t)+G(0)+S_s(t)\quad\hbox{in}\,\,\mathscr D'((0, \infty)).
\end{equation*}
Since the left hand side is absolutely continuous with respect to the Lebesgue measure and the right hand side is singular, we then obtain
for a.e. $t>0$,
\begin{align*}
&g(t)=g(0)+R(t)+S_{ac}(t)\qquad\hbox{in}\quad\mathscr D'((0, \infty)),\\
&G(t)=G(0)+S_s(t)\qquad\hbox{in}\quad\mathscr D'((0, \infty)).
\end{align*}
Then, for all $\varphi \in C_c((0, \infty))$ and a.e. $t>0$,
\begin{equation}
\label{equation singular 98}
\int _{[0,\infty)}  \varphi (x)G(t, x)dx=\int _{[0,\infty)} \varphi (x) G(0, x)dx+\int_{[0,\infty)}\varphi (x)S_s(t, x)dx.
\end{equation}
We use now that for all nonnegative $\varphi\in C_c((0,\infty))$, $t\geq 0$,
\begin{equation}
\label{total variation 45}
\int_{[0,\infty)}\varphi (x)S_s(t, x)dx\leq \int_{[0,\infty)}\varphi (x)|S_s(t, x)|dx\leq \int_{[0,\infty)}\varphi (x)|S(t, x)|dx,
\end{equation}
where $|S_s(t)|$ and $|S(t)|$ are the total variation measures of $S_s(t)$ and $S(t)$ respectively.
Then, if $\varphi\geq 0$ and $\supp(\varphi) \subset [a, c]$ for  finite $c>a>0$, we deduce from (\ref{Singular 56}), (\ref{W bound local}) and (\ref{total variation 45}) that
\begin{align*}
\int_{[0,\infty)}\varphi (x)S_s(t, x)dx&\leq \int _0^t \|W(s,\cdot)\|_{L^{\infty}([a,c])} \int_{[0,\infty)}\varphi (x)G(s, x)dxds\\
&\leq  C(M_0(u_0)+1)\int_0^t  \int_{[0,\infty)}\varphi (x)G(s, x)dxds,
\end{align*}
and then, we obtain from (\ref{equation singular 98}) that, for a.e. $t>0$,
\begin{align*}
\int _{[0,\infty)} \varphi (x) G(t, x)dx \le&\int _{[0,\infty)}\varphi (x)G(0, x)dx\\
&+C(M_0(u_0)+1)\int_0^t  \int_{[0,\infty)}\varphi (x)G(s, x)dxds.
\end{align*} 
Then by Gronwall's Lemma,
\begin{align*}
&\int _{[0,\infty)} \varphi (x) G(t, x)dx\le\\
&\leq\bigg(\int _{[0,\infty)} \varphi (x) G(0, x)dx\bigg)\bigg(1+C(M_0(u_0)+1) t e^{C(M_0(u_0)+1) t}\bigg).
\end{align*}
We deduce that, if $G(0)= 0$, then  $\int _{[0,\infty)}\varphi (x) G(t, x)dx=0$ for every $\varphi \in C_c((0,\infty))$ and then,
$G(t)=0$ in $\mathscr {D}'((0, \infty))$ for a.e. $t>0$.
\end{proof}
\begin{remark}
It would be interesting to know if, when $\alpha (0)=0$,  $\alpha(t)=0$ for a.e. $t>0$ also  or not.
\end{remark}

\subsection{An equation for the mass at the origin}

 We can obtain information of the measure at the origin $u(t,\{0\})$ from the weak formulation (\ref{ws4}), by choosing test functions like in the following Remark. 
\begin{remark}
\label{test}
Let $\varphi\in C^1_b([0,\infty))$ be nonincreasing with $\supp\varphi=[0,1]$, $\varphi(0)=1$ and $\varphi'(0)=0$. 
Then, let $\varphi_{\varepsilon}(x)=\varphi(x/\varepsilon)$ for $\varepsilon>0$. 
It follows from (\ref{mass}) and dominated convergence that
for all $t\geq 0$,
\begin{align}
\label{origin}
\lim_{\varepsilon\to 0}\int_{[0,\infty)}\varphi_{\varepsilon}(x)u(t,x)dx=u(t,\{0\}).
\end{align}
\end{remark}

\begin{proposition}
\label{prop origin}
Let  $u$ be a weak solution of (\ref{P})--(\ref{Bbound}) obtained in Theorem \ref{MT1}, and denote $\alpha(t)=u(t,\{0\})$. Then $\alpha$ is right continuous, nondecreasing and a.e. differentiable on $[0,\infty)$. Moreover, for all $t$ and $t_0$ with $t\geq t_0\geq 0$, and all $\varphi_{\varepsilon}$ as in Remark \ref{test}, the following limit exists:
\begin{gather}
\label{Klimit}
\lim_{\varepsilon\to 0}\int_{t_0}^t K_{\varphi_{\varepsilon}}(u(s),u(s))ds,\\
\shortintertext{and}
\label{eqorigin}
\alpha(t)=\alpha(t_0)+\lim_{\varepsilon\to 0}\int_{t_0}^tK_{\varphi_{\varepsilon}}(u(s),u(s))ds.
\end{gather}
\end{proposition}

\begin{proof}
Let us prove first (\ref{eqorigin}). Using $\varphi_{\varepsilon}$ in (\ref{WFKL}), we deduce by (\ref{origin}) that for all $t$ and $t_0$ with $t\geq t_0\geq 0$, the following limit exists:
\begin{align*}
\lim_{\varepsilon\to 0}\int_{t_0}^t\big(K_{\varphi_{\varepsilon}}(u(s),u(s))-L_{\varphi_{\varepsilon}}(u(s))\big)ds,
\end{align*}
and moreover
\begin{align}
\label{eqorigin1}
\alpha(t)=\alpha(t_0)+\lim_{\varepsilon\to 0}\int_{t_0}^t\big(K_{\varphi_{\varepsilon}}(u(s),u(s))-L_{\varphi_{\varepsilon}}(u(s))\big)ds.
\end{align}
We claim
\begin{align}
\label{Lorigin}
\lim_{\varepsilon\to 0}\int_{t_0}^tL_{\varphi_{\varepsilon}}(u(s))ds=0.
\end{align}
In order to prove (\ref{Lorigin}), we first obtain an integrable majorant of $L_{\varphi_{\varepsilon}}(u(s))$, and then we show 
\begin{align}
\label{Lorigin1}
\lim_{\varepsilon\to 0}L_{\varphi_{\varepsilon}}(u(s))=0\quad\forall s\geq  0.
\end{align}
Taking into account $\Gamma$, the support of $\Psi_{\varepsilon}(x,y)=\varphi_{\varepsilon}(x)-\varphi_{\varepsilon}(y)$, and using $\mathcal{L}_{\varphi_{\varepsilon}}(0)=0$ (cf. Lemma \ref{klbounds}), we have
\begin{align}
\label{maj1}
|L_{\varphi_{\varepsilon}}(u(s))|&\leq \int_{(0,\varepsilon)}u(s,x)\int_{\theta x}^{\theta^{-1}x}|\ell_{\varphi_{\varepsilon}}(x,y)|dydx\nonumber\\
&+\int_{\left[\varepsilon,\frac{\varepsilon}{\theta}\right]}u(s,x)\int_{\theta x}^{\varepsilon}|\ell_{\varphi_{\varepsilon}}(x,y)|dydx.
\end{align}
Since
\begin{align*}
\varphi_{\varepsilon}(x)-\varphi_{\varepsilon}(y)=\int_{\frac{y}{\varepsilon}}^{\frac{x}{\varepsilon}}\varphi'(z)dz\leq
\frac{\|\varphi'\|_{\infty}}{\varepsilon}|x-y|,
\end{align*}
then by (\ref{lbound})
\begin{align*}
|\ell_{\varphi_{\varepsilon}}(x,y)|\leq \frac{c}{\varepsilon}e^{\frac{x-y}{2}},\quad c=\frac{C_*(1-\theta)}{\theta^2(1+\theta)}\|\varphi'\|_{\infty},
\end{align*}
and from (\ref{maj1}) we deduce
\begin{align*}
|L_{\varphi_{\varepsilon}}(u(s))|\leq \frac{2c}{\varepsilon}&
\bigg[\int_{(0,\varepsilon)}u(s,x)\big(e^{\frac{(1-\theta)x}{2}}-e^{\frac{(1-\theta^{-1}x)}{2}}\big)dx\\
&+\int_{\left[\varepsilon,\frac{\varepsilon}{\theta}\right]}u(s,x)\big(e^{\frac{(1-\theta)x}{2}}-e^{\frac{x-\varepsilon}{2}}\big)dx\bigg].
\end{align*}
We now use  
$e^{\frac{(1-\theta)x}{2}}-e^{\frac{(1-\theta^{-1}x)}{2}}\leq \frac{(\theta^{-1}-\theta)x}{2}e^{\frac{(1-\theta)x}{2}}$,
$e^{\frac{(1-\theta)x}{2}}-e^{\frac{x-\varepsilon}{2}}\leq\frac{(\varepsilon-\theta x)}{2}e^{\frac{(1-\theta)x}{2}}$,
and (\ref{exp u}) to obtain, for all $\varepsilon>0$,
\begin{align}
\label{maj2}
|L_{\varphi_{\varepsilon}}(u(s))|&\leq c(\theta^{-1}-\theta)\int_{(0,\varepsilon)}u(s,x)e^{\frac{(1-\theta)x}{2}}dx\nonumber\\
&\quad+c(1-\theta)\int_{\left[\varepsilon,\frac{\varepsilon}{\theta}\right]}u(s,x)e^{\frac{(1-\theta)x}{2}}dx\nonumber\\
&\leq c(\theta^{-1}-\theta)\int_{\left(0,\frac{\varepsilon}{\theta}\right]}e^{\frac{(1-\theta)x}{2}}u(s,x)dx\nonumber\\
&\leq c(\theta^{-1}-\theta)e^{sC_{(1-\theta)/2}}\int_{[0,\infty)}e^{\frac{(1-\theta)x}{2}}u_0(x)dx.
\end{align}
The right hand side above is independent of $\varepsilon$, and it is clearly integrable on $[0,t]$, for all $t>0$.

Let us prove now (\ref{Lorigin1}). If we prove 
\begin{align}
\label{llorigin}
\lim_{\varepsilon\to 0}\mathcal{L}_{\varphi_{\varepsilon}}(x)=0\quad\forall x\geq 0,
\end{align}
then by (\ref{maj2}) and dominated convergence, (\ref{Lorigin1}) follows. Therefore we are left to prove (\ref{llorigin}). 
On the one hand, since $\mathcal{L}_{\varphi_{\varepsilon}}(0)=0$ for all $\varepsilon>0$ (cf. Lemma \ref{klbounds}), then
$\lim_{\varepsilon\to 0}\mathcal{L}_{\varphi_{\varepsilon}}(0)=0$. On the other hand, for all $x>0$ and $y\in[0,\infty)$, the function
$\ell_{\varphi_{\varepsilon}}(x,y)$ is well defined and 
\begin{align}
\label{lorigin3}
\lim_{\varepsilon\to 0}\ell_{\varphi_{\varepsilon}}(x,y)=0.
\end{align}
Moreoever, by (\ref{Bbound})
\begin{align*}
|\ell_{\varphi_{\varepsilon}}(x,y)|&\leq B(x,y)\frac{y}{x}e^{-y}(\varphi_{\varepsilon}(x)+\varphi_{\varepsilon}(y))
\leq 2 C_*\frac{ye^{\frac{x-y}{2}}}{x(x+y)}\mathds{1}_{\Gamma}(x,y),
\end{align*}
and then
\begin{align}
\label{lorigin4}
\int_0^{\infty}|\ell_{\varphi_{\varepsilon}}(x,y)|dy&\leq 2C_*e^{\frac{(1-\theta)x}{2}}\frac{1}{x}\int_{\theta x}^{\theta^{-1}x}\frac{y}{x+y}dy\nonumber\\
&=2C_*e^{\frac{(1-\theta)x}{2}}\int_{\theta}^{\theta^{-1}}\frac{z}{1+z}dz<+\infty.
\end{align}
It follows from (\ref{lorigin3}), (\ref{lorigin4}) and dominated convergence that $\mathcal{L}_{\varphi_{\varepsilon}}(x)\to0$ as 
$\varepsilon\to0$ for all $x>0$, and then (\ref{llorigin}) holds. That proves (\ref{Lorigin1}), which combined with (\ref{maj2}) and dominated convergence, finally proves (\ref{Lorigin}).
Using (\ref{Lorigin}) in (\ref{eqorigin1}), then the limit in (\ref{Klimit}) exists and (\ref{eqorigin}) holds. 

Since $K_{\varphi_{\varepsilon}}(u,u)\geq 0$ for all $\varepsilon>0$, it follows from (\ref{eqorigin}) that $\alpha$ is monotone nondecreasing, and then a.e. differentiable by Lebesgue Theorem.

We are left to prove the right continuity of $\alpha$. Since $\alpha$ is nondecreasing, we already know
\begin{align}
\label{rc1}
\alpha(t)\leq\liminf_{h\to 0^+}\alpha(t+h),
\end{align}
so it is sufficient to prove
\begin{align}
\label{rc2}
\limsup_{h\to 0^+}\alpha(t+h)\leq\alpha(t).
\end{align}
To this end, let $\varphi_{\varepsilon}$ as in Remark \ref{test}. Using 
$\alpha(t+h)\leq\int_{[0,\infty)}\varphi_{\varepsilon}(x)u(t+h,x)dx$ and (\ref{WFKL}) with $\varphi_{\varepsilon}$, we have
\begin{align*}
\alpha(t+h)\!\leq\! \int_{[0,\infty)}\!\!\!\varphi_{\varepsilon}(x)u(t,x)dx+\int_t^{t+h}\!\!\big(K_{\varphi_{\varepsilon}}(u(s),u(s))+L_{\varphi_{\varepsilon}}(u(s))\big)ds.
\end{align*}
From Proposition \ref{KLbounds} and (\ref{exp u}), we deduce that $K_{\varphi_{\varepsilon}}(u(s),u(s))$ and 
$L_{\varphi_{\varepsilon}}(u(s))$ are locally integrable in time for every fixed $\varepsilon>0$, so letting $h\to 0$ above, and then $\varepsilon\to 0$, we finally obtain (\ref{rc2}). The right continuity then follows from 
(\ref{rc1}) and (\ref{rc2}).
\end{proof}

\begin{remark}
By a standard approximation argument, it is possible to use $\mathds{1}_{[0,\varepsilon)}$ as a test function in (\ref{WFKL}). 
Then, by similar arguments as in the proof of Proposition \ref{prop origin}, it can be seen that
equation (\ref{eqorigin}) also holds when $\varphi_{\varepsilon}$ is replaced by $\mathds{1}_{[0,\varepsilon)}$, and then, for all $t\geq t_0\geq 0$, 
\begin{align}
\label{eqorigin2}
\alpha(t)&=\alpha(t_0)+\lim_{\varepsilon\to 0}\int_{t_0}^t\iint_{D_\varepsilon}
\frac{B(x,y)}{xy}(e^{-x}-e^{-y})u(s,x)u(s,y)dydxds,
\end{align}
where $D_{\varepsilon}=[\varepsilon,\gamma_2(\varepsilon))\times(\gamma_1(x),\varepsilon)$.
\end{remark}

\section{On entropy and entropy dissipation.}
\label{entropy}
\setcounter{equation}{0}
\setcounter{theorem}{0}
Although the entropy  $H$ given by  (\ref{S2EH}) is well defined for the weak solutions $u$  of (\ref{P}), it is not known if $t\mapsto H(u(t))$ is monotone and  may still be used as Lyapunov function to study  the long time behavior of these solutons  (we recall that the dissipation of entropy is not defined in general). But, if $u$ is a weak solution of (\ref{P}) with initial data $u _{ in }$ given by Theorem \ref{MT1}, if $\{u_n\} _{ n\in \NN }$ is the sequence  given by Proposition \ref{theoremESMI} and  if $D^{(n)}$ is the functional defined  in (\ref{S2ED}), the same calculations as in  Section 2 and Section 6  of \cite{ESMI} yield
\begin{align*}
\int_T^{\infty}D^{(n)}(u_n(t))dt\le H(U_M)+C_1\int  _{ [0, \infty) }(1+x)u _{ in }(x)dx,\,\,\forall n\in \NN
\end{align*}
for some  constant $C_1>0$, where $U_M$ is the unique equilibrium with the same mass than  $u _{ in }$, $M=M_0(u_{in})$.
Since the sequence of functions $\{b_n\}_{n\in \NN}$ is increasing,
\begin{align*}
\int_T^{\infty}D^{(m)}(u_n(t))dt \le H(U_M)+C_1\int  _{ [0, \infty) }(1+x)u _{ in }(x)dx,\,\forall n>m.
\end{align*}
Therefore, by the weak lower semi continuity of the function $D^{(m)}$ (cf. Theorem 4.6 in \cite{ESMI}), and the weak convergence of $u_n$ to $u$:
\begin{equation*}
\int_T^{\infty}D^{(m)}(u(t))dt\le  H(U_M)+C_1\int  _{ [0, \infty) }(1+x)u _{ in }(x)dx,\,\,\,\forall m\in \NN
\end{equation*}
It  follows that, for any sequence $t_n\to \infty$, and $\left\{U(t)\right\} _{ t>0 }\subset  \mathscr M_+^1([0, \infty))$ such that $u(t_n+t)\to U(t)$ in  $\mathscr M_+^1([0, \infty))$ for $a. e. \, t>0$ we have   $\lim _{ n\to \infty } D^{(m)}(u(t_n+t))=0$ and then, by weak lower semi continuity, $D^{(m)}U(t)=0$ for all $m\in \NN$ and $a. e.\, t>0$. However, it is  only possible to obtain  a partial  characterization of the measures  $U\in  \mathscr{M}_+([0,\infty))$ with total mass $M$ and such that $D^{(m)}(U)=0$ for all $m\in\NN$. 
\begin{proposition}
\label{S4Dissip}
A measure  $U\in  \mathscr{M}_+([0,\infty))$  with total mass $M>0$, satisfies  $D^{(m)}(U)=0$ for all $m\in \NN$, if and only if, there exists $\mu \le0$ and $\alpha \ge 0$ such that $U=g_\mu +\alpha \delta _0$ and 
$\int_0^\infty  g _{ \mu  }(x)dx+\alpha =M$, where
\begin{equation}
\label{gmu}
g_{\mu}(x)=\frac{x^2}{e^{x-\mu}-1},\qquad x>0.
\end{equation}
\end{proposition}
\begin{proof} It is straightforward to check that if  $U=g_\mu +\alpha \delta _0$ for some $\mu \le 0$ and $\alpha \ge 0$, such that $\int_0^{\infty} g _{ \mu  }(x)dx+\alpha =M$, then $D^{(m)}(U)=0$.
On the other hand, if
$U=g+G$ is the Lebesgue decomposition of $U$ and $D^{(m)}(U)=0$, then $D^{(m)}_1(g)=D^{(m)}_2(g,G)=D^{(m)}_3(G)=0$. From
$D^{(m)}_1(g)=0$ it follows that, for a.e. $(x,y)\in[0,\infty)^2$,
\begin{align}
\label{D1=0}
b_m(x, y)j\Big(g'(x^2+g)e^{-x},g(y^2+g')e^{-y}\Big)=0.
\end{align}
Since $b_m(x, y)>0$ for $(x, y)\in  \Gamma_{\varepsilon, m}$ for all $\varepsilon >0$ and all $m\in\NN$, where
$$\Gamma_{\varepsilon, m}=\Big\{(x,y)\in\Gamma: d((x,y), \partial\Gamma)>\varepsilon,\,\,(x, y)\in \left(\frac {1} {m}, m\right)\times \left(\frac {1} {m}, m\right)\Big\},$$
we deduce  from (\ref{D1=0}),
\begin{align*}
\frac{g(x)e^{x}}{x^2+g(x)}=\frac{g(y)e^{y}}{y^2+g(y)}\quad a.e.\; (x,y)\in\Gamma_{\varepsilon, m },
\end{align*}
and  both terms must then be equal to a nonnegative constant, say $\gamma$. If $\gamma=0$, then 
$g=0$ for a.e. $x>\varepsilon$. If $\gamma>0$, then $\gamma=e^{\mu}$ for some $\mu \in \RR$ and $g=g_{\mu}$
for a.e. $x>\varepsilon$. Letting $\varepsilon\to 0$ we obtain that either $g=0$ or
$g=g_{\mu}$ a.e. in $(0,\infty)$ and, since $g\geq 0$, then $\mu \le 0$.

From $D^{(m)}_3(G)=0$ for all $m\in\NN$,
we obtain that $j(e^{-x},e^{-y})=(e^{-x}-e^{-y})(x-y)=0$ for $G\times G$ a.e. $(x,y)\in\Gamma_{\varepsilon, m}$. Letting $\varepsilon\to 0$, we deduce that 
\begin{equation}
\label{Gequi}
G=\sum_{i=0}^{\infty}\alpha_i\delta_{x_i},
\end{equation}
for some $\alpha_i\geq 0$, $x_i\geq 0$ with $b_m(x_i,x_j)=0$ for all $i\neq j$, and all $m\in\NN$.

From $D^{(m)}_2(g, G)=0$, $g=g_\mu $ and $G$ as in (\ref{Gequi}), we deduce that, for all $m\in \NN$,
\begin{align*}
D^{(m)}_2(g,G)=\sum_{i=0}^{\infty}\alpha_i(x_i-\mu)\big(e^{-\mu}-e^{-x_i}\big) \int_0^\infty b_m(x, x_i)g_\mu (x)dx=0,
\end{align*}
and therefore, each of the terms in the sum above is zero. If $\alpha_i>0$ and $x_i>0$ for some $i\in\NN$, it then follows that 
$\mu=x_i$, which is a contradiction since $\mu\leq 0$. Hence $G=\alpha\delta_0$ for some $\alpha\geq 0$.
\end{proof}

\begin{remark}
The measure $U$ in the statement of Proposition \ref{S4Dissip} is not uniquely determined because, since $b_m(x, 0)=0$ for all $x>0$, 
it is  possible to have  $\mu < 0$ and $\alpha >0$. 
\end{remark}

\section{A simplified equation.}
\setcounter{equation}{0}
\setcounter{theorem}{0}
There may be several reasons to consider the following simplified version of equation (\ref{S1EN0}), (\ref{S1EN1}),
\begin{align}
&\frac{\partial u}{\partial t}(t, x)=u(t, x)\int_0^\infty R(x, y)u(t, y)dy, \label{S2EG}\\
&R(x, y)=b(x, y)(e^{-x}-e^{-y}). \label{S2EG2}
\end{align}
Although the integral collision operator in (\ref{S2EG}) only contains the nonlinear terms of the integral collision operator in (\ref{S1EN0}), it  may supposed to be the dominant term when $u$ is large.  This was the underlying idea in  
\cite{1972ZhETF..62.1392Z} and \cite{1972JETP...35...81Z}, when such approximation was suggested. 
 Let us also recall that, as shown in Section \ref{betascaling} in  Appendix \ref{deduction}, if the variables $k, t$ and  $f$ are suitably scaled  with the parameter $\beta $ to obtain the new variables $x, \tau $ and $u$ (cf. (\ref{S2E2XY}) and (\ref{S2EXYu})),  the equation (\ref{S1EN0}), (\ref{S1EN1}) 
 yields  equation  (\ref{EscEu}), where the  dependence on the parameter $\beta >0$ has been kept. Then, the reduced equation (\ref{S2EG}) appears as the lower order approximation as  $\beta \to \infty$. 
No rigorous result is known about the validity of such an approximation . In any case, it may be expected from  
 (\ref{S2E2XY}), to break down  for $t\gtrsim \beta ^{3}$ and $x\gtrsim \beta$.

Due to its  simpler form, the study of  (\ref{S2EG}) is slightly easier. The existence of  solutions $u\in C([0, \infty), L^1([0, \infty)))$, that do not form a Dirac mass at the origin in finite time, is proved (cf. Section \ref{S7regular}) and it is also possible to describe the long time behaviour of the solutions. Both questions remain open for the  equation (\ref{S1EN0}). 
\subsection{Existence and properties of weak solutions.} 

In this Section we prove the following result on the existence of weak solutions of the equation (\ref{S2EG}), (\ref{S2EG2}).

\begin{theorem}
\label{S8Th1}
For any initial data $u_0\in\mathscr{M}_+([0,\infty))$ satisfying 
\begin{align}
\label{Zcondition}
X_{\eta}(u_0)<\infty\quad\text{for some}\quad\eta>\frac{1-\theta}{2},
\end{align}
there exists $u\in C([0,\infty),\mathscr{M}_+([0,\infty)))$ such  that: 
\begin{align}
\label{ws1R}
&(i)\,\, \forall  \varphi \in C_b([0, \infty)),\,\,\int_{ [0, \infty) }u(\cdot , x)\varphi (x)dx \in C( [0, \infty);\RR) \\
&\hskip 2cm \text{and}\quad \int_{ [0, \infty) }u(0 , x)\varphi (x)dx=  \int\limits  _{ [0, \infty) }u_0(x)\varphi (x)dx, \nonumber\\
&(ii)\,\, \forall \varphi \in C^1_b([0, \infty)),\, \varphi '(0)=0,\nonumber \\
\nonumber
&\,\,\int_{ [0, \infty) }u(\cdot , x)\varphi (x)dx \in W_{loc}^{1, \infty}([0, \infty); \RR),\,\hbox{and for}\,\,a. e. \, t>0,\\
\label{WSZ}
&\,\,\frac {d} {dt}\int\limits_{ [0, \infty) }u(t , x)\varphi (x)dx=\frac{1}{2}\iint\limits_{[0, \infty)^2}R(x, y)u(t,x)u(t,y)(\varphi(x)-\varphi(y))dydx.
\end{align}
(We will say that $u$ is a weak solution of  (\ref{S2EG}) with initial data $u_0$). The solution also satisfies,
\begin{align}
\label{ZmassR}
&M_0(u(t))=M_0(u_0)\quad\forall t> 0,\\
\label{Zmoment}
&X_{\eta}(u(t))\leq X_{\eta}(u_0)\quad\forall t>0.
\end{align}
\end{theorem}
This result is similar to Theorem \ref{MT1}  for the equation  (\ref{P})--(\ref{Bbound}), and its proof uses similar arguments.  The main difference is that  Theorem  3 in \cite{ESMI}  can not be used to obtain approximate solutions,  and this must  be done using a classical truncation argument. Let us then consider the following auxiliary problem:
\begin{align}
\label{Zn}
&\frac{\partial u_n}{\partial t}(t,x)=u_n(t,x)\int_0^{\infty}R_n(x,y)u_n(t,y)dy,\\
\label{Zn0}
&u_n(0,x)=u_{in}(x)\\
&R_n(x, y)=b_n(x, y)(e^{-x}-e^{-y})
\end{align} 
where $b_n$ is defined in (\ref{bn}).
\begin{proposition}
\label{Znexistence}
For every $n\in\NN$ and for every nonnegative initial data $u_{in}\in L^1([0,\infty))$, there exists a nonnegative function $u_n\in C([0,\infty),L^1([0,\infty)))\cap C^1((0,\infty),L^1([0,\infty)))$ that satisfies (\ref{Zn}) and (\ref{Zn0}) in $C((0, \infty), L^1([0,\infty)))$ and $L^1([0,\infty))$ respectively
and,
\begin{align}
\label{Znmass}
M_0(u_n(t))=M_0(u_{in})\quad\forall t\geq 0.
\end{align}
Moreover, for all $\varphi $,  defined on $[0, \infty)$, measurable, non negative and non decreasing  function,
\begin{align}
\label{S5WP1}
\int  _{ [0, \infty) }u_n(t, x)\varphi (x)dx\le \int  _{ [0, \infty) }u _{ in }(x)\varphi (x)dx,\,\,\,\forall n\in \NN,\;\forall  t>0.
\end{align}
\end{proposition}

\begin{proof} The proof uses a simple Banach fixed point argument. 
For any nonnegative $f\in C([0, \infty ), L^{1 }([0,\infty)))$ we consider the solution $u$ to the problem
\begin{equation*}
\left\{
\begin{array}{l}
\displaystyle\frac{\partial u}{\partial t}(t,x)=u(t,x)\int_0^\infty R_n(x,y)f(t,y)dy\quad x>0,\;t>0,\\
u(0,x)=u _{ in }(x),\quad x>0,
\end{array}
\right.
\end{equation*}
given by:
\begin{equation*}
\label{S8OPA}
A_n(f)\equiv u(t,x)=u _{ in }(x)e^{\int_0^{t}\int_0^\infty R_n(x,y)f(s,y)dyds}.
\end{equation*}
Our goal is then to prove first that $A_n$ is a contraction on  $\mathcal{X} _{ \rho , T }$ for some  $\rho >0$ and $T>0$ where,
\begin{align*}
\mathcal{X}_{\rho, T} =\left\{f\in C([0, T); L^1([0,\infty))); \sup _{ 0\le t \le T }\|f(t)\|_1\le \rho  \right\}.
\end{align*}
For all $T>0$, $t\in [0, T)$ and $f\in \mathcal{X} _{ \rho , T }$,
\begin{align}
\label{S8EPF1}
\|A_n(f)(t)\|_1\le 
&\|u _{ in }\|_1e^{T\rho \|R_n\|_{\infty}  };
\end{align} 
and for all  $t_1$, $t_2$ such that $0\leq t_1\leq t_2\leq T$:
\begin{align*}
|A_n(f)&(t_1,x)-A_n(f)(t_2,x)|=\\
&=u _{ in }(x)\left|e^{\int_0^{t_1}\int_0^\infty R_n(x, y)f(s, y)dyds}-e^{\int_0^{t_2}\int_0^\infty R_n(x, y)f(s, y)dyds} \right|\\
&\leq u _{ in }(x) \left|\int_{t_1}^{t_2}\int_0^\infty R_n(x, y)f(s, y)dyds\right| \times\\
&\hskip 0.4cm\times e^{\theta \int_0^{t_1}\int_0^\infty R_n(x, y)f(s, y)dyds+(1-\theta )\int_0^{t_2}\int_0^\infty R_n(x, y)f(s, y)dyds}\\
&\leq u _{ in }(x)\rho \|R_n\|_\infty   |t_1-t_2|e^{T\rho \|R_n\|_\infty}.
\end{align*} 
It then follows that
\begin{align}
\label{S8EPF2}
\|A_n(f)&(t_1)-A_n(f)(t_2)\|_1 \leq \rho \|u _{ in }\|_1   \|R_n\|_\infty  e^{\rho T\|R_n\|_{\infty} } |t_1-t_2|.
\end{align}
Let now $f$ and $g$ be in $\mathcal{X} _{ \rho , T }$ and denote $v=A_n(g)$ and $u=A_n(f)$. Arguing as before,
\begin{equation}
\label{S8EPF3}
\|u(t)-v(t)\|_1\leq  \|u _{ in }\|_1\|R_n\|_\infty \|f-g\|_{C([0,T],\,L^1([0,\infty)))}Te^{\rho T\|R_n\|_\infty}.
\end{equation}
By (\ref{S8EPF1})--(\ref{S8EPF3}), if 
\begin{align*}
&\|u _{ in }\|_1e^{T\rho \|R_n\|_{\infty}  }<\rho,\,\,\hbox{and}\\
& \|u _{ in }\|_1\|R_n\|_\infty Te^{\rho T\|R_n\|_\infty}<1,
\end{align*}
then $A_n$  is a contraction on $C([0, T], L^{1 }([0,\infty)))$, and has  a fixed point $u_n$
that satisfies
\bear
\label{S2EimplicitR}
u_n(x, t)=u _{ in }(x)e^{\int_0^{t}\int_0^\infty R_n(x, y)u_n(s, y)dyds}.
\eear
This solution may then be extended to $C([0, T _{ \max}), L^{1 }([0,\infty)))$. It immediately follows from  (\ref{S2EimplicitR}) that $u_n\ge 0$. 

Moreover, since $u_n\in  C([0, T _{ \max }),  L^{1 }([0,\infty)))$  and $R_n$ is bounded, 
we deduce from (\ref{S2EimplicitR})  that $u_n\in C^1([0, T _{ \max }), L^{1 }([0,\infty)))$ and,  for every $t\in (0, T _{ \max })$, the equation  (\ref{Zn}) is satisfied  in $L^1([0,\infty))$.
For all $T<T_{\max}$ and all $t\in[0,T]$,
\begin{align*}
&\left|u _{ in }(\cdot)\frac{d}{dt}\left(e^{\int_0^t\int_0^{\infty}R_n(\cdot,y)u_n(s,y)dy}\right)\right| \leq u _{ in }(\cdot)\|R_n\|_{\infty}\times  \\
&\quad \times \|u_n\|_{C([0,T],L^1([0,\infty)))}e^{T\|R_n\|_{\infty}\|u_n\|_{C([0,T],L^1([0,\infty)))}}\in L^1([0,\infty)),
\end{align*}
then if we multiply (\ref{S2EimplicitR}) by any $\varphi\in L^{\infty}([0,\infty))$, we deduce that for all $t<T_{\max}$,
\begin{align*}
\frac{d}{dt}\int_0^{\infty}u_n(t,x)\varphi(x)dx=\iint\limits _{ (0, \infty)^2 }R_n(x,y)u_n(t,x)u_n(t,y)\varphi(x)dydx.
\end{align*}
Recalling the definition of $R_n$, then by the symmetry of $b_n$ and Fubini's theorem,
\begin{equation}
\frac{d}{dt}\int_0^{\infty}\!\!\!u_n(t,x)\varphi(x)dx=\iint\limits _{ (0, \infty)^2 }k_{\psi ,n}(x,y)u_n(s,x)u_n(s,y)dydx,\label{S5EWFn}
\end{equation}
i.e. $u_n$ is a weak solution of  (\ref{Zn}), (\ref{Zn0}) for $t\in [0, T)$.
If we chose $\varphi=1$ we deduce that  (\ref{Znmass}) holds for all $t<T_{\max}$. Then, by a classical argument, $T_{\max}=\infty$.

In order to prove (\ref{S5WP1}) let $\psi  $ be non negative and measurable function such that  $\int_0^{\infty}u_0(x)\psi (x)dx<\infty$, and consider $\{\psi_k\}_{k\in\NN}$ the sequence of simple functions  that converges monotonically to $\psi $ as $k\to\infty$. Since $\psi _k\in L^{\infty}([0,\infty))$, then
(\ref{S5EWFn}) holds with  $\varphi =\psi _k$ for all $k$, and by Lebesgue's and monotone convergence Theorems,
\begin{align*}
\int_0^{\infty}u_n(t,x)\psi (x)dx&=\int_0^{\infty}u _{ in }(x)\psi (x)dx\\
&+\int_0^t \int_0^{\infty}\!\!\!\int_0^{\infty}k_{\psi ,n}(x,y)u_n(s,x)u_n(s,y)dydxds.
\end{align*}
Using that $u_n\in C([0,\infty),L^1([0,\infty)))$, (\ref{Znmass}) and 
\begin{align*}
\iint\limits _{ (0, \infty)^2 }&k_{\psi ,n}(x,y)\big|u_n(t_1,x)u_n(t_1,y)-u_n(t_2,x)u_n(t_2,y)\big|dydx\\
&\leq 2\|k_{\psi ,n}\|_{\infty}M_0(u _{ in })\|u_n(t_1)-u_n(t_2)\|_1,
\end{align*}
so that $t\mapsto\int_0^{\infty}\!\!\!\int_0^{\infty}k_{\psi ,n}(x,y)u_n(s,x)u_n(s,y)dydx$ is continuous,
it follows by the fundamental theorem of calculus that (\ref{S5EWFn}) holds for $\psi $.
If, in addition,  $\varphi$ is nondecreasing, then $(e^{-x}-e^{-y})(\varphi(x)-\varphi(y))\leq 0$ for all $(x,y)\in[0,\infty)^2$,
and then (\ref{S5WP1}) follows.
\end{proof}

\begin{proof}[\bfseries\upshape{Proof of Theorem \ref{S8Th1} }]
Consider first a initial data $u_0\in L^1([0,\infty))$. Let $\{u_n\}_{n\in\NN}$ be the sequence of solutions to (\ref{Zn})--(\ref{Zn0}) constructed in Proposition \ref{Znexistence} for $n\in\NN$. As in the proof of Theorem \ref{MT1}, the result follows from the precompactness  (given by the conservation of $M_0(u)$) and the  equicontinuity of $\{u_n\} _{ n\in \NN }$.  These properties follow  as in the proof of Proposition \ref{precompactness} and Proposition \ref{equicontinuity} respectively. The existence of the solution $u$ follows  using the same arguments as in  Corollary \ref{cor:u} and the end of the Proof of Theorem \ref{MT1}. 

Property (\ref{Zmoment}) for $u$ follows from (\ref{S5WP1}),  the lower semicontinuity of the non negative function $e^{\eta x}$, and the weak convergence to $u$  of  $u_n$.

For a general initial data $u_0\in\mathscr{M}_+([0,\infty))$, by Corollary 8.6 in \cite{EC} there exists a sequence $\{u_{0,n}\}_{n\in\NN}\subset L^1([0,\infty))$ such that
\begin{equation}
\label{INNN}
\lim_{n\to\infty}\int_0^{\infty}\varphi(x)u_{0,n}(x)dx=\!\!\!\int\limits_{[0,\infty)}\varphi(x)u_0(x)dx,\,\forall\varphi\in C_b([0,\infty)).
\end{equation}
Since $u _{ 0, n }\in L^1([0,\infty))$, using the previous step there exists a weak solution  $u_n$ that satisfies (\ref{ws1R})--(\ref{Zmoment}).  By (\ref{ZmassR}) and (\ref{Zmoment}), the sequence $\{u_n\} _{ n\in \NN }$ is precompact in $C([0, \infty), \mathscr M_+ ([0,\infty)))$. Arguing as in  Proposition \ref{equicontinuity}, we deduce that it is also equicontinuous. Therefore, using the same arguments as in the end of the Proof of Theorem \ref{MT1}, we deduce the existence of  a subsequence, still denoted $\{u_n\}_{n\in\NN}$, and  a weak solution of 
(\ref{S2EG}), $u\in C([0, \infty),\mathscr M_+ ([0,\infty)))$, satisfying (\ref{ws1R})--(\ref{Zmoment}).  

The property  (\ref{Zmoment}) is obtained using first in the weak formulation (\ref{WSZ})  a sequence of monotone non decreasing test functions $\{\varphi _k\} _{ k\in \NN }\subset C^1_b([0, \infty)$ such that $\varphi '_k(0)=0$ and $\varphi _k(x)\to e^{\eta x}$ for all $x\ge 0$, to  obtain:
\begin{equation}
\int  _{ [0, \infty) }u(t, x)\varphi _k(x)dx\le  \int  _{ [0, \infty) }u _{ 0}(x)\varphi _k(x)dx,
\end{equation}
and then pass to the limit as $k\to \infty$.
\end{proof}

\begin{remark}
In Theorem \ref{MT1}, the initial data is required to satisfy $X_{\eta}(u_0)<\infty$ for some $\eta\in\left(\frac{1-\theta}{2},\frac{1}{2}\right)$. On the one hand, the condition $\eta>\frac{1-\theta}{2}$ is sufficient in order to have boundedness of the operators 
$K_{\varphi}(u,u)$ and $L_{\varphi}(u)$. On the other hand, the condition $\eta<1/2$ comes from the estimate (\ref{exp}). In Theorem \ref{S8Th1}, however, that last condition is not needed, thanks to the estimate (\ref{S5WP1}).
\end{remark}

We show now that the support of $u(t)$ is constant in time.
\begin{proposition}
\label{comparison}
Let $u$ be a weak solution of (\ref{S2EG}) constructed in Theorem \ref{S8Th1} for an initial data $u_0\in\mathscr{M}_+([0,\infty))$ satisfying (\ref{Zcondition}). The following statements hold:
\begin{enumerate}[(i)]
\item  For all $r>0$, $t_0$ and $t$ with $0\leq t_0\leq t$, and $\varphi\in C^1_c((0,\infty))$ nonnegative such that  $\supp(\varphi)\subset[r,L]$ for some $L>r$, 
\begin{align}
\label{supp50}
&\int_{[0,\infty)}\varphi(x)u(t,x)dx\geq e^{-(t-t_0)C_1}\int_{[0,\infty)}\varphi(x)u(t_0,x)dx,\\
\label{supp51}
&\int_{[0,\infty)}\varphi(x)u(t,x)dx\leq e^{(t-t_0)C_2}\int_{[0,\infty)}\varphi(x)u(t_0,x)dx,
\end{align}
where
\begin{align}
C_1=\frac{C_*\rho_*M_0(u_0)}{\sqrt{\theta(1+\theta)}}\frac{e^{\frac{(1-\theta)L}{2}}}{r^{3/2}},\quad
C_2=\frac{C_*\rho_*M_0(u_0)}{\sqrt{2}}\frac{e^{\frac{(1-\theta)L}{2\theta}}}{r^{3/2}}.
\end{align}
\item For all $r>0$, $t_0$ and $t$ with $0\leq t_0\leq t$,
\begin{gather}
\label{comparison2}
\hskip -0.5cm \int\limits_{[0,r)}u(t_0,x)dx\leq \int\limits_{[0,r)}u(t,x)dx\leq e^{(t-t_0)C_{r}}\!\!\int\limits_{[0,r)}u(t_0,x)dx,
\end{gather}
\begin{flalign}
\label{comparison constant}
\text{where}&& C_{r}=\frac{C_*\rho_*M_0(u_0)}{\sqrt{\theta(1+\theta)}}\frac{e^{\frac{(1-\theta)r}{2\theta}}}{r^{3/2}}.&&
\end{flalign}
\item $\supp (u(t))=\supp (u_0)$ for all $t>0$.
\end{enumerate}
\end{proposition}

\begin{proof}
Proof of (i). Since there are no integrability issues near the origin because $\supp(\varphi)\subset[r,L]$, then by Fubini's theorem 
\begin{align*}
\frac{1}{2}\iint_{[0,\infty)^2}&R(x,y)(\varphi(x)-\varphi(y))u(t,x)u(t,y)dydx\\
&=\int_{[0,\infty)}\varphi(x)u(t,x)\int_{[0,\infty)}R(x,y)u(t,y)dydx.
\end{align*}
Let us prove the lower bound (\ref{supp50}). Using (\ref{Bsupport})--(\ref{Bbound}), for all $(x,y)\in\Gamma$, $y\leq x$,
\begin{align}
\label{R11}
|R(x,y)|\leq\frac{C_*e^{\frac{x-y}{2}}(x-y)}{xy(x+y)}
\leq C^*\frac{e^{\frac{(1-\theta)x}{2}}}{x^{3/2}},\quad C^*=\frac{C_*\rho_*}{\sqrt{\theta(1+\theta)}},
\end{align}
and taking into account the support of $\varphi$, we deduce that 
\begin{align*}
\int_{[0,\infty)}&\varphi(x)u(t,x)\int_{[0,\infty)}R(x,y)u(t,y)dydx\\
&\geq \int_{r}^{L} \varphi(x)u(t,x)\int_0^{x}R(x,y)u(t,y)dydx\\
&\geq-\frac{C^*e^{\frac{(1-\theta)L}{2}}}{r^{3/2}}\int_{r}^{L}\varphi(x)u(t,x)\int_0^x u(t,y)dydx\\
&\geq -C_1\int_{r}^{L}\varphi(x)u(t,x)dx,
\end{align*}
and then, from the weak formulation, we obtain that for all $t>0$,
\begin{align*}
\frac{d}{dt}\int_{r}^{L}\varphi(x)u(t,x)dx\geq -C_1\int_{r}^{L}\varphi(x)u(t,x)dx,
\end{align*}
and (\ref{supp50}) follows by Gronwall's Lemma.

We now prove the upper bound (\ref{supp51}) by similar arguments. Since $R(x,y)\leq 0$ for $y\leq x$, then
\begin{align*}
\int_{[0,\infty)}&\varphi(x)u(t,x)\int_{[0,\infty)}R(x,y)u(t,y)dydx\\
&\leq\int_{r}^{L} \varphi(x)u(t,x)\int_x^{\infty}R(x,y)u(t,y)dydx,
\end{align*}
and since for all $(x,y)\in\Gamma$, $x\leq y$,
\begin{align*}
R(x,y)&\leq \frac{C_*e^{\frac{y-x}{2}}(y-x)}{xy(x+y)}
\leq C'\frac{e^{\frac{(1-\theta)x}{2\theta}}}{x^{3/2}},\qquad C'=\frac{C_*\rho_*}{\sqrt{2}},
\end{align*}
we deduce from the weak formulation that for all $t>0$,
\begin{align*}
\frac{d}{dt}\int_{r}^{L}\varphi(x)u(t,x)dx&
\leq \frac{C'e^{\frac{(1-\theta)L}{2\theta}}}{r^{3/2}}\int_{r}^{L}\varphi(x)u(t,x)\int_x^{\infty}u(t,y)dydx\\
&\leq C_2\int_{r}^{L}\varphi(x)u(t,x)dx,
\end{align*}
and then (\ref{supp51}) follows by Gronwall's Lemma.

Proof of (ii). We first prove the lower bound in (\ref{comparison2}). Given $r>0$, let $0<r_*<r$, and $\varphi\in C^1_c([0,\infty))$ be nonnegative, nonincreasing, and such that $\varphi(x)=1$ for all $x\in[0,r_*]$  and $\varphi(x)=0$ for all 
$x\geq r$. Since $(e^{-x}-e^{-y})(\varphi(x)-\varphi(y))\geq 0$ for all $0\leq y\leq x$, it follows from the weak formulation
\begin{gather*}
\frac{d}{dt}\int_{[0,r)}\varphi(x)u(t,x)dx\geq 0\qquad\forall t>0,
\shortintertext{hence}
\int_{[0,r)}\varphi(x)u(t,x)dx\geq\int_{[0,r)}\varphi(x)u(t_0,x)dx\qquad\forall t\geq t_0\geq 0,
\end{gather*}
and then the lower bound in (\ref{comparison2}) follows by taking the supremum over all $\varphi$ as above, i.e., letting $r_*\to r$.

Let us prove now the upper bound in (\ref{comparison2}). Given $r>0$, let $r_*$ and $\varphi$ be as before. Keeping only the positive terms in the weak formulation and taking $\Gamma$ into account, we deduce
\begin{align*}
\frac{d}{dt}\int_{[0,r)}\varphi(x)u(t,x)dx
\leq\int_{r_*}^{\frac{r}{\theta}}\int_{\theta x}^{\min\{x,r\}}
|R(x,y)|\varphi(y) u(t,x)u(t,y)dydx,
\end{align*}
and by (\ref{R11}) we obtain
\begin{align*}
\frac{d}{dt}\int_{[0,r)}\varphi(x)u(t,x)dx
&\leq\frac{C^*e^{\frac{(1-\theta)r}{2\theta}}}{r_*^{3/2}}
\int_{r_*}^{\frac{r}{\theta}}u(t,x)\int_{\theta x}^{\min\{x,r\}}\varphi(y)u(t,y)dydx\\
&\leq C_{r_*}\int_{[0,r)}\varphi(y)u(t,y)dy,
\end{align*}
where $C_{r_*}=\frac{C^*e^{\frac{(1-\theta)r}{2\theta}}}{r_*^{3/2}}M_0(u_0),$
and then, from Gronwall's Lemma, 
\begin{align*}
\int_{[0,r)}\varphi(x)u(t,x)dx\leq e^{(t-t_0)C_{r_*}}\int_{[0,r)}\varphi(x)u(t_0,x)dx\qquad\forall t\geq t_0\geq 0.
\end{align*}
The upper bound in (\ref{comparison2}) then follows by letting $r_*$ tend to $r$.

Proof of (iii). We recall the following characterization of the support of a Radon measure $\mu$ (see \cite{FO}, Chapter 7): $x\in\supp(\mu)$ if and only if 
$\int_{[0,\infty)}\varphi d\mu>0$ 
for all $\varphi\in C_c([0,\infty))$ with $0\leq \varphi\leq 1$ such that $\varphi(x)>0$.
Then, from (\ref{supp50}) and (\ref{supp51}) for $t_0=0$, we deduce that 
\begin{align*}
(0,\infty)\cap\supp (u_0)=(0,\infty)\cap\supp (u(t))\qquad\forall t>0,
\end{align*}
and from (\ref{comparison2}) for $t_0=0$, we deduce that for all $t>0$,
\begin{align*}
0\in\supp(u_0)\quad\text{if and only if}\quad0 \in\supp(u(t)),
\end{align*}
which completes the proof.
\end{proof}

The queues of the weak solutions are  decreasing in time, as proved  in the following Proposition. 
\begin{proposition}
\label{prop.64}
Let $u$ be the weak solution of (\ref{S2EG}) constructed in Theorem \ref{S8Th1} for an initial data $u_0\in\mathscr{M}_+([0,\infty))$ satisfying (\ref{Zcondition}). Then
\begin{enumerate}[(i)]
\item
For all $r\geq 0$, the map $t\mapsto\int_{[r,\infty)}u(t,x)dx$ is nonincreasing on $[0,\infty)$.
\item
For all $r>0$, if
\begin{gather}
\exists x_0\in[r,\gamma_2(r))\cap\supp (u_0),\;\exists y_0\in(\gamma_1(r),r)\cap\supp (u_0),\qquad \nonumber\\
\label{wierd}
\text{such that}\quad B(x_0,y_0)>0,
\end{gather}

then the map $t\mapsto\int_{[r,\infty)}u(t,x)dx$ is strictly decreasing on $[0,\infty)$.
\end{enumerate}
\end{proposition}

\begin{remark}
Condition (\ref{wierd}) holds, for instance, if $r$ is an interior point of the support of $u_0$.
\end{remark}

\begin{proof}
Proof of (i). For $r=0$, the result follows from the conservation of mass (\ref{ZmassR}). For $r>0$, let $\varepsilon\in(0,r)$ and 
$\varphi_{\varepsilon}\in C^1_b([0,\infty))$ be an increasing function such that $\varphi_{\varepsilon}(x)=1$ for all $x\geq r$, $\varphi_{\varepsilon}(x)=0$ for all $x\in[0,r-\varepsilon]$. Using the monotonicity of $\varphi_{\varepsilon}$, we deduce from the weak formulation (\ref{WSZ}) that for all $t\geq 0$,
\begin{align*}
\frac{d}{dt}\int_{[0,\infty)}\varphi_{\varepsilon}(x)u(t,x)dx\leq 0,
\end{align*}
and then the map $t\mapsto\int_{[0,\infty)}\varphi_{\varepsilon}(x)u(t,x)dx$ is nonincreasing. The result then follows by letting $\varepsilon\to 0$.

Proof of (ii).
Since (\ref{WSZ}) is invariant under time translations, it suffices to prove that for all $r>0$,
\begin{align*}
\int_{[r,\infty)}u(t,x)dx<\int_{[r,\infty)}u_0(x)dx\qquad\forall t>0,
\end{align*}
provided (\ref{wierd}) holds. To this end, consider $\varphi_{\varepsilon}$ as in part (i). By (\ref{WSZ})
\begin{align*}
\int_{[0,\infty)}\varphi_{\varepsilon}(x)u(t,x)dx&=\int_{[0,\infty)}\varphi_{\varepsilon}(x)u_0(x)dx\\
&+\int_0^t \int_0^{\infty}\int_0^x k_{\varphi_{\varepsilon}}(x,y)u(s,x)u(s,y)dydxds.
\end{align*}
Then, since
$\lim_{\varepsilon\to 0}k_{\varphi_{\varepsilon}}(x,y)=k_{\varphi}(x,y)$ for all $(x,y)\in[0,\infty)^2$, where $\varphi(x)=\mathds{1}_{[r,\infty)}(x)$, and for all $\varepsilon$ small enough, 
\begin{align*}
\int_0^{\infty}\int_0^x &|k_{\varphi_{\varepsilon}}(x,y)|u(s,x)u(s,y)dydx\\
&=\int_{r-\varepsilon}^{\infty}\int_{\theta x}^x |k_{\varphi_{\varepsilon}}(x,y)|u(s,x)u(s,y)dydx\\
&\leq 2C_*\rho_*\int_{r-\varepsilon}^{\infty}\int_{\theta x}^x\frac{e^{\frac{x-y}{2}}}{\sqrt{xy(x+y)}}u(s,x)u(s,y)dydx\\
&\leq \frac{2C_*\rho_*}{\sqrt{\theta(1+\theta)}(r-\varepsilon)^{3/2}}\int_{r-\varepsilon}^{\infty}e^{\frac{(1-\theta)x}{2}}u(s,x)\int_{\theta x}^x u(s,y)dydx\\
&\leq \frac{4C_*\rho_*M_0(u_0)}{\sqrt{\theta(1+\theta)}r^{3/2}}\int_0^{\infty}e^{\frac{(1-\theta)x}{2}}u_0(x)dx<\infty,
\end{align*}
we deduce from dominated convergence Theorem
\begin{align}
\int_{[r,\infty)}&u(t,x)dx=\int_{[r,\infty)}u_0(x)dx\nonumber\\
\label{XX1}
&+\int_0^t\int_{[r,\infty)}\int_{[0,r)} R(x,y)u(s,x)u(s,y)dydxds.
\end{align}
Taking $\Gamma$ into account, we observe that
\begin{align}
&\int_{[r,\infty)}\int_{[0,r)} R(x,y)u(s,x)u(s,y)dydxds\nonumber\\
\label{XX2}
&=\int_{[r,\gamma_2(r))}\int_{(\gamma_1(x),r)} R(x,y)u(s,x)u(s,y)dydxds\leq 0.
\end{align}
The goal is to show that the integral above is, indeed, strictly negative for all $s\in[0,t]$. 
By (\ref{wierd}) and Proposition \ref{comparison}  (iii),  there exists an open rectangle $G=G_1\times G_2$ centered around $(x_0,y_0)$ and contained in 
$\{(x,y)\in[0,\infty)^2:x\in[r,\gamma_2(r)),\;y\in(\gamma_1(x),r)\}$ such that
\begin{align*}
\int_{G_i} u(t,x)dx>0\qquad\forall t>0,\;i=1,2.
\end{align*}
We then obtain
\begin{align*}
&\int_{[r,\gamma_2(r))}\int_{(\gamma_1(x),r)} R(x,y)u(s,x)u(s,y)dydxds\\
&\leq \max_{(x,y)\in G}R(x,y)\int_{G_1} u(t,x)dx\int_{G_2}u(t,y)dy<0,
\end{align*}
and the result then follows from (\ref{XX1}) and (\ref{XX2}).
\end{proof}

\subsection{Global ``regular'' solutions.}
\label{S7regular}
We prove in this Section Theorem \ref{MT3t}, for initial data $u_0$ sufficiently flat around the origin. This condition on $u_0$ is  sufficient to prevent the formation of a Dirac mass in finite time. We do not know if it is necessary. We prove first the following,

\begin{proposition}
\label{MT3BISt}
For all $v_0\in L^1([0, \infty))$, $v_0\ge 0$, satisfying (\ref{flat})
for some $\eta>(1-\theta)/2$,  
there exists  a  nonnegative  global weak solution $v\in C([0, \infty), L^1([0, \infty)))$ of (\ref{S2EG}), (\ref{S2EG2}) such that
\begin{align}
\label{S2Eimplicit}
u(t,x)=u_{0}(x)e^{\int_0^t\int_0^{\infty}R(x,y)u(s,y)dyds}\,\,\,\,\forall t>0,\,a. e.\, x>0,
\end{align}
and also satisfies  $v(0)=v_0$, (\ref{S2Ec602}), (\ref{S2Ec6021}).
\end{proposition}

\begin{proof}[\bfseries\upshape{Proof of Proposition \ref{MT3BISt}}]
The proof 
has  two steps. \\
\textit{Step 1}. We consider first a compactly supported initial data, say $\supp u_0\subset [0,L]$, $L>0$.
We first prove that the operator
\begin{equation}
\label{S8OPA}
A(f)(t,x)=u_0(x)e^{\int_0^{t}\int_0^\infty R(x,y)f(s,y)dyds}
\end{equation} 
is a contraction on  $Y _{ \rho , T }$ for some  $\rho >0$ and $T>0$, where
\begin{gather*}
Y_{\rho, T} =\left\{f\in C([0, T), L^1([0,\infty),\omega dx)):\; \|f\|_T\le \rho  \right\},\\
\|f\|_T=\sup_{0\leq t<T}\int_0^{\infty}\omega(x)|f(t,x)|dx=\sup_{0\leq t<T}\|f(t)\|_{\omega},\\
\omega(x)=(1+x^{-3/2}).
\end{gather*}
Using (\ref{Bsupport})--(\ref{Bbound}), for all $(x,y)\in\Gamma$, $x\leq L$,
\begin{align*}
|R(x,y)|&
\leq\frac{C_*\rho_* e^{\frac{|x-y|}{2}}}{\sqrt{xy(x+y)}}\leq\frac{C_*\rho_*}{\sqrt{\theta(1+\theta)}}\frac{e^{\frac{(1-\theta)x}{2\theta}}}{y^{3/2}}\leq C_L\omega(y),
\end{align*}
where $C_L=\frac{C_*\rho_*e^{\frac{(1-\theta)L}{2\theta}}}{\sqrt{\theta(1+\theta)}}$.
Then, for all nonnegative $f\in Y_{\rho,T}$, $x\in[0,L]$, and $t\in[0,T)$,
\begin{align*}
\int_0^t\int_0^{\infty}|R(x,y)|f(s,y)dyds&\leq C_L\rho T,
\end{align*}
and then
\begin{align}
A(f)(t, x)\le u_0(x)e^{C_L\rho T},\nonumber\\
\|A(f)\|_T\leq  \|u_0\|_{\omega}e^{C_L\rho T}.\label{1KO}
\end{align}
Notice that $\|u_0\|_{\omega}<\infty$ by the hypothesis (\ref{flat}).
Let now $t_1$ and $t_2$ be such that $0\leq t_1\leq t_2<T$. Then, for all $x\in[0,L]$,
\begin{align*}
|A(f)(t_1,x)&-A(f)(t_2,x)|=\\
&=u_0(x)\left|e^{\int_0^{t_1}\int_0^{\infty}R(x,y)f(s,y)dyds}-e^{\int_0^{t_2}\int_0^{\infty}R(x,y)f(s,y)dyds}\right|\\
&\leq u_0(x)e^{C_L\rho T}C_L\rho|t_1-t_2|,
\end{align*}
and therefore
\begin{align*}
\|A(f)(t_1)-A(f)(t_2)\|_{\omega}
&\leq \|u_0\|_{\omega}e^{C_L\rho T}C_L\rho|t_1-t_2|,
\end{align*}
from where it follows that $A\in C([0, T), L^1([0,\infty),\omega dx))$.  
On the other hand, 
if we chose $\rho=2\|u_0\|_{\omega}$ and $T>0$ such that $e^{C_L\rho T}\leq 2$, we deduce from (\ref{1KO}) that $\|A(f)\|_T\leq\rho$, i.e., $A(f)\in Y_{\rho,T}$.

Let now $f$ and $g$ be in $Y _{ \rho , T }$ .
By similar computations as before,
\begin{align*}
\|A(f)-A(g)\|_{T}&\leq \|u_0\|_{\omega}e^{C_L\rho T} C_LT\|f-g\|_T,
\end{align*}
and if $T$ is such that
\begin{align*}
 \|u_0\|_{\omega}e^{C_L\rho T} C_LT<1,
\end{align*}
then $A$ is a contraction on $Y _{ \rho , T }$, and has then a fixed point $u$
that satisfies (\ref{S2Eimplicit}) for all $t\in (0, T)$ and $a. e. \, x>0$.
It then follows in particular that $u\ge 0$. Let us denote
\begin{equation*}
T _{ \max }=\sup \Big\{T>0;  \exists\, \rho >0, \exists  u\in Y _{ \rho , T }\,\,\hbox{satisfying}\,\,  (\ref{S2Eimplicit}),\,\forall t\in [0, T)\Big\}.
\end{equation*}
We claim that if $T _{ \max }<\infty$, then $\lim\sup_{t\to T_{\max}}\|u(t)\|_{\omega}=\infty$. Suppose that $T_{\max}<\infty$ and 
$\limsup_{ t\to T_{\max} }\|u(t)\|_{\omega}=\ell<\infty$, and let $t_n\to T_{\max}$.
For every $n\in\NN$ we define $\rho_n=2\|u(t_n)\|_{\omega}$, 
and the map 
\begin{align*}
A_n(f)(t,x)=u(t_n, x)e^{\int _0^t\int _0^\infty R(x, y)f(s, y)dyds},
\end{align*}
for $f\in C([0, T), L^1([0,\infty),\omega dx))$, $T>0$.
For every $T>0$, $t\in[0,T)$, $x\in[0,L]$, and $f\in Y_{\rho_n,T}$,
\begin{align*}
&A_n(f)(t,x)\leq u(t_n, x)e^{C_L\rho_n T},\\
&\|A_n(f)\|_{T }\leq \|u(t_n)\|_{\omega}e^{C_L\rho_n T},
\end{align*}
and for all $T$ such that $T\leq (\ln 2)/(C_L\rho_n)=T_n$, it follows that $A_n(f)\in Y_{\rho_n,T}$. 
Notice that by hypothesis, $\rho_n\leq 2\ell$ for all $n$, and then
\begin{align*}
T_n\geq \frac{\ln 2}{C_L2\ell}:=\tau_1\quad\forall n\in\NN.
\end{align*}
Now let $f$ and $g$ be in $Y_{\rho_n,T}$ for $T>0$. 
Arguing as before
\begin{align*}
\|A_n(f)-A_n(g)\|_{T}&\leq \|u(t_n)\|_{\omega}e^{C_L2\ell T} C_LT\|f-g\|_T,
\end{align*}
and since
\begin{align*}
u(t_n,x)\leq u_0(x)e^{C_L2\ell T_{\max}},\quad\forall n\in\NN,
\end{align*}
then
\begin{align*}
\|A_n(f)-A_n(g)\|_{T}&\leq \|u_0\|_{\omega}e^{C_L2\ell(T+T_{\max})} C_LT\|f-g\|_T.
\end{align*}
If we chose $\tau_2>0$ such that 
\begin{align*}
 \|u_0\|_{\omega}e^{C_L2\ell(\tau_2+T_{\max})} C_L\tau_2<1,
\end{align*}
and we let $\tau_*=\min\{\tau_1,\tau_2\},$
then $A_n$ is a contraction from $Y_{\rho_n,\tau_*}$ into itself,
and has then a fixed point, say $v_n$. The function $v_n$ satisfies
\bean
v_n(t, x)=u(t_n, x)e^{\int _0^t\int _0^\infty R(x, y)\, v_n(t, y)\, dy},\quad\forall t\in [0, \tau_*),\,\text{a.e. }x\in[0,\infty).
\eean
Therefore, the function $w_n$ defined as
\begin{equation*}
w_n(t,x)=\begin{cases}
u(t, x)&\text{if}\;t\in[0,t_n)\\
v_n(t-t_n, x),&\text{if}\; t\in [t_n, t_n+\tau_*)
\end{cases}
\end{equation*}
satisfies the integral equation:
\bean
w_n(t, x)=u_0(x)e^{\int _0^t\int _0^\infty R(x, y)w_n(s, y)dyds},\quad\forall t\in [0, t_n+\tau_*).
\eean
Since $t_n\to T_{\max}$, then $t_n+\tau_*>T _{ \max }$ for $n$ large enough, and this contradicts the definition of $T _{ \max }$.
We deduce that, either $T _{ \max }=\infty $, and the solution is said to be global,  or 
$\limsup_{t\to T_{\max}}\|u(t)\|_{\omega}= \infty $ and  the solution is said to blow  up in finite time, at $T _{ \max }$.

Since for all $T<T _{ \max }$, $t\in[0,T]$ and $a. e.\,x\in[0,L]$,
\begin{align*}
&\left|\frac{d}{dt}\left(u_0(x)\varphi(x)e^{\int_0^t\int_0^{\infty}R(x,y)u(s,y)dyds}\right)\right|\\
&\leq u_0(x)|\varphi(x)|e^{C_LT\|u\|_T}C_LT\|u\|_T\quad(\text{integrable in }x),
\end{align*}
we may then multiply both sides of the equation (\ref{S2Eimplicit}) by
a function $\varphi \in C_b([0,\infty))$ and integrate on $[0,\infty)$:
\begin{align*}
\frac {d} {dt}\int _0^\infty u(t, x)\varphi (x)dx=\int _0^\infty u(t, x)\varphi (x)\left(\int _0^\infty R(x, y)u(t, y)dy\right)dx,
\end{align*}
and since
\begin{align*}
\left| u(t,\cdot)u(t,\cdot)\varphi (t,\cdot)R(\cdot, \cdot)\right|\in L^1([0,\infty)\times [0,\infty))\quad\forall t\in[0,T_{\max}),
\end{align*}
by Fubini's Theorem and the antysimmetry of $R(x,y)$,
\begin{align}
\label{S5EW18}
\frac {d} {dt}\int _0^\infty & u(t, x)\varphi (x)dx=\int _0^\infty \int _0^\infty  \varphi (x)R(x, y) u(t,x)u(t,y)dxdy\nonumber\\
 &=\frac {1} {2}\int _0^\infty \int _0^\infty  (\varphi (x)-\varphi (y)) R(x, y)u(t,x)u(t,y)dxdy.
\end{align}
This shows that $u$ is a weak solution of  (\ref{S2EG}), (\ref{S2EG2}) .
If $\varphi =1$:
$$
\frac {d} {dt}\int _0^\infty u(t, x)dx=0,
$$
and then $\|u(t)\|_1=\|u_0\|_1$ for all $t>0$. Then, since
\begin{align*}
\int_0^{\infty}R(x,y)u(s,y)dx&\leq \frac{C_L\|u_0\|_1}{x^{3/2}},
\end{align*}
we obtain from (\ref{S2Eimplicit})
\begin{align*}
u(t,x)\leq u_0(x)e^{ \frac{tC_L\|u_0\|_1}{x^{3/2}}},
\end{align*}
and then 
\begin{align}
\label{S7E167XYm}
\|u(t)\|_{\omega}\leq \int_0^{\infty}\omega(x)e^{\frac{tC_L\|u_0\|_1}{x^{3/2}}}u_0(x)dx.
\end{align}
Notice that (\ref{flat}) implies 
\begin{align}
\label{zzz2}
\forall r>0,\qquad\int_0^1 u_0(x)\frac{e^{\frac{r}{x^{3/2}}}}{x^{3/2}}dx<\infty.
\end{align}
Indeed, if we write $x^{-3/2}=e^{-\frac{3}{2}\ln x}$, then for all $r>0$,
\begin{gather*}
\int_0^1u_0(x) \frac{e^{\frac{r}{x^{3/2}}}}{x^{3/2}}dx= \int_0^1 u_0(x)e^{\frac{r}{x^{3/2}}-\frac{3}{2}\ln x}dx
\leq \int_0^1 u_0(x)e^{\frac{r'}{x^{3/2}}}dx<\infty,\\
\shortintertext{where}
r'=r+e^{-1}=\max_{x\in[0,1]}\left(r-\frac{3}{2}x^{3/2}\ln x\right).
\end{gather*}
We then obtain from (\ref{S7E167XYm}), (\ref{zzz2}) that
\begin{align}
\label{S7E167XY}
\|u(t)\|_{\omega}\leq \int_0^{\infty}\omega(x)e^{\frac{tC_L\|u_0\|_1}{x^{3/2}}}u_0(x)dx<\infty\quad\forall t\in[0,T_{\max}),
\end{align}
therefore $\lim_{t\to T_{\max}}\|u(t)\|_{\omega}<\infty$ if $T_{\max}<\infty$, and then by the alternative, $T_{\max}=\infty$.\\

\textit{Step 2}. For a general initial data $u_0$, let $u_{0,n}(x)=u_0(x)\mathds{1}_{[0,n]}(x)$, and $u_n$ be the weak solution constructed in Step 1 for the initial data $u_{0,n}$  that satisfies
\begin{align}
\label{DP0}
u_n(t,x)=u_{0,n}(x)e^{\int_0^t\int_0^{\infty}R(x,y)u_n(s,y)dyds},
\end{align}
and $\|u_n(t)\|_1=\|u_{0,n}\|_1\leq\|u_0\|_1$ for all $t>0$ and all $n\in\NN$. Then, arguing as in the proof of Theroem \ref{MT1}, a subsequence of $\{u_n\}_{n\in\NN}$ (not relabelled) converges to some $u\in C([0,\infty),\mathscr{M}_+([0,\infty)))$ 
in the space $C([0,\infty),\mathscr{M}_+([0,\infty)))$. On the other hand,
since for all $n\in\NN$,
\begin{align}
\label{DP00}
\int_0^{\infty}R(x,y)u_n(s,y)dy\leq\frac{C^*}{x^{3/2}}\int_x^{\infty}e^{\frac{(1-\theta)y}{2}}u_n(s,y)ds\leq 
\frac{C_0}{x^{3/2}},
\end{align}
where $C_0=C^*\int_0^{\infty}e^{\eta y}u_0(y)dy$, it follows  from (\ref{flat}) that for all $\varepsilon>0$ there exists $\delta>0$ such that for all $E\subset [0,\infty)$ mesasurable with $|E|<\delta$,
\begin{align}
\label{DP1}
\int_E u_n(t,x)dx&\leq\int_{E}u_0(x)e^{\frac{C_0 t}{x^{3/2}}}dx<\varepsilon \quad\forall n\in\NN,\;\forall t>0.
\end{align}
Moreover, for all $\varepsilon>0$ there exists $M>0$ such that
\begin{align}
\label{DP2}
\int_{M}^{\infty}u_n(t,x)dx&\leq e^{-\eta M}\int_M^{\infty}e^{\eta x}u_n(t,x)dx\nonumber\\
&\leq e^{-\eta M}\int_0^{\infty}e^{\eta x}u_0(x)dx<\varepsilon\quad\forall n\in\NN,\;\forall t>0.
\end{align}
It then follows from (\ref{DP1})--(\ref{DP2}) and Dunford-Pettis Theorem, that for all $t>0$, a subsequence of $u_n(t)$ (not relabelled) converges to a function $U(t)\in L^1([0,\infty))$ in the weak topology $\sigma(L^1,L^{\infty})$.  Therefore we deduce that for all $t>0$,
\begin{align*}
\int_0^{\infty}\varphi(x)U(t,x)dx=\int_{[0,\infty)}\varphi(x)u(t,x)dx\,\forall\varphi\in C_b([0,\infty)),
\end{align*}
i.e., the measure $u(t)$ is absolutely continuous with respect to the Lebesgue measure, with density $U(t)$. With some  abuse of notation we identify $u$ and $U$. The goal now is to pass to the limit in (\ref{DP0}) as $n\to\infty$.
Since $R(x,\cdot)\in L^{\infty}([0,\infty))$ for a.e. $x>0$ and all $t>0$, and 
\begin{align}
\label{DP99}
\int_0^{\infty}|R(x,y)&|u_n(s,y)dy\leq \frac{C^*}{x^{3/2}}\left(\int_0^x e^{\frac{x-y}{2}}u_n(s,y)dy+\right.\noindent\\
&\hskip 4cm \left. +\int_x^{\infty}e^{\frac{y-x}{2}}u_n(s,y)dy\right)\nonumber\\
&\leq\frac{C^*}{x^{3/2}}\left(e^{\eta x}\|u_0\|_1+\int_0^{\infty}e^{\eta y} u_0(y)dy\right),\,\forall n\in\NN,
\end{align}
it follows by the weak convergence $u_n(t)\rightharpoonup u(t)$ and dominated convergence, that for all $t>0$, a.e. $x>0$,
\begin{align*}
\lim_{n\to\infty}\int_0^t\int_0^{\infty}R(x,y)u_n(s,y)dyds=\int_0^t\int_0^{\infty}R(x,y)u(s,y)dyds,
\end{align*}
and then, using that $u_{0,n}\to u_0$ a.e., (\ref{DP00}), and dominated convergence,
\begin{align*}
\lim_{n\to\infty}&\int_0^{\infty}u_{0,n}(x)e^{\int_0^t\int_0^{\infty}R(x,y)u_n(s,y)dyds}dx\\
&=\int_0^{\infty}u_{0}(x)e^{\int_0^t\int_0^{\infty}R(x,y)u(s,y)dyds}dx.
\end{align*}
Therefore, $u$ satisfies (\ref{S2Eimplicit}) for all $t>0$ and $a. e.\, x>0$.

Arguing as in (\ref{DP00}) we obtain (\ref{S2Ec6021}), and arguing as in Step 1 we obtain (\ref{S2Ec602}). 

We now claim that 
\begin{align}
\label{DP11}
u\in C([0,\infty),L^1((0,\infty))).
\end{align}
For all $T>0$, $t_1$ and $t_2$ with $0\leq t_1\leq t_2\leq T$, we have by (\ref{DP99}),
\begin{align*} 
&\|u(t_1)-u(t_2)\|_1\\
&\leq \int_0^{\infty} u_0(x)\left|e^{\int_0^{t_1}\int_0^{\infty}R(x,y)u(s,y)dyds}-e^{\int_0^{t_2}\int_0^{\infty}R(x,y)u(s,y)dyds}\right|dx\\
&\leq \int_0^{\infty}u_0(x)e^{\frac{TC_0}{x^{3/2}}}\left(\int_{t_1}^{t_2}\int_0^{\infty}|R(x,y)|u(s,y)dyds\right) dx\\
&\leq |t_1-t_2|\int_0^{\infty}u_0(x)e^{\frac{TC_0}{x^{3/2}}}\frac{C^*}{x^{3/2}}\left(e^{\eta x}\|u_0\|_1+\int_0^{\infty}e^{\eta y} u_0(y)dy\right)dx,
\end{align*}
and then (\ref{DP11}) follows using (\ref{flat}).  Arguing as in Step 1 we deduce that $u$ is a weak solution of  (\ref{S2EG}), (\ref{S2EG2}).
\end{proof}

\begin{proof}[\bfseries\upshape{Proof of Theorem \ref{MT3t}}]
 Theorem \ref{MT3t} follows from Proposition \ref{MT3BISt} since the function $b(k, k')=\frac{\Phi\mathcal B_\beta}{kk'}$ satisfies (\ref{B1})--(\ref{Bbound}).
\end{proof}

\begin{remark}  
\label{S5RXBT2}
The same proof shows that Theorem \ref{MT3t} is still true for the equation 
\begin{equation}
\label{S1ER1NC}
\frac {\partial v} {\partial t}(t,k)=v(t, k)\!\!\!\!\int\limits_{[0, \infty)} \!\!\!\! v(t, k')
\big(e^{- \beta k }-e^{-\beta  k'}\big)\frac{\mathcal B_\beta (k, k')}{kk'}dk'.
\end{equation}
where the redistribution function $\mathcal B_\beta $ is kept without  truncation. This is possible because the property (\ref{flat}) is also propagated by the weak solutions of (\ref{S1ER1NC}) such that
\begin{align}
\label{FOBISNC}
v(t, k)=v_0(k)e^{\int_0^{t}\int_0^\infty \big(e^{- \beta k }-e^{-\beta  k'}\big)\frac{\mathcal B_\beta (k, k')}{kk'}\, v(s, k')\, dk'\, ds}.
\end{align} 
Notice in particular that the integral term in the exponential is well defined when $v(t)$ satisfies (\ref{flat}).
\end{remark}

\begin{remark}
Let $u$ and $v$ be two solutions of (\ref{S1ER1}), with a compactly supported initial data $u_0\in L^1([0,\infty))$ satisfying (\ref{flat}) and such that $\supp (u_0)\subset [0,L]$, $L>0$. It follows from the representation(\ref{S2Eimplicit}) that, for all $t>0$ and  $a.e.\, x>0$,
\begin{align*}
&|u(t,x)-v(t, x)|\leq u_0(x)e^{\frac{t C_L\|u_0\|_1}{x^{3/2}}} \frac{C_L}{x^{3/2}}\int_0^t\int_0^{\infty}|u(s,y-v(s,y)|dyds,
\end{align*}
and then, by Gronwall's Lemma, $u=v$ for a.e. $t>0$ and a.e. $x>0$.
\end{remark}

\subsection{$\boldsymbol{M_\alpha} $ as Lyapunov functional. } 
\label{entropy}
The  goal of this Section is the study of the functionals $M _{ \alpha  }(u(t))$, defined in (\ref{definition moment}), and $ D _{ \alpha}(u(t))$, defined in  (\ref{S1ERDH}),
acting on the weak solutions of problem (\ref{S2EG}), and to prove, in particular, Theorem \ref{characterization}. 

Let us start with the following simple lemma, that establishes a monotonicity property for the moments of a solution to (\ref{S2EG}).

\begin{lemma}
\label{LmomZ}
Let $u$ be the weak solution of (\ref{S2EG}) given by Theorem \ref{S8Th1} for an initial data $u_0\in\mathscr{M}_+([0,\infty))$ satisfying (\ref{Zcondition}). Then, the weak formulation (\ref{WSZ}) holds for $\varphi(x)=x^{\alpha}$ for all $\alpha\geq 1$. Moreover, for all $t_0\geq0 $,
\begin{gather}
\label{momentsA}
M_{\alpha}(u(t))\leq M_{\alpha}(u(t_0))\qquad\forall t\geq t_0.
\end{gather}
\end{lemma}

\begin{proof}
Let $\alpha\geq1$ and $\varphi(x)=x^{\alpha}$. We first notice from (\ref{Zcondition}) and (\ref{Zmoment}) that $M_{\alpha}(u(t))<\infty$ for all $t\geq0$. 
Then, consider an approximation
$\{\varphi_k\}_{k\in\NN}\subset C^1_b([0,\infty))$ such that $\varphi_k$ is nondecreasing, $\varphi_k'(0)=0$, $\varphi_k'\leq \varphi'$ for all $k\in\NN$, and $\varphi_k\to\varphi$ pointwise as $k\to\infty$. Using the definite sign of the right hand side of (\ref{WSZ}) for the test function $\varphi_k$, we obtain
$$\frac{d}{dt}\int_{[0,\infty)}\varphi_k(x)u(t,x)dx\leq 0\qquad\forall t>0,\;\;\forall k\in\NN,$$
from where, for all $t_0\geq 0$, $\int_{[0,\infty)}\varphi_k(x)u(t,x)dx\leq\int_{[0,\infty)}\varphi_k(x)u(t_0,x)dx$ for all $t\geq t_0$ and all 
$k\in\NN$, and then (\ref{momentsA})
follows from dominated convergence theorem, by letting $k\to\infty$.

Let us prove now that (\ref{WSZ}) holds for $\varphi(x)=x^{\alpha}$. From (\ref{WSZ}) for the test function $\varphi_k$,
\begin{align}
\int_{[0,\infty)}&\varphi_k(x)u(t,x)dx=\int_{[0,\infty)}\varphi_k(x)u_0(x)dx\nonumber\\
&+\int_0^t \iint_{[0,\infty)^2}k_{\varphi_k}(x,y)u(s,x)u(s,y)dydxds. \label{E100}
\end{align}
Using that $\varphi_k'\leq \varphi'$ for all $k\in\NN$, we obtain from (\ref{kboundC1}) that for all $(x,y)\in\Gamma$,
\begin{equation*}
|k_{\varphi_k}(x,y)|\leq C\alpha\max\{x,y\}^{\alpha-1}e^{\frac{|x-y|}{2}},
\quad C=\max\left\{\frac{(1-\theta)^2}{\theta\delta_*(1+\theta)},\rho_*\right\},
\end{equation*}
and, since $|x-y|\leq (1-\theta)\max\{x,y\}$ and $\max\{x,y\}\leq\theta^{-1}\min\{x,y\}$ for all $(x,y)\in\Gamma$, we then deduce using also (\ref{Zmoment}) and (\ref{momentsA}), that for all $t\geq 0$ and $k\in\NN$,
\begin{align*}
&\iint_{[0,\infty)^2}|k_{\varphi_k}(x,y)|u(t,x)u(t,y)dydx\\
&\leq \frac{C\alpha}{\theta^{\alpha-1}}\iint_{[0,\infty)^2}\min\{x,y\}^{\alpha-1}e^{\frac{(1-\theta)}{2}\max\{x,y\}}u(t,x)u(t,y)dydx\\
&\leq\frac{2C\alpha}{\theta^{\alpha-1}}\bigg(\int_{[0,\infty)}e^{\frac{(1-\theta)x}{2}}u(t,x)dx\bigg)\bigg(\int_{[0,\infty)}y^{\alpha-1}u(t,y)dy\bigg)\\
&\leq \frac{2C\alpha}{\theta^{\alpha-1}}\bigg(\int_{[0,\infty)}e^{\eta x}u_0(x)dx\bigg)\bigg( \int_{[0,\infty)}y^{\alpha-1}u_0(y)dy\bigg).
\end{align*}
On the other hand, $k_{\varphi_k}(x,y)\to k_{\varphi}(x,y)$ for all $(x,y)\in[0,\infty)^2$ as $k\to\infty$. Passing to the limit as $k\to\infty$ in (\ref{E100}), it then follows from dominated convergence theorem that for all $t\geq 0$,
\begin{align}
\int_{[0,\infty)}& \varphi(x)u(t,x)dx=\int_{[0,\infty)}\varphi(x)u_0(x)dx\nonumber\\
&+\int_0^t \iint_{[0,\infty)^2}k_{\varphi}(x,y)u(s,x)u(s,y)dydxds \label{E101},
\end{align}
and then (\ref{WSZ}) holds.
\end{proof}

If $u$ is  a weak solution to (\ref{S2EG})  given by Theorem \ref{S8Th1}, then by Lemma \ref{LmomZ} the following identity holds,
\begin{equation}
\label{S7EHD}
\frac {d} {dt}M_{ \alpha  }(u(t))=\frac {1} {2}D _{ \alpha  } (u(t)) \qquad \forall t>0.
\end{equation}
Since $D_\alpha (u(t))\le 0$ for all $t>0$, this shows that $M_\alpha $ is a  Lyapunov functional on these  solutions. The identity (\ref{S7EHD}) is reminiscent of the  usual  entropy -  dissipation of entropy identity.
As already observed in the Introduction, since the support of the function $B$ is contained in the region $\Gamma  \subset [0,\infty)^2$, if $a>0$ and $b>0$ are such that $(a, b)\not \in \Gamma $ (they do not see each other) then, for all $\varphi \in C_b^1([0, \infty))$ such that $\varphi '(0)=0$,
$$
\iint _{ [0, \infty)^2 }\delta (x-a)\delta (y-b)R(x, y)(\varphi (x)-\varphi (y))dxdy=0.
$$
Let us then see  some of the  consequences of this  simple observation.

\begin{definition}
\label{def. disjoint}
We say that two points $a$ and $c$ on $[0,\infty)$ are $\Gamma$-disjoint if  $(a, c)\not \in \Gamma $.
We say that two sets $A$ and $C$ on $[0,\infty)$ are $\Gamma$-disjoint if  for all $(a,c)\in A\times C$, $(a, c)\not \in \Gamma$,   i.e., if $A\times C\subset [0,\infty)^2\setminus\mathring{\Gamma}$.
\end{definition}
Since the support  of any given measure  $u\in\mathscr{M}_+([0,\infty))$ is, by definition, a closed subset of $[0,\infty)$, then
\begin{equation}
\label{supp D}
(\supp (u))^c=\bigcup_{k=0}^{\infty}I_k,\,\,\, I_k\;\text{open interval },\,\, I_k\cap I_j=\emptyset\;\;\text{if}\;\;k\neq j.
\end{equation} 
We may write $I_k=(a_k,b_k)$ for $0\leq a_k<b_k$ for all $k\in\NN$, except if $\supp u\subset [r,\infty)$, $r>0$, for which $I_k=[0=a_k,b_k)$ for some $k$.
We now define
\begin{align*}
\mathcal{I}=\{I_k:\gamma_1(b_k)\geq a_k\},
\end{align*}
and denote $\{C_k\}_{k\in\mathcal{J}}$ the connected components of $\big(\bigcup_{I\in\mathcal{I}}I\big)^c$. Notice that, in general, $\mathcal{J}$ could be uncountable.
Finally define, for all  $u\in\mathscr{M}_+([0,\infty))$
\begin{equation}
\label{S7EAk}
A_k (u)=C_k\cap\supp (u),\,\,\,\forall k\in \mathcal{J}.
\end{equation} 
Notice by (\ref{S7EAk}) that $A_k(u)$ is a closed subset of $[0,\infty)$ for all $k\in\NN$, since it is the intersection of two closed sets.

We write $A_k (u)=A_k$ when no confusion is possible.
\begin{lemma}
$\mathcal{J}$ is a countable set.
\end{lemma}
\begin{proof}
Given two elements of $\mathcal{I}$, there is at most a finite number of elements of $\mathcal{I}$ between them. More precisely, we claim that, for any given $I_i\in\mathcal{I}$, $I_j\in\mathcal{I}$, with $I_i=(a_i,b_i)$, $I_j=(a_j,b_j)$, $0<b_i\leq a_j$, then:
$\text{card}(\{I_k=(a_k,b_k)\in\mathcal{I}: b_i\leq a_k<b_k\leq a_j\})<\infty.$
The proof of this fact start with this trivial remark: if $I_k\in\mathcal{I}$, then $|I_k|=b_k-a_k\geq b_k-\gamma_1(b_k)$. Using that, if we consider the decreasing sequence $b_j$, $\gamma_1(b_j)$, $\gamma_1^2(b_j)=\gamma_1(\gamma_1(b_j))$,
$\gamma_1^3(b_j)$,..., then $\gamma_1^m(b_j)<b_i$ for some integer $m$, and therefore there could be only $m$ elements of $\mathcal{I}$ between $I_i$ and $I_j$.
 
For the sake of the argument, let us say that given two elements $I_i=(a_i,b_i)$ and $I_j=(a_j,b_j)$ of $\mathcal{I}$, there are $2$ more elements 
$I_1=(a_1,b_1)$ and $I_2=(a_2,b_2)$ of $\mathcal{I}$ between them, i.e., 
\begin{align*}
a_i<b_i\leq a_1<b_1\leq a_2<b_2\leq a_j<bj.
\end{align*}
Then, there are $3$ connected components in $(a_i,b_j)\setminus \big(I_i\cup I_1\cup I_2\cup I_j\big)$, namely
 $[b_i,a_1]$, $[b_1,a_2]$ and $[b_2,a_j]$. With this idea, it can be proved that the number of connected components of
 $[0,\infty)\setminus \big(\bigcup_{I\in\mathcal{I}}I\big)$, i.e., the collection $\{C_k\}_{k\in\mathcal{J}}$, is at most countable. 
\end{proof}

We prove now several useful  properties of the collection $\{A_k\}_{k\in\NN}$. 
\begin{lemma}
\label{lemma. blocks}
Let $u\in\mathscr{M}_+([0,\infty))$ and consider the collection $\{A_k\}_{k\in\NN}$ constructed above. Then
$A_i$ and $A_j$ are $\Gamma$-disjoint if and only if $i\neq j$, and 
\begin{align}
\label{supp blocks}
\supp (u)=\bigcup_{k=0}^{\infty}A_k.
\end{align}
\end{lemma}
\begin{proof}
It is clear that $A_i$ and $A_i$ are not $\Gamma$-disjoint, since $A_i\times A_i$ contains points on the diagonal, and therefore on $\mathring\Gamma$. Now, if $i\neq j$, we first observe that $A_i$ and $A_j$ are disjoint. Indeed, by definition $A_i\subset C_i$ and $A_j\subset C_j$, where $C_i$ and $C_j$ are different connected components of $[0,\infty)\setminus\big(\bigcup_{I\in\mathcal{I}}I\big)$, therefore disjoint. We now prove that $A_i$ and $A_j$ are in fact $\Gamma$-disjoint. Let us assume that $A_i$ is on the left of $A_j$, i.e., $\sup A_i<\inf A_j$. It follows from the construction that there exists at least one $I_k=(a_k,b_k)\in\mathcal{I}$ between $A_i$ and $A_j$, i.e., such that
$$\sup A_i\leq a_k<b_k\leq \inf A_j.$$
By definition of $\mathcal{I}$, the points $a_k$ and $b_k$ are $\Gamma$-disjoint, and then, for all 
$(a_i,a_j)\in A_i\times A_j$,
$$\gamma_1(a_j)\geq\gamma_1(b_k)\geq a_k\geq a_i,$$
hence $a_i$ and $a_j$ are $\Gamma$-disjoint. Finally, (\ref{supp blocks}) follows from the construction. Indeed, since by definition $A_k=C_k\cap\supp (u)$, then $\cup_{k\in\NN}A_k\subset\supp u$. On the other hand, by definition
$\cup_{k\in\NN} C_k=[0,\infty)\setminus\big(\cup_{I\in\mathcal{I}}I\big)$, and then by (\ref{supp D})
\begin{align*}
\supp (u)=\bigcap_{k\in\NN}I_k^c\subset \bigcup_{k\in\NN}C_k,
\end{align*}
from where the inclusion $\supp (u)\subset \cup_{k\in\NN}A_k$ follows.
\end{proof}

In the remaining part of the section we will use several times the following simple remark.
\begin{remark}
\label{simple remark function z}
Consider the function $z(x)=x-\gamma_1(x)$, $x\geq 0$, where $\gamma_1$ is given by (\ref{g}) in Remark \ref{cone}. Then, $z$ is a continuous and strictly increasing function on $[0,\infty)$, with $z(0)=0$.
\end{remark}
In the next Lemma we prove that any two sets $A_i$ and $A_j$ of the collection $\{A_k\}_{k\in\NN}$ are separated from each other by a positive distance, given by the function $z(x)$ of Remark \ref{simple remark function z} .  
\begin{lemma}
\label{corollary uniform distance}
Let $u\in\mathscr{M}_+([0,\infty))$ and consider the collection $\mathcal{A}=\{A_k\}_{k\in\NN}$ constructed above.  Suppose that 
$\emph{card}(\mathcal{A})\geq 2$. For any $k\in\NN$, let us denote
$x_k=\min A_k$ and $y_k=\sup A_k$. Given two elements $A_i$, $A_j$ in $\mathcal{A}$, suppose that $y_i<x_j$. Then,
\begin{equation}
\label{bound distance AA}
\emph{dist}(A_i,A_j)\geq x_j-\gamma_1(x_j)>0.
\end{equation}
Moreover, for every $\varepsilon>0$, let 
\begin{equation}
\label{definition of A epsilon}
\mathcal{A}_{\varepsilon}=\{A_k\in \mathcal A: A_k\subset(\varepsilon,\infty)\}.
\end{equation}
If $\mathcal{A}_{\varepsilon}\neq \emptyset$ and $\emph{card}(\mathcal{A}_{\varepsilon})\geq 2$, then 
\begin{equation}
\label{bound distance A epsilon}
\emph{dist}(A_i,A_j)> \varepsilon-\gamma_1(\varepsilon)>0\qquad\forall A_i,A_j\in\mathcal{A}_{\varepsilon},\,i\neq j.
\end{equation}
\end{lemma}

\begin{proof}
Since  $A_i$ and $A_j$  are closed sets and $y_i<x_j$, it follows that  $\text{dist}(A_i,A_j)=x_j-y_i$. 
By Lemma \ref{lemma. blocks}, the closed sets $A_i$ and $A_j$ are $\Gamma$-disjoint and then, by Definition
\ref{def. disjoint},
$$
y_i\le \gamma _1(x_j).
$$
Therefore   $\text{dist}(A_i,A_j)\ge x_j-\gamma _1(x_j)$ and, since $x_j>0$,  (\ref{bound distance AA}) follows from Remark
\ref{simple remark function z}.

Let now $\varepsilon>0$ be fixed and consider $A_i$ and $A_j$ in $\mathcal{A}_{\varepsilon}$. Without loss of generality, we may assume that 
$y_i<x_j$. Using Remark \ref{simple remark function z}, it then follows from (\ref{bound distance AA}) and (\ref{definition of A epsilon}) that
\begin{equation*}
\text{dist}(A_i,A_j)\geq z(x_j)> z(\varepsilon)>0.
\end{equation*}
\end{proof}

\begin{lemma}
\label{constant mass blocks}
Let $u$ be the weak solution of (\ref{S2EG}) constructed in Theorem \ref{S8Th1} for an initial data $u_0\in\mathscr{M}_+([0,\infty))$ satisfying (\ref{Zcondition}), and consider the collection $\mathcal{A}=\{A_k(u_0)\}_{k\in\NN}$ constructed above. Then 
\begin{equation}
\label{mass blocks}
\int_{A_k}u(t,x)dx=\int_{A_k}u_0(x)dx\qquad\forall t>0,\;\forall k\in\NN.
\end{equation}
\end{lemma}

\begin{proof}
In the trivial case $A_k=\supp (u_0)$ for all $k\in\NN$, then (\ref{mass blocks}) is just the conservation of mass (\ref{ZmassR}).
Suppose then that $\text{card}(\mathcal{A})\geq 2$. We consider separately two different cases. 

(i) Suppose that there exists $\varepsilon>0$ such that $[0,\varepsilon]\subset\supp(u_0)$.  Then, since $[0,\varepsilon]$ can not intersect  $A_k$ for two different values of $k$, there exists $k_0\in \NN$ such that $[0,\varepsilon]\subset A _{ k_0 }$. In particular $A _{ k_0 }\not \in \mathcal A_\varepsilon $. Let us see that
\begin{equation}
\label{decomposition A case 1}
\mathcal A=\mathcal A_\varepsilon  \cup \{A _{ k_0 }\}.
\end{equation}

Arguing by contradiction, suppose that for some $\ell \not =k_0$ we have $A_\ell \in \mathcal A \setminus \mathcal A_\varepsilon $. Since $[0,\varepsilon]\subset A_{k_0}$ and $A_{k_0}\cap A_{\ell}=\emptyset$, then $x_{\ell}=\min A_{\ell}>\varepsilon$. Therefore $A_{\ell}\in\mathcal{A}_{\varepsilon}$, which is a contradiction.

We wish now to estimate from below  the distances $\text{dist}(A_i, A_j)$ for all $A_i\in \mathcal A$, $A_j\in \mathcal A$, $i\not =j$.
By  (\ref{bound distance A epsilon}) and (\ref{decomposition A case 1}),
\begin{equation}
\label{uniform bound 10}
\text{dist}(A_i,A_j)>\varepsilon-\gamma_1(\varepsilon)>0\qquad\forall i\not = k_0, \forall j\not= k_0, i\not =j.
\end{equation}
On the other hand, for all $i\not =k_0$, $x_i=\min A_i >\varepsilon $ by (\ref{decomposition A case 1}) and then,
by (\ref{bound distance AA}) and Remark \ref{simple remark function z},
\begin{equation}
\text{dist}(A_i,A _{ k_0 })\ge x_i-\gamma _1(x_i)=z(x_i)>z(\varepsilon ). \label{uniform bound 1348}
\end{equation}
By  (\ref{decomposition A case 1}), (\ref{uniform bound 10}) and (\ref{uniform bound 1348}) we have then:
\begin{equation}
\text{dist}(A_i,A _{ j })>z(\varepsilon )>0, \,\,\,\forall A_j\in \mathcal A, \, \forall A_i\in \mathcal A. \label{uniform bound 1248}
\end{equation}

For any fixed  $k\in\NN$, we now claim that, since $A_i$ is closed for every $i\in\NN$, by (\ref{uniform bound 1248})   the set
\begin{equation}
D_k=\bigcup_{i\in\NN,i\neq k}A_i   \label{uniform bound 1247}
\end{equation}
is a closed subset of $[0, \infty)$. In order to prove that property, let us assume, by contradiction, that there exists a point $x_*\in \overline D_k \setminus D_k$. Let $\{x_n\} _{ n\in \NN }\subset D_k$ be a sequence such that converges to $x_*$. In particular $\{x_n\} _{ n\in \NN }$ is a Cauchy sequence.  Therefore, by (\ref{uniform bound 1248}), there exists $k_*\in \NN\setminus \{k\}$ such that, for some $n_*$ sufficiently large:
$$
x_n\in A _{ k_* },\,\,\forall n\ge n_*.
$$
Since $A _{ k_* }$ is a closed set, it follows that  $x_*\in A _{ k_* }\subset D_k$, and this is a contradiction. 

By  (\ref{uniform bound 1248}), $D_k$ and $A_k$ are disjoint subsets of $[0,\infty)$. Therefore,  by Urysohn's lemma, there exists a function $\varphi\in C_b([0,\infty))$ such that
\begin{equation*}
\varphi(x)=1\quad\forall x\in A_k\quad\text{and}\quad
\varphi(x)=0\quad\forall x\in D_k.
\end{equation*}
Using (\ref{uniform bound 1248}) and a density argument, we may assume that $\varphi\in C^1_b([0,\infty)).$
Then, since $\supp (u(t))=\supp (u_0)$ (cf. Proposition \ref{comparison} (iii)), it follows from (\ref{supp blocks})
\begin{align*}
\int_{[0,\infty)}\varphi(x)u(t,x)dx=\int_{A_k}u(t,x)dx,
\end{align*}
and since $A_i$ and $A_j$ are $\Gamma$-disjoint for $i\neq j$ (cf. Lemma \ref{lemma. blocks}), then by construction of $\varphi$,
\begin{align*}
\iint_{[0,\infty)^2}&R(x,y)(\varphi(x)-\varphi(y))u(t,x)u(t,y)dydx\\
&=\sum_{i=0}^{\infty}\iint_{A_i\times A_i}R(x,y)(\varphi(x)-\varphi(y))u(t,x)u(t,y)dydx\\
&=\iint_{A_k\times A_k}R(x,y)(\varphi(x)-\varphi(y))u(t,x)u(t,y)dydx=0,
\end{align*}
from where (\ref{mass blocks}) follows by the weak formulation.

(ii) Suppose that  the assumption of part (i) does not hold. In this case, there exists a strictly decreasing sequence $\{x_n\}_{n\in\NN}$ with 
$x_n\to 0$ as $n\to\infty$ such that $x_n\notin\supp (u_0)$ for all $n\in\NN$. Moreover, since $\supp(u_0)$ is a closed set, for each $n\in\NN$ there exists 
$\delta_n>0$ such that  
$$(x_n-\delta_n,x_n+\delta_n)\subset (\supp(u_0))^c.
$$ 
For every $n\in\NN$ and  $k\in\NN$  fixed such that
\begin{equation}
A_k\in \mathcal A _{ x_n },  \label{2348P}
\end{equation}
 where $\mathcal{A}_{x_n}$ is defined in (\ref{definition of A epsilon}), we consider the set:
$$
D _{ k, n }=\bigcup_{\substack{A_i\in\mathcal{A}_{x_n}\\A_i\neq A_k}} A_i
$$
Using now (\ref{bound distance A epsilon}) for $\varepsilon=x_n$ we deduce that $D _{ k, n }$ is a closed set by the same argument as for $D_k$ in  (\ref{uniform bound 1247}). By Urysohn's lemma again, we can then construct a test function $\varphi\in C^1_b([0,\infty))$ such that
\begin{equation*}
\varphi(x)=1\quad\forall x\in A_k\qquad\text{and}\qquad\varphi(x)=0\quad\forall x\in[0,x_n]\cup D _{ k, n }.
\end{equation*}
Arguing as in part (i), we then deduce that
\begin{equation}
\label{mass blocks xn}
\int_{A_k}u(t,x)dx=\int_{A_k}u_0(x)dx\qquad\forall t>0,\;\forall A_k\in\mathcal{A}_{x_n}.
\end{equation}
We use now that
\begin{equation*}
\mathcal{A}=\bigg(\bigcup _{ n\in \NN }\mathcal A _{ x_n }\bigg) \bigcup \Big\{ A_i\in \mathcal A: A_i \not \subset (0, \infty)\Big\}
\end{equation*}
because
\begin{equation*}
\bigcup _{ n\in \NN }\mathcal A _{ x_n } = \left\{ A_j \in \mathcal A: A_j  \subset (0, \infty)\right\}.
\end{equation*}
But, if $A_i \not \subset (0, \infty)$, then $0\in A_i$. Therefore, if $0\not \in \supp (u_0)$ there is no such $A_i$. 
If $0 \in \supp (u_0)$, since the sets $A_k$ are pairwise disjoint, such subset $A_i$ is unique. It follows that there exists at most a unique $k_0\in \NN$ such that:
\begin{equation}
\mathcal{A}=\bigg(\bigcup _{ n\in \NN }\mathcal A _{ x_n }\bigg) \bigcup \Big\{ A _{ k_0 }\Big\}.
\label{555P}
\end{equation}
The equality (\ref{mass blocks}) then follows from (\ref{mass blocks xn}), (\ref{555P}) and the conservation of mass (\ref{ZmassR}).
\end{proof}

We may prove now the main result of this Section.
\begin{proof}[\bfseries\upshape{Proof of Theorem \ref{characterization}}]
Let us prove $(i)\implies(iii)$. Suppose that $D_{\alpha}(u)=0$, and let, for $\varepsilon>0$,
$$\Gamma_{\varepsilon}=\{(x,y)\in\Gamma: d((x,y), \partial\Gamma)>\varepsilon,\;|x-y|>\varepsilon\}.$$
Since
$b(x,y)(e^{-x}-e^{-y})(x^{\alpha}-y^{\alpha})<0$ for all $(x,y)\in\Gamma_{\varepsilon}$,
it follows from $(i)$ that 
$
\supp (u\times u)\subset \Gamma_{\varepsilon}^c.
$
Letting $\varepsilon\to 0$, we deduce that
\begin{align}
\label{supp2}
\supp (u\times u)\subset \Delta\cup(\mathring{\Gamma})^c,\qquad\Delta=\{(x,x):x\geq 0\}.
\end{align}
Notice that any two points $y<x$ in the support of $u$ have to be at distance, namely, $x-y\geq x-\gamma_1(x)$. Otherwise
$\gamma_1(x)<y$ and then $(x,y)\in\mathring{\Gamma}\setminus\Delta$, in contradiction with (\ref{supp2}). Moreover, since the map $z(x)= x-\gamma_1(x)$ is continuous and strictly increasing on $[0,\infty)$, with $z(0)=0$,
it follows that the support of $u$ consists, at most, on a countable number of points, where the only possible accumulation point is $x=0$. Therefore $A_k=\{x_k\}$ for all $k\in\NN$, and then $(iii)$ holds. 

Let us prove $(iii)\implies (i)$. If $u$ is as in $(iii)$, then $\supp (u\times u)=\{(x_i,x_j):i,j\in\NN\}$, and then
\begin{align*}
D_{\alpha}(u)=\sum_{i\leq j}\chi(i,j)\alpha_i\alpha_jb(x_i,x_j)(e^{-x_i}-e^{-x_j})(x_i^{\alpha}-x_j^{\alpha})=0,
\end{align*}
where $\chi(i,j)=2$ if $i\neq j$ and $\chi(i,j)=1$ if $i=j$. Indeed, the terms with $i=j$ vanish due to the factor $(e^{-x_i}-e^{-x_j})(x_i^{\alpha}-x_j^{\alpha})$, and for those terms with $i\neq j$, then $b(x_i,x_j)=0$ since $(x_i,x_j)\notin \Gamma$.

We now prove $(iii)\implies(ii)$. Using (\ref{supp blocks}) in Lemma \ref{lemma. blocks} and the definition of $x_k$, for any $v\in\mathcal{F}$, 
\begin{align*}
M_{\alpha}(v)=\sum_{k=0}^{\infty}\int_{A_k}x^{\alpha}v(x)dx\geq \sum_{k=0}^{\infty}x_k^{\alpha}m_k=M_{\alpha}(u),
\end{align*} 
and since $u\in\mathcal{F}$, $u$ is indeed the minimizer of $M_\alpha $.

We finally prove $(ii)\implies (iii)$. Let $u$ be a minimizer of $M_\alpha $ and let $v=\sum_{k=0}^{\infty}m_k\delta_{x_k}$. We already know by the previous case that $v$ is also a minimizer of  $M_\alpha $, hence $M_{\alpha}(u)=M_{\alpha}(v)$. Since moreover
\begin{align*}
M_{\alpha}(v)=\sum_{k=0}^{\infty}x_k^{\alpha}m_k=\sum_{k=0}^{\infty}x_k^{\alpha}\int_{A_k}u(x)dx,
\end{align*}
it follows that 
\begin{align*}
\sum_{k=0}^{\infty}\int_{A_k}(x^{\alpha}-x_k^{\alpha})u(x)dx=0.
\end{align*}
By definition of $x_k$, all the terms in the sum above are nonnegative, and therefore
\begin{align*}
\int_{A_k}(x^{\alpha}-x_k^{\alpha})u(x)dx=0\qquad\forall k\in\NN,
\end{align*}
which implies that $A_k=\{x_k\}$ for all $k\in\NN$, and therefore $u=v$.  
\end{proof}

\subsection{Long time behavior.}
\label{long time Z}
\label{behavior}
This Section is devoted to he proof of  Theorem \ref{S7E001}, that we have divided in several steps. For a given increasing sequence $t_n\to\infty$ as $n\to\infty$, let us define
\begin{align}
\label{long time}
u_n(t)=u(t+t_n),\qquad t\geq 0,\;n\in\NN,
\end{align}
where $u$ is the weak solution of (\ref{S2EG}) constructed in Theorem \ref{S8Th1} for an initial data $u_0\in\mathscr{M}_+([0,\infty))$ satisfying (\ref{Zcondition}). We first notice by (\ref{S7EHD}) and Lemma \ref{prop.64} that for all $\alpha>1$ and $t>0$,
\begin{align*}
\frac{1}{2}\int_0^t |D_{\alpha}(u(s))|ds=M_{\alpha}(u_0)-M_{\alpha}(u(t))\leq M_{\alpha}(u_0),
\end{align*}
so by letting $t\to\infty$ we deduce $D_{\alpha}(u)\in L^1([0,\infty))$. Since moreover
\begin{align*}
\int_0^{t}D_{\alpha}(u_n(s))ds=\int_{t_n}^{t_n+t}D_{\alpha}(u(s))ds,\qquad\forall t\geq 0,
\end{align*}
it follows that 
\begin{align}
\label{limit dissipation}
\lim_{n\to\infty}\int_0^{t}D_{\alpha}(u_n(s))ds=0\qquad\forall t\geq 0.
\end{align}

\begin{proposition}
\label{long time limit}
Let $u$ be the weak solution of (\ref{S2EG}) constructed in Theorem \ref{S8Th1} for an initial data $u_0\in\mathscr{M}_+([0,\infty))$ satisfying (\ref{Zcondition}).  For every sequence $\{t_n\} _{ n\in \NN }$ such that $t_n\to \infty$, there exist a subsequence, still denoted $\{t_n\} _{ n\in \NN }$, and 
\begin{align}
U\in C([0,\infty),\mathscr{M}_+([0,\infty)))
\end{align}
such that  for all $\varphi\in C([0,\infty))$ satisfying (\ref{GC}), and all $t>0$,
\begin{align}
\label{u99Z}
\lim_{n\to\infty}\int_{[0,\infty)}\varphi(x)u(t+t_n,x)dx=\int_{[0,\infty)}\varphi(x)U(t,x)dx.
\end{align}
Moreover, $U$ is a weak solution of (\ref{S2EG})   such that  $M_0(U(t))=M_0(u_0)$ and for all $t>0$.
\end{proposition}

\begin{proof}
The proof is the same as the first part of the  proof of Theorem \ref{MT1} for equation (\ref{P}).
\end{proof}

\begin{lemma}
\label{lemma. contained support}
Let $u$, $u_0$ and $U$ be as in Proposition \ref{long time limit}. Then 
\begin{align}
\label{contained support}
\supp (U(t))=\supp (U(0))\subset\supp (u_0)\qquad\forall t\geq 0.
\end{align} 
\end{lemma}

\begin{proof}
On the one hand, since $U$ is a weak solution of (\ref{S2EG}), then by Proposition \ref{comparison} it follows that $\supp (U(t))=\supp (U(0))$ for all $t>0$, where $U(0)$ is given by (\ref{u99Z}) for $t=0$. On the other hand, again by Proposition \ref{comparison} we have, in particular, that $\supp (u_n(0))=\supp (u_0)$ for all $n\in\NN$. The result then follows from the convergence of $u_n(0)$ towards $U(0)$ in the sense of (\ref{u99Z}). Indeed, let $x_0\in\supp U(0)$. We use the characterization of the support of a measure given in the proof of part (iii) of Proposition \ref{comparison}. Then
\begin{align*}
\rho_{\varphi}=\int_{[0,\infty)}\varphi(x)U(0,x)dx>0,
\end{align*}
for all $\varphi\in C_c([0,\infty))$ such that $0\leq\varphi\leq 1$ and $\varphi(x_0)>0$. Using then (\ref{u99Z}) for $t=0$, we deduce that for all $\varphi$ as before, there exists $n_*\in\NN$, such that
\begin{align*}
\int_{[0,\infty)}\varphi(x)u_n(0,x)dx\geq\frac{\rho_{\varphi}}{2}>0\qquad\forall n\geq n_*,
\end{align*}
and then $x_0\in \supp (u_n(0))=\supp (u_0)$.
\end{proof}
A partial identification of the limit $U$ is given in our next Proposition.

\begin{proposition}
\label{S5ECVT}
Let $u$, $u_0$ and $U$ be as in Proposition \ref{long time limit}. Then
\begin{align}
\label{stat 222}
U(t)=\mu \qquad\forall t\geq 0,
\end{align}
where $\mu $ is the measure defined in (\ref{stat}).
\end{proposition}

\begin{proof}
We first prove that $D_{\alpha}(U(t))=0$ for a.e. $t>0$ and for all $\alpha>1$. Indeed, if we define as in  proof of Theorem \ref{MT1}, $u_n(t)=u(t+t_n)$, we deduce by the same arguments
\begin{align*}
\lim_{n\to\infty}\int_0^t D_{\alpha}(u_n(s))ds=\int_0^t D_{\alpha} (U(s))ds=0\qquad\forall t\geq 0,
\end{align*}
hence $D_{\alpha}(U(t))=0$ for a.e. $t\geq 0$. \\
Then by Theorem \ref{characterization} , there exist 
$m_j(t)\geq 0$, $x_j(t)\geq 0$ such that,
\begin{gather}
\label{stationary U1}
U(t)=\sum_{j=0}^{\infty}m_j(t)\delta_{x_j(t)},\\
\label{disjoo}
x_i(t),\:x_j(t)\quad\text{are}\;\;\Gamma\text{-disjoint}\;\;\forall i\neq j.
\end{gather}
By (\ref{contained support}) in Proposition \ref{lemma. contained support},
\begin{align}
\label{xk}
x_j(t)=x_j(0):=x_j'\in\supp(u_0)\qquad\forall t\geq 0,\;\forall j\in\NN.
\end{align}
Furthermore, since by Proposition \ref{long time limit},  $U$  is a weak solution of (\ref{S2EG}),  it follows from Lemma \ref{constant mass blocks} that for all $t\geq 0$, $j\in\NN$,
\begin{align*}
m_j(t)=\int_{\{x_j(t)\}}U(t,x)dx=\int_{\{x_j(0)\}}U(0,x)dx=m_j(0):=m_j',
\end{align*}
and then by (\ref{stationary U1}) we conclude that  $U$ is independent of $t$

Let us prove now that $U$ satisfies Properties 1-4.  Properties 1 and 2 are already proved in (\ref{disjoo}) and (\ref{xk}). In order to prove 3, let $k\in\NN$ and 
$\varphi\in C^1_c([0,\infty))$ be such that $\varphi(x)=1$ for all $x\in A_k$ and $\varphi(x)=0$ for all $x\in\cup_{i\neq k} A_i$. This construction is possible by Urysohn's Lemma. Then, by (\ref{mass blocks}) in Lemma \ref{constant mass blocks},
\begin{align*}
\int_{[0,\infty)}\varphi(x)u_n(t,x)dx=\int_{A_k}u_n(t,x)dx=m_k,\qquad\forall n\in\NN,
\end{align*}
and then by (\ref{u99Z}) in Proposition \ref{long time limit},
\begin{align*}
\int_{[0,\infty)}\varphi(x)U(x)dx=m_k.
\end{align*}
Since $\supp (U)\subset\supp (u_0)$ by Lemma \ref{lemma. contained support}, we then deduce
\begin{align*}
\int_{[0,\infty)}\varphi(x)U(x)dx=\int_{A_k}U(x)dx=\sum_{j\in\mathcal{J}_k}m_j',
\end{align*}
and thus Property 3 holds. 

Let us prove Property 4. 
Let $k\in\NN$ and suppose that $x_k=\min\{x\in A_k\}>0$. By (\ref{sum masses}) in Property 3, the set $\mathcal{J}_k$ in non empty. Let then $x_j'\in\mathcal{J}_k$. If $x_j'=x_k$, there is nothing left to prove. Suppose then 
$x_j'\neq x_k$, which by definition of $x_k$ implies $x_j'>x_k$. We first notice that between $x_k$ and $x_j'$, there can only be a finite number of elements in $\mathcal{J}_k$. This is because $x_k>0$ and the points in $\mathcal{J}_k$ are pairwise $\Gamma$-disjoint, thus, the only possible accumulation point for any sequence in $\mathcal{J}_k$ is $x=0$. Consequently, the point $x_{j_0}'=\min\{x\in\mathcal{J}_k\}$ is well define. Again, if $x_{j_0}'=x_k$, there is nothing left to prove. Suppose then 
$x_{j_0}'>x_k$, and let $0<\varepsilon<(x_{j_0}'-x_k)/2$. On the one hand,
\begin{align}
\label{no mass}
\int_{[x_k,x_{j_0}'-\varepsilon]}U(x)dx=0.
\end{align}
On the other hand, let us show the integral in (\ref{no mass})  is strictly positive, which will be  a contradiction.
Since $A_k\subset\supp (u_0)$, in particular
\begin{align*}
\delta=\int_{[x_k,x_{j_0}'-\varepsilon)}u_0(x)dx>0,\quad\text{and}\quad\int_{A_k\cap[x_{j_0}'-\varepsilon,\infty)}u_0(x)dx>0,
\end{align*}
and then by Proposition \ref{prop.64},
\begin{align*}
\int_{[0,x_{j_0}'-\varepsilon)}u(t,x)dx>\int_{[0,x_{j_0}'-\varepsilon)}u_0(x)dx\qquad\forall t>0.
\end{align*}
We now deduce from Lemma \ref{constant mass blocks} that
\begin{align*}
\int_{\{x<x_k\}}u(t,x)dx=\text{constant}=\int_{\{x<x_k\}}u_0(x)dx\qquad\forall t>0,
\end{align*}
and then we obtain
\begin{align*}
\int_{[x_k,x_{j_0}'-\varepsilon)}u(t,x)dx>\int_{[x_k,x_{j_0}'-\varepsilon)}u_0(x)dx=\delta\qquad\forall t>0.
\end{align*}
It then follows from (\ref{u99Z}) that
\begin{align*}
\int_{[x_k,x_{j_0}'-\varepsilon]}U(x)dx\geq \limsup_{n\to\infty}\int_{[x_k,x_{j_0}'-\varepsilon]}u_n(t,x)dx>\delta>0,
\end{align*}
in contradiction with (\ref{no mass}).
\end{proof}

\begin{proof}[\bfseries\upshape{Proof of Theorem \ref{S7E001}}]
By Proposition \ref{long time limit},  Lemma \ref{lemma. contained support}, and Proposition \ref{S5ECVT},
there exists a sequence, $\{t_n\}_{n\in\NN}$ such that, if  $u_n(t)=u(t+t_n)$ for all $t>0$ and $n\in \NN$, then $u_n$ converges in $C([0,\infty),\mathscr{M}_+([0,\infty)))$ to the measure $\mu $ defined in (\ref{stat}). 

Let us assume that for some other  sequence $\{s_m\}_{m\in\NN}$, the sequence $\omega _m(t)=u(t+s_m)$ is such that $\omega _m$ converges in  $C([0,\infty),\mathscr{M}_+([0,\infty)))$ to a measure $W\in C([0,\infty),\mathscr{M}_+([0,\infty)))$.

Arguing as before, there exists a subsequence of $\{\omega _m\}_{m\in\NN}$, still denoted $\{\omega _m\}_{m\in\NN}$, such that, $\{\omega _m(t)\}_{m\in\NN}$ converges narrowly to a  measure $W\in\mathscr{M}_+([0,\infty))$ for every $t\ge 0$ as $m\to\infty$. Moreover, the limit $W$ is of the form
$$
W=\sum_{j=0}^{\infty}c_j\delta_{y_j},
$$ 
and satisfies the properties  1-4 in Theorem \ref{S7E001}.  We claim that $W=U$.

By Point (i) of Proposition \ref{prop.64},   for any $x\geq 0$, the map $t\mapsto\int_{[x,\infty)}u(t,y)dy$ is monotone nonincreasing on $[0,\infty)$. Therefore the following limit exists:
$$
F(x)=\lim_{t\to\infty}\int_{[x,\infty)}u(t,y)dy,\quad x\geq 0.
$$
From, 
$$
\int_{[x,\infty)}u(t,y)dy\geq \int_{[x,\infty)}\omega _m(t,y)dy,\quad\forall m\in\NN,
$$
we first deduce  that,
$$
\int_{[x,\infty)}u(t,y)dy\geq \int_{[x,\infty)}W(y)dy.
$$
On the other hand, it follows from the narrow convergence,
$$
\int_{[x,\infty)}W(y)dy\geq\limsup_{m\to\infty}\int_{[x,\infty)}w_m(t,y)dy,
$$
and then
$$
F(x)=\int_{[x,\infty)}W(y)dy.
$$
The same argument yields
$$
F(x)=\int_{[x,\infty)}\mu (y)dy,
$$
and then, using that $M_0(W)=M_0(u_0)=M_0(\mu)$, it follows that $W$ and $\mu$ have the same (cumulative) distribution function, and therefore $W=\mu$ (cf. \cite{Klenke}, Example 1.44, p.20).
\end{proof}

We describe in the following example the behavior of a particularly simple solution of the reduced equation (\ref{S2EG}) for which, although the sequence $\{A_k(u_0)\} _{ k\in \NN }$ has only one element,  the asymptotic limit $\mu $ has two Dirac measures.

\begin{example}
\label{S5Eex1}
Let $0<a<b<c$ be such that $B(a,b)>0$, $B(b,c)>0$, and $B(a,c)=0$, and let $x_0>0$, $y_0>0$ and $z_0>0$ be such that 
$x_0+y_0+z_0=1$. If we define
\begin{align*}
u_0=x_0\delta_a+y_0\delta_b+z_0\delta_c,
\end{align*} 
it follows from the choice of the constant $a, b, c$ that $A_0(u_0)=\{a, b, c\}$ and $A_k(u_0)=\emptyset$ for all $k>0$.
On the other hand, by Proposition \ref{comparison}, (iii) the weak solution $u$ of (\ref{S2EG}) given by Theorem \ref{S8Th1}  is of the form,
\begin{align*}
u(t)=x(t)\delta_a+y(t)\delta_b+z(t)\delta_c,\,\,\,\forall t>0,
\end{align*} 
 where, in addition, $x(t)+y(t)+z(t)=1$ for all $t>0$.
Using the weak formulation (\ref{WSZ}) for the test functions $\mathds{1}_{[b,\infty)}$,
$\mathds{1}_{[c,\infty)}$, and the conservation of mass, we obtain the following system of equations:
\begin{align*}
&x'(t)=R(a,b)x(t)y(t),\,\,\,x(0)=x_0\\
&y'(t)=-R(a,b)x(t)y(t)+R(b,c)y(t)z(t),\,\,\,y(0)=y_0\\
&z'(t)=-R(b,c)y(t)z(t),\,\,\,z(0)=z_0.
\end{align*}
Since $x'(t)\ge 0$ for all $t$ and $x(t)\in (x_0, 1)$,
\begin{equation*}
\lim _{ t\to \infty } x(t)=x_\infty \in [x_0, 1].
\end{equation*}
Moreover, for all $t>0$,
\begin{align*}
&y(t)=y_0e^{\int_0^t \big(R(b,c)z(s)-R(a,b)x(s)\big)ds}\\
&z(t)=z_0e^{-R(b,c)\int_0^t y(s)ds},
\end{align*}
and, by the conservation of mass,
\begin{align}
\frac{y(t)}{z(t)}&=\frac{y_0}{z_0}e^{\int_0^t\big(R(b,c)-x(s)(R(a,b)+R(b,c))\big)ds}\leq\frac{y_0}{z_0}e^{C t}, \label{S5Ex1}\\
C&=\big(R(b,c)-x_0(R(a,b)+R(b,c))\big).\nonumber 
\end{align}
If we suppose that
\begin{align*}
x_0>\frac{R(b,c)}{R(a,b)+R(b,c)},
\end{align*}
then $C<0$ and, by (\ref{S5Ex1}),  
\begin{align*}
\lim_{t\to\infty}\frac{y(t)}{z(t)}=0.\noindent
\end{align*}
Using that for all $t>0$, $z(t)\leq z_0$, we also have, using again   (\ref{S5Ex1}), 
$y(t)\leq y_0 e^{Ct}$, and then 
$\lim_{t\to\infty}y(t)=0$. However,  since for all $t>0$, 
$$
z'(t)=-R(b,c)z(t)y(t)\geq -R(b,c)z(t)y_0e^{Ct},
$$
we have,
$$
z(t)\geq z_0\exp\bigg(\frac{R(b,c)y_0(1-e^{Ct})}{C}\bigg).
$$
Then, since $z'(t)<0$,
$$z_\infty=\lim_{t\to\infty}z(t)\ge z_0\exp\bigg(\frac{R(b,c)y_0}{C}\bigg)>0,$$ 
and the measure $\mu $ is,
\begin{equation*}
\mu =x_\infty \delta _a+z_\infty \delta _c.
\end{equation*}
\end{example}

\appendix

\section{Some useful estimates}
\label{A1}
\setcounter{equation}{0}
\setcounter{theorem}{0}

\begin{lemma}
\label{klbounds}
If $B$ satisfies (\ref{B1})--(\ref{Bsupport2}), and $\varphi$ is $L$-Lipschitz on $[0,\infty)$, then for all $(x,y)\in\Gamma$:
\begin{align}
\label{kbound}
&|k_{\varphi}(x,y)|\leq LC_*Ae^{\frac{|x-y|}{2}},\quad A=\max\left\{\frac{(1-\theta)^2}{\theta\delta(1+\theta)},\rho_*\right\},\\
\label{lbound}
&|\ell_{\varphi}(x,y)|\leq \frac{LC_*(1-\theta)}{\theta^2(1+\theta)}e^{\frac{x-y}{2}}.
\end{align}
Moreover, the function $\mathcal{L}_{\varphi}$ given in (\ref{ll}) is continuous on $[0,\infty)$ and for all $x\in[0,\infty)$,
\begin{align}
\label{llbound}
|\mathcal{L}_{\varphi}(x)|\leq \frac{LC_*(1-\theta)}{\theta^2(1+\theta)}\big(e^{\frac{x-\gamma_1(x)}{2}}-e^{\frac{x-\gamma_2(x)}{2}}\big).
\end{align}
In particular, $\mathcal{L}_{\varphi}(0)=0$.
\end{lemma}

\begin{proof}
We first prove (\ref{kbound}). Let $(x,y)\in\supp (B)=\Gamma$, and assume, by the symmetry of $k_{\varphi}$, that $0\leq y\leq x$. 
By the mean value theorem, $|e^{-x}-e^{-y}|\leq e^{-y}(x-y)$, and from (\ref{Bbound}) and the Lipschitz condition,
\begin{align*}
|k_{\varphi}(x,y)|&\leq LC_*e^{\frac{x-y}{2}}\frac{(x-y)^2}{(x+y)xy}.
\end{align*}
Then by (\ref{Bsupport})--(\ref{Bsupport2})
\begin{equation*}
\frac{(x-y)^2}{(x+y)xy}\leq\left\{
 \begin{array}{ll}
 \frac{(1-\theta)^2}{\theta\delta_*(1+\theta)}&\text{if }(x,y)\in\Gamma_1,\\
  \rho_*&\text{if }(x,y)\in\Gamma_2,
 \end{array}
 \right.
 \end{equation*}
and (\ref{kbound}) follows.

In order to prove (\ref{lbound}) we use (\ref{Bbound}) and the Lipschitz condition to have, for all $(x,y)\in\Gamma$,
\begin{align*}
\label{lbound'}
|\ell_{\varphi}(x,y)|&\leq LC_* e^{\frac{x-y}{2}}\frac{y|x-y|}{x(x+y)}.
\end{align*}
Using that $\Gamma\subset\{(x,y)\in[0,\infty)^2:\theta x\leq y\leq \theta^{-1}x\}$, then
\begin{equation*}
\frac{y|x-y|}{x(x+y)}\leq\frac{(1-\theta)}{\theta^2(1+\theta)},
\end{equation*}
and (\ref{lbound}) follows. We obtain (\ref{llbound}) directly from (\ref{lbound}) and Remark \ref{cone}.

We finally prove the continuity of $\mathcal{L}_{\varphi}$ on $[0,\infty)$. By (\ref{Bcont})--(\ref{Bsupport2}), $\mathcal{L}_{\varphi}(x)$ is continuous for all $x>0$, so we only need to prove $\mathcal{L}_{\varphi}(x)\to 0$ as $x\to 0$. 
This follows from (\ref{llbound}) and the mean value theorem, using
$
\gamma_2(x)-\gamma_1(x)\leq (\theta^{-1}-\theta)x.
$
\end{proof}

\begin{remark}
Under the hypothesis of Lemma \ref{klbounds}, the function $k_{\varphi}$ could not be continuous at the origin $(x,y)=(0,0)$, since we do not know if
$\lim_{(x,y)\to(0,0)}k_{\varphi}(x,y)=0$. However we have the following.
\end{remark}

\begin{lemma}
\label{k continuous}

If $B$ satisfies (\ref{B1})--(\ref{Bsupport2}), then $k_{\varphi}\in C([0,\infty)^2)$ for all $\varphi\in C^1([0,\infty))$ with $\varphi'(0)=0$, and $k_{\varphi}(0,0)=0$.
\end{lemma}

\begin{proof}
By definition and (\ref{Bcont}), it is clear that $k_{\varphi}\in C([0,\infty)^2\setminus\{0\})$.
If we prove that $\lim_{(x,y)\to(0,0)}k_{\varphi}(x,y)=0$, the continuity at the origin follows. To this end we mimic the proof of (\ref{kbound}) using $\varphi(x)-\varphi(y)=\varphi'(\xi)(x-y)$ for some $\xi\in(\min\{x,y\},\max\{x,y\})$ instead of the Lipschitz condition, and we obtain
\begin{align}
\label{kboundC1}
|k_{\varphi}(x,y)|
&\leq \max\left\{\frac{(1-\theta)^2}{\theta\delta(1+\theta)},\rho\right\} |\varphi'(\xi)|e^{\frac{|x-y|}{2}}
\end{align}
for all $(x,y)\in\Gamma$ and all $\varphi\in C^1([0,\infty)).$ If $\varphi'(0)=0$, it follows from (\ref{kboundC1}) that
$\lim_{(x,y)\to(0,0)}k_{\varphi}(x,y)=0$.
\end{proof}

\begin{proposition}
\label{KLbounds}
Suppose that $B$ satisfies (\ref{B1})--(\ref{Bsupport2}), $\varphi$ is $L$-Lipschitz on $[0,\infty)$, and $u\in\mathscr{M}_+([0,\infty))$. Then
\begin{align}
\label{Kbound}
&|K_{\varphi}(u,u)|\leq LC_*A\bigg(\int_{[0,\infty)}e^{\frac{x-\gamma_1(x)}{2}}u(x)dx\bigg)\bigg(\int_{[0,\infty)}u(y)dy\bigg),\\
\label{Lbound}
&|L_{\varphi}(u)|\leq \frac{LC_*(1-\theta)}{2\theta^2(1+\theta)}\int_{[0,\infty)}\big(e^{\frac{x-\gamma_1(x)}{2}}-e^{\frac{x-\gamma_2(x)}{2}}\big)u(x)dx,
\end{align}
where $A$ is given in (\ref{kbound}).
\end{proposition}

\begin{proof}
In order to prove (\ref{Kbound}), we use Remark \ref{cone}, Remark \ref{Ksym} and (\ref{kbound}):
\begin{align*}
|K_{\varphi}(u,u)|
&\leq LC_*A\int_0^{\infty}e^{\frac{x}{2}}u(x)\int_{\gamma_1(x)}^x e^{-\frac{y}{2}}u(y)dydx\\
&\leq LC_*A\int_0^{\infty} e^{\frac{x-\gamma_1(x)}{2}}u(x)\int_{\gamma_1(x)}^x u(y)dydx,
\end{align*}
from where (\ref{Kbound}) follows.
The estimate (\ref{Lbound}) follows directly from (\ref{llbound}).
\end{proof}
Let us define now
\begin{align}
&K_{\varphi,n}(u,u)=\frac{1}{2}\iint_{[0,\infty)^2}k_{\varphi,n}(x,y)u(t,x)u(t,y)dydx,\\
&k_{\varphi,n}(x,y)=b_n(x,y)(e^{-x}-e^{-y})(\varphi(x)-\varphi(y)), \label{kn}\\
&L_{\varphi,n}(u_n)=\frac{1}{2}\int_{[0,\infty)}\mathcal{L}_{\varphi,n}(x)u(t,x)dx,\\
&\mathcal{L}_{\varphi,n}(x)=\int_0^{\infty}\ell_{\varphi,n}(x,y)dy\\
&\ell_{\varphi,n}(x,y)=b_n(x,y)y^2e^{-y}(\varphi(x)-\varphi(y)).
\end{align}

\begin{remark}
\label{nbounds}
Since $\phi_n\leq x^{-1}$, the estimates (\ref{kbound}), (\ref{lbound}) and (\ref{llbound}) in Lemma \ref{klbounds} hold for $k_{\varphi,n}$, $\ell_{\varphi,n}$ and $\mathcal{L}_{\varphi,n}$ respectively, and estimates
(\ref{Kbound}) and (\ref{Lbound}) in Lemma \ref{KLbounds} hold for $K_{\varphi,n}(u,u)$ and $L_{\varphi,n}(u)$ respectively, for all $n\in\NN$.
\end{remark}

\begin{lemma}
\label{llconver}
$\mathcal{L}_{\varphi,n}\to\mathcal{L}_{\varphi}$ as $n\to\infty$ uniformly on the compact sets of $[0,\infty)$ for all $\varphi$ $L$-Lipschitz on $[0,\infty)$.
\end{lemma}

\begin{proof}
Let $R>0$ and $x\in[0,R]$.  On the one hand, if $x\in[0,1/n]$, we have
$|\mathcal{L}_{\varphi}(x)-\mathcal{L}_{\varphi,n}(x)|\leq 2|\mathcal{L}_{\varphi}(x)|\to 0$ as $n\to \infty$, since $\mathcal{L}_{\varphi}(0)=0$ (cf. Lemma \ref{klbounds}). On the other hand, if $x\in[1/n,R]$ and $y\in[1/n,n]$, by definition $\phi_n(x)\phi_n(y)=(xy)^{-1}$, and then
\begin{align*}
|\mathcal{L}_{\varphi}(x)-\mathcal{L}_{\varphi,n}(x)|&\leq\int_0^{\frac{1}{n}}|\ell_{\varphi}(x,y)-\ell_{\varphi,n}(x,y)|dy\nonumber\\
&+\int_n^{\infty}|\ell_{\varphi}(x,y)-\ell_{\varphi,n}(x,y)|dy.
\end{align*}
The two integrals in the right hand side above are treated in the same way. Using
$|\ell_{\varphi}(x,y)-\ell_{\varphi,n}(x,y)|\leq|\ell_{\varphi}(x,y)|$ and (\ref{lbound}),
\begin{align*}
&\int_0^{\frac{1}{n}}|\ell_{\varphi}(x,y)|dy
\leq \frac{LC_*(1-\theta)}{\theta^2(1+\theta)}e^{\frac{R}{2}}\int_0^{\frac{1}{n}}e^{-\frac{y}{2}}dy\xrightarrow[n\to\infty]{}0,\\
&\int_n^{\infty}|\ell_{\varphi}(x,y)|dy\leq \frac{LC_*(1-\theta)}{\theta^2(1+\theta)}e^{\frac{R}{2}}\int_n^{\infty}e^{-\frac{y}{2}}dy\xrightarrow[n\to\infty]{}0,
\end{align*}
and the result follows.
\end{proof}

\section{The function $\mathcal B _{ \beta  }$, properties and scalings}
\setcounter{equation}{0}
\setcounter{theorem}{0}
\label{deduction}
In this Section, we describe several properties of the function $\mathcal B_\beta $. First, the parameter $\beta $ is used to scale  the variables, in such a way that the total mass of the solution is conserved. Then, for each $\beta >0$ fixed, the behavior of $\mathcal B_\beta (k, k')$ is studied when $k$ and $k'$ are varying on $(0, \infty)$.

\subsection{ $\beta$-scalings of $\mathcal B_\beta $.}
\label{betascaling}
It looks natural from (\ref{S1E1}) to introduce the scaled variable 
\begin{equation}
\label{S2E1XY}
\textbf x=\beta\, \textbf k,
\end{equation}
and define
\begin{equation}
\label{S2E2XY}
F(\tau,x)=f(t,k),\qquad \tau=\beta^3 t,\quad x=\beta k.
\end{equation}
The scaling (\ref{S2E2XY}) preserves the total number of particles:
\begin{align*}
\int_0^{\infty}  x^2 F(\tau, x)dx=\int_0^{\infty}k^2f(t,k)dk=\int_0^{\infty}k^2 f(0,k)dk\quad\forall\tau>0.
\end{align*}
In terms of $F$, 
\begin{gather*}
k^2\frac{\partial f}{\partial t}(t,k)=\beta^4x^2\frac{\partial F}{\partial \tau}(\tau,x),\\
\tilde{q} (f, f')=\beta^6FF'\big(e^{-x}-e^{-x'}\big)+\beta^3\big(F'e^{-x}-Fe^{-x'}\big),
\end{gather*}
and if we define
\begin{align}
\label{BBB}
B_{\beta}(x,x')=\beta^{-1}\mathcal{B}_{\beta}(k,k'),
\end{align}
that is,
\begin{align}
\label{BbetaA}
\!\!\!\!\!B_\beta (x, x')=\sqrt \beta e^{\frac {(x'+x )} {2}}\!\!\!\int _0^\pi 
\frac {(1+\cos^2\theta)} { |\textbf{x}'-\textbf{x}| }e^{-\beta \frac {m(x-x')^2+\frac { |\textbf{x}'-\textbf{x}|^4} {4m\beta ^2}} {2 |\textbf{x}'-\textbf{x}|^2}} d\cos \theta,
\end{align}
the equation (\ref{S1E1})  then reads
\begin{align}
\label{eqFscaling}
x^2\frac {\partial F} {\partial \tau }(\tau,x)=&\int _0^\infty B_\beta (x, x')FF'\big(e^{-x}-e^{-x'}\big)xx'dx'+\nonumber \\
&+\beta ^{-3}\int _0^\infty B_\beta (x, x')\big(F'e^{-x}-Fe^{-x'}\big)xx'dx'.
\end{align}
If we now define
\begin{equation}
\label{S2EXYu}
u(\tau, x)=x^2F(\tau, x)
\end{equation}
then from (\ref{eqFscaling}) we finally obtain
\begin{align}
\frac {\partial u} {\partial \tau }(\tau,x)&=\int _0^\infty \frac {B_{\beta}(x,x')} {xx'}\big(e^{-x}-e^{-x'}\big)uu'dx'+\nonumber\\
&+\beta ^{-3}\int _0^\infty  \frac {B_{\beta}(x,x')} {xx'}\big(u'x^2e^{-x}-ux'^2e^{-x'}\big)dx', \label{EscEu}
\end{align}
The second term in the right hand side of  (\ref{EscEu}) seems then negligible  when $\beta$ tends to  $\infty$, but no rigorous result on that direction is known.

\newpage
\subsection{The function $B_\beta (x, x')$ for $\beta $ fixed.}
In this Section we show some properties of the kernel $B_{\beta}$ defined in (\ref{BbetaA}).

\begin{figure}[h]
\centering
\includegraphics[width=\textwidth]{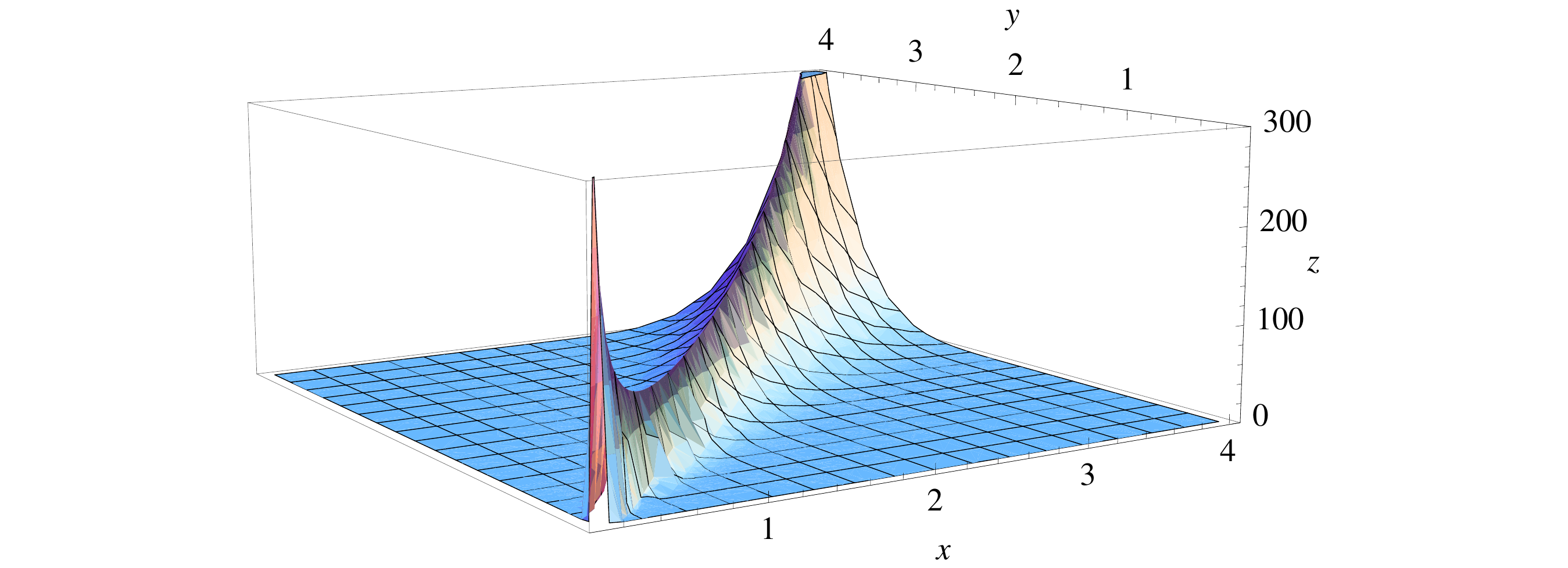}
\caption{The kernel $B_{\beta}(x,y)$ for $\beta=100$, $m=1$, $(x,y)\in[0.1,4]^2$.}
\label{fig1}
\end{figure}

\begin{figure}[h]
\captionsetup{width=0.8\textwidth}
 \centering
  \subfloat{
    \includegraphics[width=3.2cm]{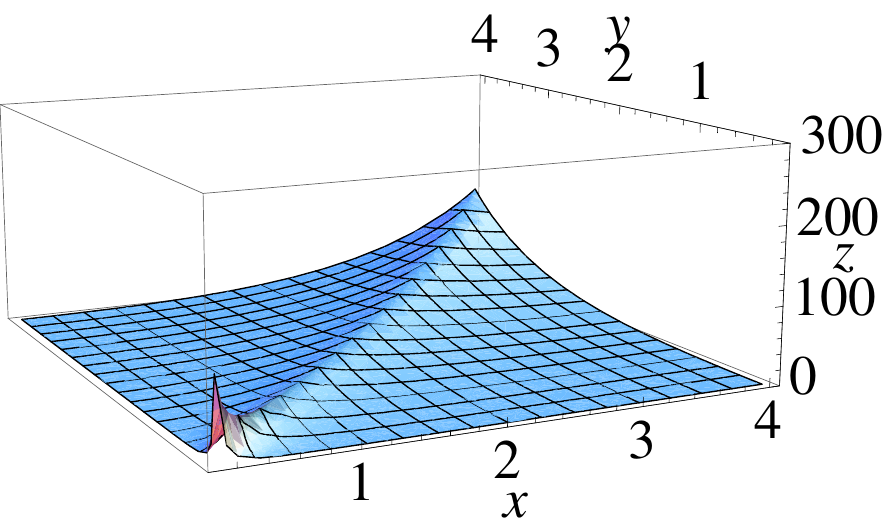}}
  \subfloat{
    \includegraphics[width=3.2cm]{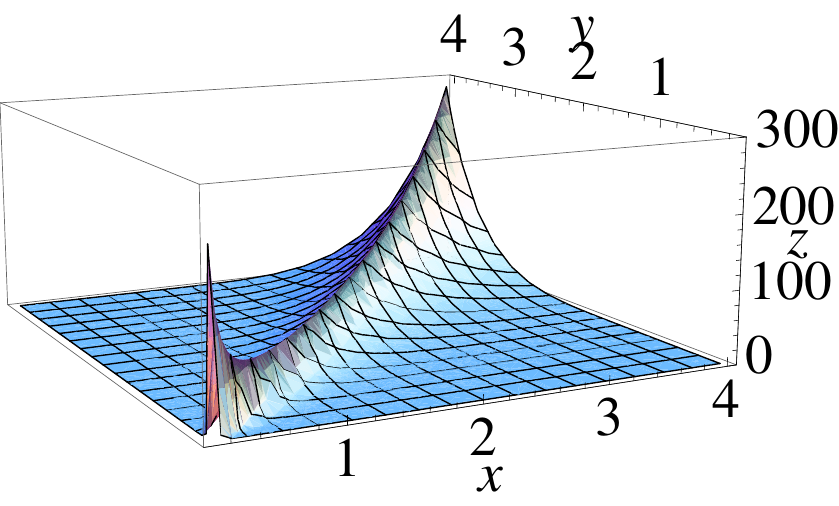}}
  \subfloat{
    \includegraphics[width=3.2cm]{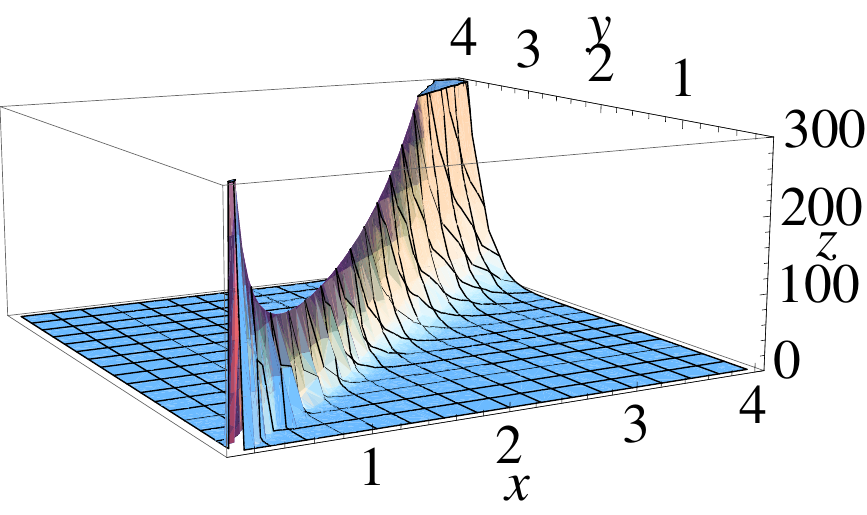}}
 \caption{From left to right, the kernel $B_{\beta}(x,y)$ for $m=1$, $(x,y)\in[0.1,4]^2$ and $\beta=10$, $\beta=50$ and $\beta=200$.}
 \label{fig234}
\end{figure}

\begin{figure}[h]
\captionsetup{width=0.8\textwidth}
\centering
\includegraphics[width=11cm]{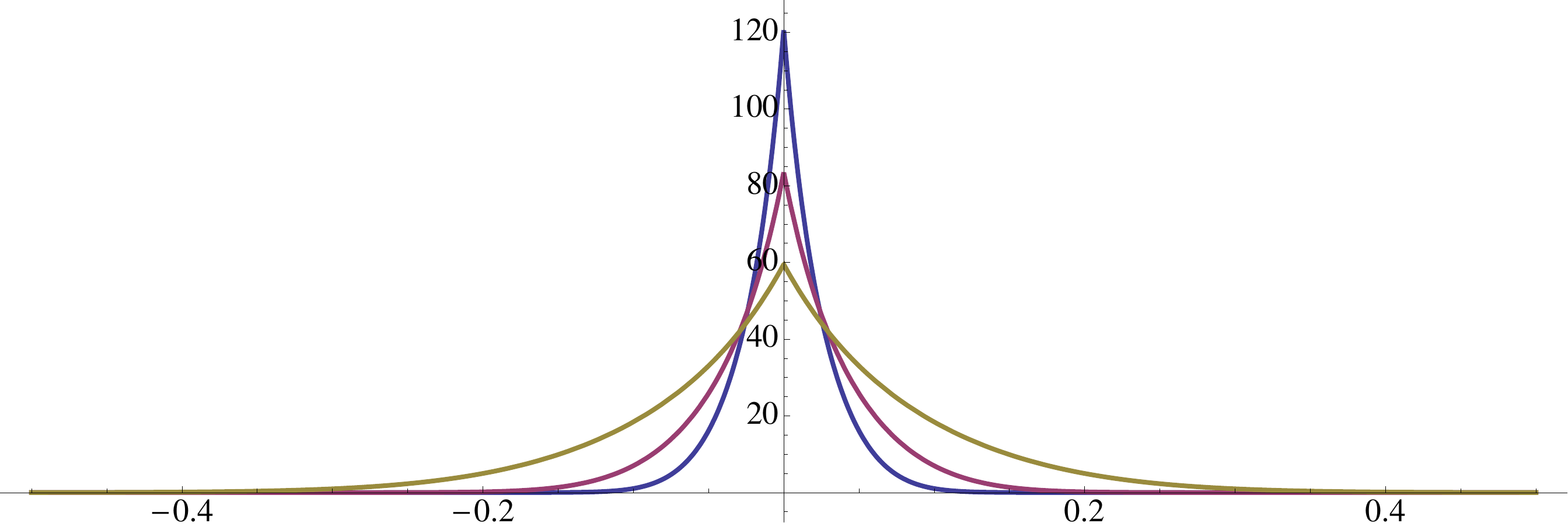}
\caption{Sections of $B_{\beta}$ for $\beta=50$ and $m=1$. The horizontal axis corresponds to the variable $\xi=(x-y)/\sqrt{2}$. The vertical axis corresponds to $B_{\beta}(x,y)$ for $x+y=\text{constant}$. 
In blue, $x+y=0.3$, in red $x+y=0.5$ and in yellow, $x+y=1.$}
\label{fig5}
\end{figure}

\begin{proposition}
For all $\beta>0$, $x>0$ and $x'>0$,
\begin{gather}
B_\beta (x, x')\le \sqrt \beta\;\frac {4\big(10\max^2\{x, x'\}+\min^2\{x, x'\}\big)} {15 \max^3\{x, x'\}}e^{\frac {(x'+x )} {2}},\label{ScalE1}\\
\intertext{and for all $x>0,\;x'>0$ with $x\neq  x'$,}
\lim _{ \beta \to \infty  }B_\beta (x, x')=0.\label{ScalE2}
\end{gather}
\end{proposition}

\begin{proof} 
For all $x>0$ and $x'>0$,
\begin{align}
\label{Bproper1}
\frac{e^{-\frac {(x+x')}{2}}}{\sqrt{\beta}}B_\beta (x, x')&\le \int _0^\pi \frac {(1+\cos^2\theta)} { |\textbf{x}'-\textbf{x}| }d\cos \theta=\int  _{ -1 }^1 \frac {(1+t^2)} { \sqrt{x^2+x'^2-2xx't} }dt \nonumber\\
&=\frac {4\big(10\max^2\{x, x'\}+\min^2\{x, x'\}\big)} {15 \max^3\{x, x'\}},
\end{align}
and then (\ref{ScalE1}) holds.
If $x'\not =x$, we have first
$$
\lim _{ \beta \to \infty }e^{-\beta \frac {m(x-x')^2+\frac { |\textbf{x}'-\textbf{x}|^4} {4m\beta ^2}} {2 |\textbf{x}'-\textbf{x}|^2}} =0\qquad\forall\theta\in[0,\pi],
$$
and since
$$
 \frac {(1+\cos^2\theta)} { |\textbf{x}'-\textbf{x}| } e^{-\beta \frac {(x-x')^2+\frac { |\textbf{x}'-\textbf{x}|^4} {\beta ^2}} {2m |\textbf{x}'-\textbf{x}|^2}} \le  \frac {(1+\cos^2\theta)} { |\textbf{x}'-\textbf{x}|}
 \in L^1(d\cos\theta) \quad\forall\beta>0,
$$
then (\ref{ScalE2})  follows from Lebesgue's convergence Theorem.
\end{proof}

\begin{proposition}
\begin{align}
&B_\beta (x, x)=\sqrt{\beta}\left(\frac{2\sqrt{2\pi m\beta}}{x^2}+\mathcal{O}\left(\frac{1}{x}\right)^3\right)e^x\quad\text{as}\quad x\to \infty,\label{ScalE5}\\
&B_\beta (x, x)=\sqrt{\beta}\;\frac {44} {15 }\left(\frac {1} {x}+1\right)+\mathcal O(x)\quad\text{as}\quad x\to 0.\label{ScalE6}
\end{align}
\end{proposition}

\begin{proof}
By definition, for all $x>0$,
\begin{align*}
&B_\beta (x, x)=\sqrt{\frac{\beta}{2}}\frac{e^{x}}{x}\int _0^\pi 
\frac {(1+\cos^2\theta)} {\sqrt{1-\cos\theta} }e^{- \frac {x^2(1-\cos \theta)} {4m \beta }} d\cos \theta\\
&=\sqrt{\beta}\frac{2e^x\sqrt{\beta m}}{x^6}\Big(\sqrt{2\pi}\big(6\beta^2m^2-2\beta m x^2+x^4\big)Erf\left(\frac{x}{\sqrt{2\beta m}}\right)-\\
&\hspace{3cm}-12e^{-\frac{x^2}{2\beta m}}(\beta m)^{3/2}x\Big),
\end{align*}
and the result follows.
\end{proof}

The function $B_\beta $ is exponentially decreasing in the direction orthogonal  to the first diagonal, as shown in the next two Propositions.
\begin{proposition}
\label{ScalT4}
For all $\beta>0$,
\begin{align}
&\nabla B_\beta (x, x')\cdot (1, -1) > 0 \quad\hbox{if}\quad x'>x>0,\\
&\nabla B_\beta (x, x')\cdot (1, -1) < 0 \quad\hbox{if}\quad x>x'>0.
\end{align}
\end{proposition}
\textbf{Proof.} It is only a straightforward calculation. With the help of Mathematica, using the change of variables $t=\cos\theta$,
\begin{align}
&\frac {\partial B_\beta} {\partial x}(x,x')=\frac{e^{\frac {(x'+x )} {2}}}{4m\sqrt{\beta}}\int _{-1}^{1} 
\frac {(1+t^2)} {|\textbf{x}'-\textbf{x}|^{5}}e^{-\beta \frac {m(x-x')^2+\frac { |\textbf{x}'-\textbf{x}|^4} {4m\beta ^2}} {2 |\textbf{x}'-\textbf{x}|^2}}\Theta(x, x', t) dt,
\label{EscE20}\\
&\Theta(x, x', t)=4(\beta m) ^2(t-1)x'(x-x')(x+x')-(x-tx') |\textbf{x}'-\textbf{x}|^4+\nonumber\\
&\hskip 2cm +2\beta m \left(x'^2-2tx'(x-1)+(x-2)x\right) |\textbf{x}'-\textbf{x}|^2.\label{EscE21}
\end{align}
The expression of  $\frac {\partial B_\beta} {\partial x'}$ is obtained from (\ref{EscE20}) and (\ref{EscE21}) using the permutation $x\leftrightarrow x'$. Then,
\begin{align}
&\nabla B_\beta (x, x')\cdot (1, -1)=\frac{e^{\frac {(x'+x )} {2}}}{4m\sqrt{\beta}}\int _{-1}^{1} 
\frac {(1+t^2)} {|\textbf{x}'-\textbf{x}|^{5} }e^{-\beta \frac {m(x-x')^2+\frac { |\textbf{x}'-\textbf{x}|^4} {4m\beta ^2}} {2 |\textbf{x}'-\textbf{x}|^2}}\times \nonumber \\
&\hskip 3cm \times 
\big(\Theta (x, x', t) -\Theta(x',x,t)\big)dt,
\end{align}
\begin{align*}
\Theta (x, x', t) -\Theta(x',x,t)&=(x'-x)\Big[4(\beta m) ^2(1-t)(x+x')^2\\
&+4\beta m (1+t)|\textbf{x}'-\textbf{x}|^2+(1+t)|\textbf{x}'-\textbf{x}|^4 \Big],
\end{align*}
and the result follows.
\qed

\begin{proposition}
\label{S3P15}
For all $\beta>0$, $x>0$ and $x'>0$,
\begin{align}
B _{ \beta  }(x, x')\le \mathscr B _{ \beta  }(x, x'),
\end{align}
where
\begin{align}
&\mathscr{B}_{\beta}(x, x')=\sqrt \beta e^{-\beta \frac {m(x-x')^2+\frac { (x-x')^4} {4m\beta ^2}} {2 (x+x')^2}}N (x+x', |x-x'|),\\
&N (p, q)= \frac {
8\, e^{\frac {p} {2}}\big(10(p+q)^2+(p-q)^2\big)}
 {15 (p+q)^3},\quad \forall p>0,\;\forall q>0.
\end{align}
\end{proposition}
\begin{proof}
For all $\textbf x\in \RR^3$ and $\textbf x'\in \RR^3$ such that $|\textbf x|=x$, $|\textbf x'|=x'$,
$$
|x-x'|\leq|\textbf{x}-\textbf{x}'|\leq x+x'.
$$
Therefore,
$$
B _{ \beta  }(x, x')\le \sqrt{\beta} e^{\frac {(x'+x )} {2}} e^{-\beta \frac {m(x-x')^2+\frac { (x-x')^4} {4m\beta ^2}} {2(x+x')^2}}\int _0^\pi 
\frac {(1+\cos^2\theta)} { |\textbf{x}'-\textbf{x}| } d\cos \theta,
$$
and the result follows using (\ref{Bproper1}).
\end{proof}

\begin{corollary} 
\begin{equation}
\forall x>0, x'>0:\,\, B _{ \beta  } (x, x')\le B _{ \beta  }\left(\frac {x+x'} {2}, \frac {x+x'} {2}\right),\label{ScalE30}
\end{equation}
\begin{align*}
&B_\beta \left(\frac {x+x'} {2}, \frac {x+x'} {2}\right)=\sqrt \beta\,\left(\frac {2\sqrt{2\pi m\beta}} {(x+x')^2}+\mathcal O\left(\frac {1} {x+x'}\right)^3\right)e^{\frac {x+x'} {2}}
\end{align*}
as $x+x'\to \infty$, and
\begin{align*}
&B_\beta \left(\frac {x+x'} {2}, \frac {x+x'} {2}\right)=\frac {44 \sqrt \beta\,} {15 }\left(\frac {1} {x+x'}+1\right)+\mathcal O(x+x'),\,\,\,x+x'\to 0.
\end{align*}
If $x+x'\to \infty$, and $|x-x'|\le \theta x$:
\begin{equation}
\label{ScalE31}
|e^{-x}-e^{-x'}|B_\beta (x, x')\le 2 \sqrt \beta\,\left(\frac {2\sqrt{2\pi m\beta}} {(x+x')^2}+\mathcal O\left(\frac {1} {x+x'}\right)^3\right)\left|\, \sinh \left(\frac {\theta x} {2} \right)\right|.
\end{equation}
For all $\rho >0$ fixed and $x>0$, $x'>0$ such that $x+x'=\rho $,
\begin{equation}
\label{ScalE51}
B _{ \beta  }(x, x')\le \sqrt \beta\,e^{-\beta \frac {(x-x')^2} {2m \rho ^2}}\Phi (\rho , |x-x'|).
\end{equation}
\end{corollary}

\begin{proof} 
By Proposition \ref{ScalT4}, the function $B_\beta$ is strictly  decreasing in the direction orthogonal to the first diagonal, and then property  (\ref{ScalE30}) follows. In order to prove (\ref{ScalE31}) we have  first, when $x+x'\to \infty$,
\begin{equation*}
|e^{-x}-e^{-x'}|B_\beta (x, x')\le 2 \left(\frac {2\sqrt{2\pi m\beta}} {(x+x')^2}+\mathcal O\left(\frac {1} {x+x'}\right)^3\right)\left| \,\sinh \left(\frac {x'-x} {2} \right)\right|
\end{equation*}
If moreover, $0\le x'-x \le \theta x$ then
\begin{equation*}
0\le(e^{-x}-e^{-x'})B_\beta (x, x')\le 2 \left(\frac {2\sqrt{2\pi m\beta}} {(x+x')^2}+\mathcal O\left(\frac {1} {x+x'}\right)^3\right) \sinh \left(\frac {\theta x} {2} \right)
\end{equation*} 
If $-\theta x \le x'-x\le 0$ then, 
\begin{equation*}
0\le -\sinh \left(\frac {x'-x} {2} \right)=\sinh \left(\frac {x-x'} {2} \right)\le \sinh \left(\frac {\theta x} {2} \right),
\end{equation*} 
and (\ref{ScalE31}) follows.
\end{proof}

\begin{proposition}
\label{S2P26}
For all $\varphi \in C_c((0, \infty)\times (0, \infty))$:
\begin{align}
\lim _{ \beta \to \infty }\iint  _{ (0, \infty)^2 }\varphi (x, y) \Phi_\beta  (x, y) \mathscr B _{ \beta  }(x, y)dxdy=\nonumber \\
= \frac {88} {15}\sqrt\frac {{m \pi }} {2} erf(1)\int _{ (0, \infty) }\varphi \left(\frac {z} {2}, \frac {z} {2} \right)e^{\frac {z} {2}}dz \label{SADirac}
\end{align}
\end{proposition}

\begin{proof}
Define the new variables
\begin{align*}
\xi =x-y,\,\,\zeta =x+y,\,\,\,\psi (\xi, \zeta )=\varphi \left(\frac {\xi +\zeta } {2}, \frac {\zeta -\xi } {2} \right)
\end{align*}
and denote $\Psi_{\beta}(\xi,\zeta)=\Phi_{\beta}(x,y)$. Then,
\begin{align*}
I=&\iint  _{ (0, \infty)^2 }\varphi (x, y)\Phi_\beta  (x, y) \mathscr B _{ \beta  }(x, y)dxdy=\\
=&\iint  _{ D }\hskip -0.2cm  e^{- \frac {\beta ^2\xi ^2+\xi ^4} {2m \beta \zeta^2}}\Psi _{ \beta  } (\xi , \zeta )
\mathscr B_\beta \left(\frac {\xi +\zeta } {2}, \frac {\zeta -\xi } {2} \right)\psi (\xi, \zeta )d\xi d\zeta 
\end{align*}
where $D=\{(\zeta,\xi)\in\RR^2:\zeta >0,\;-\zeta <\xi <\zeta\}$. We write now,
$$
\frac {\beta ^2\xi ^2+\xi ^4} {2m \beta \zeta^2}=\frac {\beta \xi ^2} {2m\zeta^2}\left( 1+ \frac { \xi ^2 } {\beta ^2} \right)
$$
and the change of variables:
$$
\sqrt{\frac {\beta } {2m}}\frac {\xi } {\zeta }=z_1,\,\,\zeta =z_2;\,\,\, \xi =\sqrt{\frac {2m} {\beta }}z_1z_2, \, \zeta =z_2
$$
whose Jacobian is $\sqrt{2m/\beta }\, z_2$ and,
\begin{align*}
I=&  \iint _{ \Omega  } e^{-z_1^2\left( 1+\frac {2m z_1^2z_2^2} {\beta }\right)} \mathscr B_\beta \left(Z_1, Z_2 \right)\times \\
&\times 
\Psi  \left(\beta ^{-1} \sqrt{\frac {2m} {\beta }}z_1z_2, \beta ^{-1} z_2, \right) \psi \left( \sqrt{\frac {2m} {\beta }}z_1z_2, z_2\right)\sqrt{2m/\beta }\, z_2dz_1dz_2\\
&Z_1=\frac {1} {2}\left(z_2+\sqrt{\frac {2m} {\beta }} z_1z_2  \right),\,\,Z_2=\frac {1} {2}\left(z_2-\sqrt  \frac {2m} {\beta } z_1z_2  \right)
\end{align*}
Due to the cut off function $\Phi _\beta (x, y)$, the actual  domain of integration $\Omega _\beta $ is:
\begin{align*}
\Omega _\beta =\Big\{(z_1,z_2)\in \RR\times \RR^+;\,\, \sqrt{2m}|\,z_1|\le \theta z_2^{1/2}\left(1-\frac {2m} {\beta }z_1^2 \right)^{1/2} \Big\}
\end{align*}
where $\Omega $ is the domain where $z_2>0$, $z_1\in (-1, 1)$.  As $\beta \to \infty$, 
\begin{align*}
\lim _{ \beta \to \infty  }e^{-z_1^2\left( 1+\frac {2m z_1^2z_2^2} {\beta }\right)} \psi \left( \sqrt{\frac {2m} {\beta }}z_1z_2, z_2\right)=
e^{-z_1^2} \psi \left(0, z_2\right).
\end{align*}
On the other hand, using (\ref{ScalE6}), for all $z_1, z_2$,
\begin{align}
\label{S2EL1}
\lim _{ \beta \to \infty  } \frac {\mathscr B_\beta \left(Z_1,Z_2\right)} {\beta }
 = \frac {44\, e^{\frac {z_2} {2}} } {15 z_2}
\end{align}
By definition of $\Psi $, for all $z_1\in \RR$ and $z_2>0$ fixed, if $\beta $ is sufficiently large,
\begin{equation*}
\Psi  \left(\beta ^{-1} \sqrt{\frac {2m} {\beta }}z_1z_2, \beta ^{-1} z_2, \right)=1
\end{equation*}
Then,
\begin{align*}
\lim _{\beta \to \infty }I=&\frac {44} {15}\sqrt {2m} \iint_{\Omega } e^{-z_1^2}
e^{\frac {z_2} {2}} \psi \left(0, z_2\right)dz_1dz_2\\
=&\frac {44} {15}\sqrt\frac {{m \pi }} {2}  erf (1)\int _{ (0, \infty) }\varphi \left(\frac {z_2} {2}, \frac {z_2} {2} \right)e^{\frac {z_2} {2}}dz_2
\end{align*}
\end{proof}

The function $B_\beta (x, y)\ge 0$  coincides with  $\mathscr B_\beta (x, y)$ for $x=y$ and is below that function, that tends to a Dirac measure along the first diagonal as $\beta \to \infty$.
From properties (\ref{ScalE2}) and (\ref{SADirac}), the truncation of $\mathcal B_\beta $ may then be seen as reasonable.\\

\noindent
\textbf{Acknowledgments.}
The research of the first author is supported by the Basque Government through the BERC 2014-2017 program, by the Spanish Ministry of Economy and Competitiveness MINECO: BCAM Severo Ochoa accreditation SEV-2013-0323, and by MTM2014-52347-C2-1-R of DGES. The research of the second author is supported by grants MTM2014-52347-C2-1-R of DGES and  IT641-13 of the Basque Government. 

The authors gratefully thank the referees for their careful reading of the manuscript and  their valuable comments and recommendations.

\end{document}